\documentclass[11pt]{amsart}
\usepackage{amsfonts,amscd,array,tabu}

\usepackage[all,cmtip]{xy}
\usepackage{amsmath}
\usepackage{amsthm}
\usepackage{amssymb}
\usepackage{enumitem}
\usepackage{hyperref}
\usepackage{tikz-cd}	
\newtheorem{thm}{Theorem}[section]
\newtheorem{conj}[thm]{Conjecture}
\newtheorem{prop}[thm]{Proposition}
\newtheorem{definition}[thm]{Definition}
\newtheorem{lemma}[thm]{Lemma}

\newtheorem{lem}[thm]{Lemma}
\newtheorem{cor}[thm]{Corollary}
\newtheorem{remark}{Remark}
\theoremstyle{remark}
\newcommand*{\Oo}{\mathcal{O}}
\newcommand*{\cC}{\mathcal{C}}
\newcommand*{\Aa}{\mathbb{A}}
\newcommand*{\Q}{\mathbb{Q}}
\newcommand*{\E}{\mathbb{E}}
\newcommand*{\N}{\mathbb{N}}

\newcommand*{\cH}{\mathbb{cH}}

\newcommand*{\Z}{\mathbb{Z}}

\newcommand*{\C}{\mathbb{C}}
\newcommand*{\F}{\mathbb{F}}

\newcommand*{\PP}{\mathbb{P}}

\newcommand*{\Gal}{\textrm{Gal}}

\newcommand*{\spec}{\textrm{Spec\,}}

\newcommand*{\Cl}{\textrm{Cl}}

\newcommand*{\ra}{\rightarrow}
\newcommand*{\Aut}{\textrm{Aut}}
\newcommand*{\ol}{\overline}

\newcommand*{\inv}{\mathrm{inv}}

\newcommand*{\ZZ}{\mathbb{Z}}

\def\Aut{{\rm Aut}}
\def\stab{{\rm stab}}
\def\sym{{\rm Sym}}
\def\sur{{\rm Surj}}

\def\tr{{\rm tr}}

\def\rk{{\rm rk}}
\def\Pic{{\rm Pic}}

\def\Sp{{\rm Sp}}
\def\GSP{{\rm GSp}}
\def\Div{{\rm Div}}
\def\Pic{{\rm Pic}}

\def\coker{{\rm Coker}}
\def\Hom{{\rm Hom}}
\def\ker{{\rm Ker}}
\def\im{{\rm Im}}

\def\Frob{{\rm Frob}}

\usepackage{amscd,amssymb,amsmath}
\usepackage{color}

\author{Michael Lipnowski, Will Sawin, Jacob Tsimerman}
\begin{document}
\title[Cohen Lenstra roots of unity]{Cohen-Lenstra heuristics and bilinear pairings in the presence of roots of unity}
\maketitle

\begin{abstract}Let $L/K$ be a quadratic extension of global fields.  We study Cohen-Lenstra heuristics for the $\ell$-part of the relative class group $G_{L/K} := \Cl(L/K)$ when $K$ contains $\ell^n$th roots of unity. While the moments of a conjectural distribution in this case had previously been described, no method to calculate the distribution given the moments was known. We resolve this issue by introducing new invariants associated to the class group, $\psi_{L/K}$ and $\omega_{L/K},$ and study the distribution of $(G_{L/K}, \psi_{L/K}, \omega_{L/K})$ using a linear random matrix model.  Using this linear model, we calculate the distribution (including our new invariants) in the function field case, and then make local adjustments at the primes lying over $\ell$ and $\infty$ to make a conjecture in the number field case, which agrees with some numerical experiments. 
\end{abstract}
\section{Introduction}

In this paper, we study the distribution of the $\ell$-part of the relative class groups $\mathrm{Cl}(L/K)$ of a quadratic extension of global fields $L/K$ where $K$ contains the $\ell^n$th roots of unity for an odd prime $\ell$. Malle observed in \cite{Malle} that the usual Cohen-Lenstra heuristics for the $\ell$-part of the class group of quadratic extensions of $\mathbb Q$ do not match numerical data in this setting (already when $\ell=3, n=1, K = \mathbb Q(\mu_3)$). We give a modified prediction which is compatible with both numerical data and the function field model.

In fact, we have found it fruitful to study the class group together with two extra invariants, $\omega_{L/K} \in ( \wedge^2 \mathrm{Cl}(L/K) ) [\ell^n] $ and $\psi_{L/K} :  \mathrm{Cl}(L/K)^\vee [\ell^n] \to \mathrm{Cl}(L/K) [\ell^n]$, which we define using class field theory and Galois cohomology in Definitions \ref{psi-nt-defi} and \ref{omega-nt-defi}. They reveal, in two different ways, a bilinear structure on the class group closely related to the Weil pairing on abelian varieties and the Cassels-Tate pairing on Tate-Shafarevich groups.

We define a set $\mathcal C_{\ell,n}$ of triples of a {finite abelian $\ell$-}group $G$, an element of $(\wedge^2 G)[\ell^n]$, and a homomorphism $G^\vee [\ell^n] \to G[\ell^n]$; we construct a measure $Q^t\mu$ on $\mathcal C_{\ell,n}$ which has several natural characterizations - as the unique measure with certain moments, as the limit of two different random matrix models as the matrix size goes to infinity, and by an explicit formula.

We conjecture that, for a fixed field $K$ containing the $\ell^n$th roots of unity, distribution of the $\ell$-part of the relative class group $ \mathrm{Cl}(L/K)$ together with these two invariants converges to $Q^t\mu$ as the discriminant of the extension $L/K$ goes to $\infty$, where $t$ is half the degree of $K$ over $\mathbb Q$. We prove a weaker form of this conjecture in the function field case. (The strength of this statement is exactly analogous to the strength of the form of Cohen-Lenstra over function fields proved by Ellenberg, Venkatesh, and Westerland \cite{EVW}, and the work in their paper is the key input in our proof.)

In addition, we perform some numerical experiments in the case $K = \mathbb Q(\mu_3)$. In this case the modifications needed to the Cohen-Lenstra predictions for the distribution of the class group were well-understood, so we focus on checking that our new invariants have the expected distribution, which they do. 

\subsection{A main result}

We think the invariants $\psi$ and $\omega$ we define are interesting, and, for reasons discussed in the next subsection, they are helpful in our proofs. But their definitions will have to wait until after the introduction (see Definitions \ref{psi-nt-defi}, \ref{omega-nt-defi}, and \ref{relativepsiomega}).  For now, we state a corollary of our main result in the function field setting making no mention of the $\psi$ and $\omega$ invariants.  
\medskip 

Fix an odd prime $\ell$. For each finite field $\mathbb F_q$ and positive integer $g$, let $\cH_{g, \mathbb F_q} $ denote the set of smooth, projective hyperelliptic curves over $\F_q$ of genus $g$. 
To a hyperelliptic curve $C$, we may associate the Picard group $\Pic^0(C)$ and its $\ell$-power part $\Pic^0(C)_{\ell}$.

The next theorem describes the probability that this $\ell$-power part takes a given value, in the limit where $g$ and $q$ both go to infinity, when $q$ is congruent to $1$ mod $\ell$. This formula uses the Pochhammer symbol $$(a;q)_k = \prod_{j=0}^{k-1} (1- a q^j) .$$

\begin{thm}\label{intro-simplified} For any pair of sequences $g_i$ and $q_i$ such that $g_i \to \infty, q_i \to \infty$, where the $q_i$ are all odd prime powers congruent to $1$ modulo $\ell^n$ but not modulo $\ell^{n+1}$, the limit
\[ \lim_{i \to \infty} \frac{ \left| \left \{ C \in \cH_{g, \mathbb F_q} \mid \Pic^0(C)_{\ell} \cong G \right\} \right| } { | \cH_{g, \mathbb F_q} | }   \]
exists and is equal to
$$\frac{\prod_{i=1}^{\infty}(1+\ell^{-i})^{-1}}{|\Aut(G)|}\cdot |(\wedge^2G)[\ell^n]|\cdot  A_n \left( G[\ell^n] \right)$$ where 
$$ A_n \left(   \oplus_{i=1}^{n} (\Z/\ell^i\Z)^{m_i} \right) = (\ell^{-1};\ell^{-1})_{m_n}\cdot \prod_{i=1}^{n-1} (\ell^{-1};\ell^{-2})_{\lceil m_i/2\rceil}=  \prod_{i=1}^n \prod_{ \substack{ 0 \leq j < m_i \\ 2\mid j\textrm{  or }i=n }} (1 -\ell^{-j-1}). $$

\end{thm} 

Accounting for $\psi$ and $\omega$, the formula simplifies considerably: the probability of a group $G$ together with invariants $\omega$ and $\psi$, equals the moment associated to $(G, \omega, \psi)$ (which happens to equal $\frac{1}{ \operatorname{Sym}^2 G[\ell^n]}$ ) divided by the number of automorphisms of $G$ fixing $\omega$ and $\psi$, times a constant independent of $G$, if $\psi$ is invertible, and $0$ otherwise. This shape of formula, where the measure equals the moment divided by the number of automorphisms times a constant, is completely analogous to what occurs for the classical Cohen-Lenstra heuristics without roots of unity, and thus serves to partially explain the more complicated Theorem \ref{intro-simplified}.

We make an analogous conjecture in the number field case:
 
 \begin{conj}\label{ntmain} Let $\ell$ be an odd prime and $n$ a natural number. Let $K$ be a number field which contains the $\ell^n$th roots of unity but not the $\ell^{n+1}$st roots of unity. Let $t = \frac{1}{2} [ K : \mathbb{Q} ].$  
 
Let $G$ be a finite abelian $\ell$-group and fix $\omega_G \in( \wedge^2 G)[\ell^n]$ and $\psi_G \in \Hom ( G^\vee [\ell^n], G[\ell^n] )$. Let $S_{K,X}$ denote those quadratic extensions $L/K$ for which $\left| \mathrm{Norm}_{K/\mathbb{Q}}( \mathrm{Disc}(L/K) ) \right| \leq X.$ Then
  
\[ \lim_{X \to \infty} \frac{1} {  | S_{K,X} | } \left| \left\{ L/K \in S_{K,X} \mid  ( \mathrm{Cl}(L/K)_{\ell} , \omega_{L/K}, \psi_{L/K}  )  \cong (G,\omega_G, \psi_G)  \right\} \right|  =Q^t \mu (G,\omega_G, \psi_G).\] \end{conj}

The invariants $\psi_{L/K}$ and $\omega_{L/K}$ and the measure $Q^t \mu$ needed to interpret this conjecture will be defined in Sections \ref{s-beg} and \ref{number-field-invariants}.

\subsection{Prior work}

We now explain some of the prior work on this problem, which will also clarify why we use these new invariants.

Cohen and Lenstra \cite{CL} made predictions for the distribution of the $\ell$-part of the class groups of a random quadratic number field. Cohen and Martinet \cite{CM} generalized these predictions to the $\ell$-part of the class groups of random quadratic extensions of a fixed number field, or even more generally to $\Gamma$-extensions of a fixed number field for a group $\Gamma$. Malle \cite{Malle} found numerical evidence which suggested that these generalized heuristics fail when the base number field contains the $\ell$th roots of unity. In \cite{Malle10}, he proposed a modified conjecture when the base field contains the $\ell$th roots of unity, but not the $\ell^2$th roots of unity. He left open what the distribution should be for higher powers of $\ell$.

One fruitful approach to study class groups of number fields is to first study the analogous problem over function fields.  The rich structure of function fields sometimes suggests hidden structures not directly apparent from the number field perspective. 

Friedman and Washington suggested in \cite{FW} that as $K$ varies through imaginary quadratic function fields over $\mathbb F_q$, $\mathrm{Cl}(K)_\ell $ should behave statistically as $ \mathrm{coker}\left(1 - F \right)$, where $F$ is a random matrix in $\mathrm{GSp}^{(q)}_{2g} (\mathbb Z_\ell),$ the coset of $\mathrm{Sp}_{2g} (\mathbb Z_\ell)$ inside $\mathrm{GSp}_{2g} (\mathbb Z_\ell)  $ consisting of symplectic similitudes of similitude factor $q.$  The key motivation for this is that $\mathrm{Cl}(K)_\ell$ is isomorphic to the cokernel of $1-F$ for some element $F$ in $\mathrm{GSp}^{(q)}_{2g}(\mathbb{Z}_\ell)$. More specifically, $F$ is the matrix by which Frobenius acts on the $\ell$-adic Tate module of the Jacobian of the curve underlying $K$. Friedman and Washington guessed that Frobenius should behave like a random element of that group.

In fact, Yu proved in \cite{Yu} that Frobenius does behave like a random element of that group in the limit as $q \to \infty$, so Friedman and Washington's suggestion is valid in that regime.  The key step in Yu's method is showing that the monodromy of the covering of the moduli space of hyperelliptic curves of genus $g$ defined by the Tate modules of their Jacobians is exactly $\Sp_{2g} (\mathbb Z_\ell)$ which, through Deligne's equidistribution theorem, shows that as $q \to \infty$, Frobenius is suitably equidistributed in the appropriate coset of this group.

However, Friedman and Washington did not calculate the distribution of $\coker (1-F)$ for $F$ in $\mathrm{GSp}^{(q)}_{2g}(\mathbb{Z}_\ell)$. Instead, they calculated the distribution of the cokernel of a random matrix in $M_{2g}(\mathbb Z_\ell)$. They conjectured that these two distributions agree in the limit as $g$ goes to $\infty$.  Achter showed that this is false when $q \not \equiv 1 \mod \ell$ \cite{Achter}. This raised the question of what the true limit $\coker (1-F)$ is as $g$ goes to $\infty$ (or if this limit even exists), and how it depends on $q$.

For the analogy to number fields, one should observe that the roots of unity of $K$ are exactly $\mathbb F_q^\times$, and therefore $K$ contains the $\ell^n$th roots of unity if and only if $\ell^n$ divides $|\mathbb F_q^\times| = q-1$. So the analogue of a number field that contains the $\ell^n$th roots of unity but not the $\ell^{n+1}$th is the case when $q \equiv 1 \mod \ell^n$ but $q \not \equiv 1 \mod \ell^{n+1}.$ If the distribution of $\mathrm{coker}\left(1 - F \right)$ converges, as $g$ goes to $\infty$, to the same value for all $q$ in this congruence class, then we could conjecture that class groups of such number fields have the same distribution.

If Friedman and Washington's conjecture were correct, this limit would match the limit as $g$ goes to $\infty$ of the distribution of cokernels of random matrices, which is the classical Cohen-Lenstra distribution. Thus, the fact that this conjecture fails is compatible with the numerical evidence that the Cohen-Lenstra conjecture fails in the presence of $\ell$-power roots of unity. 


Garton was able to make progress towards calculating the limit. In the case $\ell || q-1,$ Garton gave in \cite{Garton} a formula for the pointwise $g \to \infty$ limit of these random matrix measures. In fact, the formula he derives is identical to Malle's conjectured limiting distribution for the class group of quadratic extensions of number fields containing $\ell$th roots of unity, justifying Malle's function-field motivation for his conjecture \cite{Malle10}. He also found the $g\to\infty$ limit in the case when $\ell^2 \mid \mid q-1$. When $\ell^n \mid\mid q-1$ for $n>2$, he was not able to show that the limit existed, but was able to show that any limit of a subsequence has the expected moments. 

We prove the convergence when $\ell^n \mid \mid q-1$ for any $n$, and we give an explicit formula for the distribution.  Our approach is indirect: we calculate the distribution of the group $\mathrm{coker}\left(1 - F \right)$ together with the two extra invariants $\psi,\omega$ then sum over all possibilities for $\psi$ and $\omega.$ The definitions of the $\psi$ and $\omega$ invariants in this setting are purely group-theoretic, and may seem more motivated than the arithmetic definition. Because the $\psi$ and $\omega$ invariants are helpful for this proof, and dramatically simplify the formula for the measure, we believe they are of greater importance, and so we study them together with the class group throughout this work.

Ellenberg, Venkatesh, and Westerland proved in \cite{EVW} a form of Cohen-Lenstra for the $\ell$-parts of class groups of quadratic function fields over $\mathbb F_q$ where $q \not \equiv 1 \mod \ell$. The desired statement is that the distribution of the class group converges to the expected distribution for fixed $q$, as the degree of the discriminants grows to $\infty$. Ellenberg, Venkatesh, and Westerland obtain this convergence when $q$ goes to $\infty$ arbitrarily slowly with the degree of the discriminant. This is much more difficult than the $q \to\infty$ case and requires all the tools from the $q \to \infty$ case, in particular the monodromy computations of \cite{Yu} plus sophisticated \'{e}tale cohomology and topological arguments. However, they did not calculate the distribution where $q \equiv 1 \mod \ell$, instead only calculating the moments.

In \cite{LT}, two of us made progress on the $q \equiv 1 \mod \ell$ case by defining the $\omega$ invariant in the function field context.

In this paper, using both the $\omega$ and $\psi$ invariants and other new tricks, we prove a result exactly analogous to Ellenberg, Venkatesh, and Westerland's. This relies heavily on the upper bounds for cohomology groups of certain spaces proven in \cite{EVW}. 

Because class field theory describes the class group as the Galois group of the maximal unramified abelian extension, some prior work has tried to generalize the Cohen-Lenstra heuristics to Galois groups of non-abelian unramified extensions. Venkatesh and Ellenberg defined \cite{VE}, and Wood and Wood generalized \cite{WW}, a ``lifting invariant" associated to such unramified extensions, and conjectured values for the corresponding moments. We expect that the lifting invariant specializes to our $\omega$ invariant in the abelian case. However, neither work gave a distribution for the group together with this invariant, while we do give a precise distribution.

\subsection{Plan of the paper}

In Section \ref{s-beg}, we define the set $\mathcal C_{\ell,n}$ of groups with extra invariants $\omega, \psi$ that we consider for the rest of the paper. In fact, we define a category of such groups. The measures we study in this paper will be on the set of isomorphism classes of this category. We state Theorem \ref{measure-exists-unique}, which describes a measure $\mu$ on this set which is uniquely characterized by its moments. The proof of this theorem is contained in Section \ref{linearrandommodels}.

In Section \ref{s-ffdef}, we associate an element of $\mathcal C_{\ell,n}$ to a curve over a finite field, or more abstractly, to a symplectic similitude $F$ of a free $\mathbb Z_\ell^{2g}$-module with symplectic structure. In these cases the group $G$ is the $\ell$-part of the Picard group and the cokernel of $1-F$ respectively, and $\omega$ and $\psi$ are defined in a relatively straightforward manner. We define measures $\mu_g^q$ by averaging these elements over all hyperelliptic curves of genus $g$. We state a conjecture that these measures converge to $\mu$ as $g$ goes to $\infty$, as well as the slightly weaker convergence result Theorem \ref{ffmain}, which we prove in Section \ref{functionfieldproof}.

In Section \ref{number-field-invariants}, we begin to study the number field case. We define extra invariants $\omega$ and $\psi$ on the class group of a number field, or on the relative class group of a number field extension. We conjecture that, averaging over quadratic extensions $L/K$ of a fixed number field $K$ containing the $\ell^n$th-roots of unity but not the $\ell^{n+1}$th, the distribution of the relative class group, with these extra invariants, converges to the measure $Q^t \mu$, obtained from $\mu$ by iteratively applying a quotienting operation $Q$ (Conjecture \ref{ntmain}). The next three sections are devoted to fleshing out this conjecture.  

In Section \ref{s-alternate-psi}, we give alternate, equivalent, definitions of the invariant $\psi$; one definition uses torsors, and the other uses Hilbert symbols.

In Section \ref{s-internal-consistency}, we check that the triples $(G, \omega, \psi)$ obtained from number fields always lie in the support of the measure $Q^t \mu$.

In Section \ref{s-NT-FF-compatible}, we check that the definitions of $\omega$ and $\psi$ in the number field case, when transferred to the function field case, match our definitions in Section \ref{s-ffdef} for curves.

In Section \ref{linearrandommodels}, we give the main analytic arguments of the paper. The most important result is that the measure $\mu$ matches the measure arising from the random matrix model studied by Friedman and Washington. To prove this, we first introduce a different random matrix model, which is linear in the sense that it involves cokernels of random elements of an affine subspace of $M_{2g}(\mathbb Z_\ell)$. We then show that $\mu$ matches the measure arising from this random matrix model. To complete the proof, we show that our two random matrix models have the same moments, and that the measure $\mu$ is uniquely determined by its moments, demonstrating that $\mu$ matches the measure from the Friedman-Washington matrix model instead. Without the extra invariants, that the moments determine the measure follows from work of Wood \cite[Theorem 3.1]{Wood}, but we use a different strategy to account for the invariants.  

Finally, in this section, we use our control of the measure $\mu$ to describe the measure $Q^t \mu$ as well.  


In Section \ref{functionfieldproof}, we prove the function field equidistribution result Theorem \ref{ffmain}. To do this, we define spaces whose $\mathbb F_q$-point counts are related to the moments of the distribution of $ \Pic^0(C)_\ell, \omega_C,$ and $\psi_C$. Counting the connected components of these spaces is a purely group-theoretic exercise once the monodromy result of \cite{Yu} is used. To count the points of these spaces in the $g \to\infty ,q\to\infty$ limit, it suffices to bound their Betti numbers. We show these spaces are covered by certain spaces defined in \cite{EVW}, allowing us to show their Betti numbers are at most the Betti numbers of the spaces in \cite{EVW}, which were already bounded in \cite{EVW}.

In Section \ref{data}, we give numerical evidence for Conjecture \ref{ntmain}.

\subsection{Linearization}

We want to highlight the key idea behind the construction of the linear random model, as we think it might clarify broader work in arithmetic statistics.

We imagine that the distribution of $\mathrm{coker}\left(1 - F \right)$, for $F \in \mathrm{GSp}^{(q)}_{2g} (\mathbb Z_\ell),$ is closely related to the distribution of $\mathrm{coker} (\log F)$, where $\log F$ lies in the logarithm of $ \mathrm{GSp}^{(q)}_{2g} (\mathbb Z_\ell)$, i.e. in a coset of the Lie algebra $\mathfrak{sp}_{2g} (\mathbb Z_\ell)$ inside the larger Lie algebra $\mathfrak{gsp}_{2g} (\mathbb Z_\ell)$ (with the coset taken as additive groups). If $F$ were congruent to $1$ modulo $\ell$, this heuristic could be made rigorous using the  $\ell$-adic convergence of the logarithm power series.  Because $F$ is almost certainly not congruent $1$ to modulo $\ell$, there seems to be no hope of relating the non-linear random matrix model to its linearization directly.  Nonetheless, we show by an indirect argument that these distributions are the same.

Thus, we hope that further comparison results could be proven between random matrix models involving random elements of $\ell$-adic groups, which are often closely related to function field distributions, and random matrix models involving random elements of their Lie algebras, which can be much easier to work with. 

The first example of this phenomenon was the comparison result discovered by Friedman and Washington \cite{FW} between the cokernels of $1-F$ for random $F$ in $\mathrm{GL}_n(\mathbb Z_\ell)$ and the cokernel of random $n\times n$ matrices over $\mathbb Z_\ell$.

An example where this could be applied is the work of Poonen-Rains \cite{PR} and the subsequent work of Bhargava-Kane-Lenstra-Poonen-Rains \cite{BKLPR} modelling Selmer groups of elliptic curves.  The second of these works gives a model for Selmer groups (and therefore also the ranks and Tate-Shafarevich groups) of elliptic curves as cokernels of random alternating matrices. In the function field setting, Selmer groups are known to be cokernels of random orthogonal matrices (see \cite[Theorem 4.4]{Landesmann} and \cite{FLR}). The fact that the alternating matrices are the Lie algebra of the orthogonal groups might explain the effectiveness of their heuristic from the function field perspective.

\subsection{Acknowledgments}
We would like to thank Melanie Wood for helpful conversations regarding this paper.

While working on this paper, W.S. served as a Clay Research Fellow.

\tableofcontents

\section{Bilinearly Enhanced Groups}\label{s-beg}

\subsection{Elements of $\wedge^2$ and a bilinear pairings} \label{wedge2bilinear}

We fix an odd prime $\ell$ and a positive integer $n$. For a finite, abelian group $G$ of odd order, we define $G^\vee = \Hom(G, \mathbb Q/\mathbb Z)$ and $\wedge^2G$ to be the subgroup of $G\otimes G$ spanned by
$x^y = x\otimes y-y\otimes x$ for all $x,y \in G.$ For an element $\omega\in\wedge^2G$, an integer $r\geq 0$, we define a bilinear form $\omega_r$ on $G^\vee [\ell^r]$, valued in $\frac{1}{\ell^r} \mathbb{Z} / \mathbb{Z}$, as follows:
By identifying $G^\vee[\ell^r]\cong \Hom(G,\Z/\ell^r\Z)$ we obtain a natural map
\begin{align*}
\omega_r: G^\vee[\ell^r]\otimes G^\vee[\ell^r] &\ra \Hom(G\otimes G, \Z/\ell^r\Z\otimes \Z/\ell^r\Z) \\
&\cong \Hom(G\otimes G,\Z/\ell^r\Z) \\
&\cong(G\otimes G)^\vee[\ell^r] \\
&\xrightarrow{\text{evaluate at } \omega} \frac{1}{\ell^r} \mathbb{Z} / \mathbb{Z}.
\end{align*}

\subsection{The Category of $\ell^n$-BEGs}

\begin{definition}\label{beg}

Consider a triple $(G,\omega,\psi)$: 
\begin{itemize}
\item $G$ is a finite abelian $\ell$-group
\item $\omega\in\wedge^2G[\ell^n]$
\item $\psi:G^\vee[\ell^n]\ra G[\ell^n]$ 
\end{itemize}

For $\gamma \in G^\vee[\ell^n]$ and $\delta \in G^\vee$ define
$$\langle \gamma,\delta \rangle := \delta( \psi(\gamma) ).$$

We say the triple $(G,\omega,\psi)$ is an $\ell^n$-Bilinearly Enhanced Group ($\ell^n$-BEG) if $\psi,\omega$ satisfy the following compatibility condition: for all $r \geq 0$ and all $\alpha,\beta\in G^\vee[\ell^{n+r}],$

\begin{equation}\label{psiomegacomp}
 \langle \ell^r\alpha,\beta\rangle = \langle \ell^r\beta,\alpha \rangle + 2 \cdot \omega_{G,n+r}(\alpha,\beta).
\end{equation}

\end{definition}

\begin{definition}\label{begcategory}

We denote by $\cC_{\ell,n}$ the following category:

\begin{itemize}
 \item The objects of $\cC_{\ell,n}$ consist of all $\ell^n$-BEGs.
 \item A morphism between two objects $(G,\omega_G,\psi_G)$ and $(H,\omega_H,\psi_H)$ consists of a group homomorphism $f:G\ra H$ such that
 $f_*\omega_G=\omega_H$ and $f_*\psi_G=\psi_H$.
 \end{itemize}
 
 \end{definition} 
 
 Note that $\cC_{\ell,n}$ is not an abelian category. Moreover, a morphism in $\cC_{\ell,n}$  is an epimorphism iff the map on abelian groups is surjective, whereas
 there are more monomorphisms then one might initially expect.

\subsection{Random Measures on $\cC_{\ell,n}$}

We shall be interested in studying measures on $\cC_{\ell,n}$. Given that it is a category, it is natural to study measures by considering their \emph{moments}, i.e. by integrating along them the test functions $\#\sur(*,G^{\bullet})$ for various $\ell^n$-BEGs $G^\bullet$. This is very convenient since in the number field and function field settings we are trying to model, the moments are what we have direct access to.

In Section \ref{linearrandommodels} we define a `universal' measure $\mu$  on $\cC_{\ell,n}$ with various natural properties. We justify this measure by showing that it arises as the limiting measure in two random matrix models. Perhaps most importantly, we prove that it is determined by its moments.

\begin{thm}\label{measure-exists-unique}

There is a unique probability measure $\mu$ on $\cC_{\ell,n}$ satisfying
$$\E_{\mu}\#\sur(*,(G,\omega_G,\psi_G))  = \frac{1}{| \sym^2G[\ell^n] |}.$$ 

Moreover, the support of $\mu$ consists precisely of those $(G,\omega_G,\psi_G)$ such $\psi_G$ is an isomorphism, in which case  
$$\mu(G,\omega_G,\psi_G) = \frac{c_\ell}{| \Aut(G,\omega_G,\psi_G) | \cdot | \sym^2G[\ell^n] |}$$ 
where $c_\ell=\prod_{i=0}^{\infty} (1-\ell^{-(2i+1)}).$

\end{thm}

Moreover, we prove in Lemma \ref{momimpmeas} that $\mu$ is determined by its moments in a strong sense, meaning that if another measure $\nu$ has moments which are close to the moments of 
$\mu$, then $\nu$ itself is close to $\mu$

\subsection{Generalized Random Measures on $\cC_{\ell,n}$}

It will be necessary for us to define a slight generalization of the universal measure $\mu$. To motivate this, consider the classical case of Cohen-Lenstra setting. In the case
of imaginary quadratic fields, the $\ell$-part of the class group is well modelled by the cokernel of a large square matrix. However, if one is interested
in the case of real quadratic fields then this amounts to quotienting out a random abelian group by `one additional random element' (which might be 0), as was done in \cite{CL} and \cite{CM}.
We thus define an operator $Q$ to formalize the idea of quotienting out by a random element:

\begin{definition}\label{Qtmu} Let $\nu$ be a measure on triples $\mathcal C_{\ell,n} $. Define the measure $Q\nu$ as follows:

$$Q\nu(G,\omega_G,\psi_G):= \int_{\mathcal C_{\ell,n }}   \frac{\#\{f:\Z_\ell\rightarrow H\mid (H,\omega_H,\psi_H)/\im f\sim (G,\omega_G,\psi_G)\}}{|H|} d\nu(H,\omega_H,\psi_H).$$

\end{definition}

This gives us a one parameter family of generalizations $Q^t\mu$ of $\mu$. We prove that these measures are also determined by their moments in Lemma \ref{Qmomimpmeas}, and compute them as well as their supports and moments explicitly.

\section{Definitions, conjectures, and statements of results in the function field setting }\label{s-ffdef}

Let $C/\mathbb{F}_q$ be a curve. Our goal in this section is to introduce additional invariants with which to adorn the group $\Pic^0(C)(\mathbb{F}_q)_\ell$; these additional invariants are non-trivial only when the function field $\mathbb{F}_q(C)$ contains $\ell^n$-power roots of unity or equivalently when $\ell^n | q-1.$ To this effect, we pass to the Jacobian of the curve, which we think of simply as a principally polarized abelian variety. The group $\Pic^0(C)(\mathbb{F}_q)_\ell$ can be recovered purely from the data of the Tate module of the Jacobian and the action of Frobenius thereon.  The Weil pairing funishes the Tate module with a symplectic pairing, and Frobenius acts as a symplectic similitude with respect to this pairing.  This is the setting in which we will define our two additional invariants. 
We work in this setting both to achieve the greatest possible generality, and because it will be paramount for defining our random models. Afterwards, we will specialize to the case
of Abelian varieties, and then even further to Jacobians of curves.

\subsection{The $\omega$ and $\psi$ invariants attached to a symplectic similitude} \label{omegapsisymplecticsimilitude}
Suppose $\omega: T \times T \rightarrow \mathbb{Z}_\ell$ is a perfect symplectic pairing for $T$ a free $\mathbb Z_\ell$-module of rank $2g$.  We will suggestively refer to $\omega$ as the \emph{Weil pairing}. 

Let $V = T_{\mathbb{Q}_\ell} =T \otimes \mathbb Q_\ell.$  Let $F \in \mathrm{GSp}^{(q)}(T,\omega),$ i.e.
$$\omega(Fx,Fy) = q \cdot \omega(x,y) \text{ for all } x,y \in T.$$

Let $\omega_T := \sum_{i = 1}^g  \omega(e_i,f_i)^{-1} \cdot (e_i \wedge f_i)  \in \wedge^2 T,$ where $e_i,f_i$ runs over a symplectic basis of $T,$ i.e. a basis $\{e_i, f_i: i = 1,\ldots,g \}$ for which 
\begin{itemize}
\item
every $e_i$ or $f_i$ is orthogonal to every $e_j$ or $f_j$ if $i \neq j$
\item
$\omega(e_i,f_j)$ is non-zero in $\mathbb{Z}_\ell / \ell.$
\end{itemize} It is easy to check that $\omega_T$ does not depend on the choice of symplectic basis.

By the recipe from \S \ref{wedge2bilinear}, $\omega_{T}$ defines the sequence of alternating bilinear pairings

\begin{align*}
\omega_m: \Hom(T/\ell^m, \frac{1}{\ell^m} \mathbb{Z}_\ell / \mathbb{Z}_\ell ) \times \Hom(T / \ell^m, \frac{1}{\ell^m} \mathbb{Z}_\ell / \mathbb{Z}_\ell ) &\rightarrow \frac{1}{\ell^m} \mathbb{Z}_\ell / \mathbb{Z}_\ell \\
\left( \frac{1}{\ell^m} \omega(\bullet, s), \frac{1}{\ell^m} \omega(\bullet, t) \right) &\mapsto \frac{1}{\ell^m} \omega(s,t).
\end{align*}

Define $H := \frac{T}{(1-F)T}.$  The element $\omega_T$ pushes forward to $\omega^o_H \in \wedge^2 H.$  As explained in $\S \ref{wedge2bilinear},$ these induce alternating bilinear pairings $\omega^o_{H,m}$ on $\Hom \left(H, \frac{1}{\ell^m} \mathbb{Z}_\ell / \mathbb{Z}_\ell \right).$  More concretely, 
$$\Hom \left(H, \frac{1}{\ell^m} \mathbb{Z}_\ell / \mathbb{Z}_\ell \right) = \ker \left( 1 - F^\vee | \Hom \left(T / \ell^m, \frac{1}{\ell^m} \mathbb{Z}_\ell / \mathbb{Z}_\ell \right) \right) \subset \Hom \left( T / \ell^m, \frac{1}{\ell^m} \mathbb{Z}_\ell / \mathbb{Z}_\ell \right),$$ where $F^\vee: \Hom \left(T / \ell^m, \frac{1}{\ell^m} \mathbb{Z}_\ell / \mathbb{Z}_\ell \right) \to \Hom \left(T / \ell^m, \frac{1}{\ell^m} \mathbb{Z}_\ell / \mathbb{Z}_\ell \right)$ is the transpose of $F$, and 
\begin{equation*}
\omega^o_{H,m} = \omega_m |_{H^\vee [\ell^m] \times H^\vee[\ell^m]}.
\end{equation*}

Note that, because $F$ acts on $\omega^o_H$ by the identity and by multiplication by $q$, $\omega^o_H$ is $\ell^n$-torsion.

\bigskip

Now suppose $\ell^n || q-1.$ The snake lemma for the diagram 
$$\begin{CD}
0 @>>> T @>>>  V @>>> V/T @>>> 0\\
@. @V{1-F}VV @V{1-F}VV @V{1-F}VV @. \\
0 @>>> T @>>> V @>>> V/T @>>> 0
\end{CD}$$ 

defines an isomorphism $\mathrm{snake}: \ker(1-F | \; V/T) \cong \coker(1-F| \; T) = H.$  
For all $m > 0,$ the Weil pairing identification of $\Hom( T/\ell^m, \mathbb{Z}_\ell / \ell^m)$ with $T / \ell^m$ identifies $F^\vee$ with $qF^{-1}.$  So there are isomorphisms  
\begin{align*}
\Hom(H, \mathbb{Z}_\ell / \ell^n) &\rightarrow \ker( 1 - F^\vee | \; \Hom(T, \mathbb{Z}_\ell / \ell^n) ) \\
&= \ker(1 - F^\vee | \; \Hom(T / \ell^n, \mathbb{Z}_\ell / \ell^n) ) \\
&\xrightarrow{\text{Weil pairing}} \ker( 1 - qF^{-1} |\; T / \ell^n ) \\ 
&= \ker( 1 - F^{-1} |\; T / \ell^n ) \hspace{2.0cm} \text{because } q \equiv 1 \mod \ell^n \\
&=^{ \cdot 1/\ell^n} \ker\left(1-F | \; (V/T) [\ell^n] \right) \\
&= \ker\left(1-F | \; V/T \right) [\ell^n] \\
&=^{\mathrm{snake}} H[\ell^n]. 
\end{align*}

\begin{definition}\label{def-psi-G} Define the invariant $\psi_H : \Hom(H, \frac{1}{\ell^n} \mathbb{Z}_\ell / \mathbb{Z}_\ell ) \rightarrow H[\ell^n]$ as multiplication by $\ell^n$ composed with all of the above maps; it is an isomorphism since all of the constituent maps are isomorphisms.  \end{definition}

We define a corresponding pairing
\begin{align*}
\Hom \left(H, \frac{1}{\ell^n} \mathbb{Z}_\ell / \mathbb{Z}_\ell \right) \times \Hom \left(H, \mathbb{Q}_\ell / \mathbb{Z}_\ell \right) &\rightarrow \frac{1}{\ell^n}\mathbb{Z}_\ell / \mathbb{Z}_\ell \\
(\alpha, \beta) &\mapsto \langle \alpha, \beta \rangle_H := \beta( \psi_H(\alpha) ).
\end{align*}  

\bigskip

Every element of $\Hom(H ,  \mathbb{Z}_\ell / \ell^n ) = \ker \left(1 - F^\vee | \; \Hom(T / \ell^n, \mathbb{Z}_\ell / \ell^n) \right)$ can be expressed as $\omega(\bullet, s) \mod \ell^n$ for some $s \in T$ uniquely determined mod $\ell^n T.$  Because because $F \in \mathrm{GSp}^{(q)}(T,\omega),$ lying in $\ker(1 - F^\vee)$ amounts to $(1 - qF^{-1})s \in \ell^n T$ or equivalently, $(F-1)s \in \ell^n T$ because $q \equiv 1 \mod \ell^n T.$  Unravelling the definition of $\psi,$ we find

\begin{equation} \label{unravelpsi}
\psi_H \left( \frac{1}{\ell^n} \omega(\bullet, s) \right) = \frac{1}{\ell^n} (F-1) s \mod (F-1)T.
\end{equation}

\begin{lemma}[Compatibility between $\psi$ and $\omega$] \label{compatibilitypsiomega}
Let $r \geq 0$ be any non-negative integer.  Let $\alpha, \beta \in \Hom \left(H, \frac{1}{\ell^{n+r}} \mathbb{Z}_\ell / \mathbb{Z}_\ell \right).$  Then $\omega^o_H$ and $\psi_H$ satisfy the following compatibility relation:
$$\langle \ell^r \alpha, \beta \rangle_H - \langle \ell^r \beta, \alpha \rangle_H = \frac{q-1}{\ell^n} \omega^o_{H,n+r}(\alpha, \beta).$$  
\end{lemma}

\begin{proof}
We can represent $\alpha$ and $\beta$ as
\begin{align*}
\alpha &= \frac{1}{\ell^{n+r}} \omega(\bullet, s) \mod \mathbb{Z}_\ell, \\
\beta &= \frac{1}{\ell^{n+r}} \omega(\bullet, t) \mod \mathbb{Z}_\ell,
\end{align*}
where $s,t \in T$ are unique mod $\ell^{n+r} T$ and satisfy $(F-1)s, (F-1)t \in \ell^{n+r} T.$  

By definition of the pairings $\omega^o_{H,m},$
\begin{equation} \label{omegapartcompatibility}
\omega^o_{H,n+r}(\alpha,\beta) = \frac{1}{\ell^{n+r}} \omega(s,t) \mod \mathbb{Z}_\ell.
\end{equation}

By the calculation from \eqref{unravelpsi},
\begin{align} \label{psipartcompatibilitypart1}
\langle \ell^r \alpha, \beta \rangle_H &= \beta \left( \psi_H(\alpha) \right) \nonumber \\
&= \beta \left( \frac{1}{\ell^n} (F-1)s, t \right) \nonumber \\
&= \frac{1}{\ell^{n+r}} \omega \left( \frac{1}{\ell^n} (F-1)s,t \right) \mod \mathbb{Z}_\ell
\end{align}
and likewise
\begin{align} \label{psipartcompatibilitypart2}
\langle \ell^r \beta, \alpha \rangle_H &= \alpha \left( \psi_H(\beta) \right) \nonumber \\
&= \alpha \left( \frac{1}{\ell^n} (F-1)t, s \right) \nonumber \\
&= \frac{1}{\ell^{n+r}} \omega \left( \frac{1}{\ell^n} (F-1)t,s \right) \mod \mathbb{Z}_\ell.
\end{align}

Combining \eqref{omegapartcompatibility},  \eqref{psipartcompatibilitypart1}, and  \eqref{psipartcompatibilitypart2}:
\begin{align*}
&\langle \ell^r \alpha, \beta \rangle_H - \langle \ell^r \beta, \alpha \rangle_H - \frac{q-1}{\ell^n} \omega^o_{H,n+r}(\alpha, \beta) \mod \mathbb{Z}_\ell \\
&= \frac{1}{\ell^{2n+r}} \left( \; \omega( (F-1)s,t) - \omega( (F-1)t,s ) - (q-1) \omega(s,t)  \; \right) \\
&=  \frac{1}{\ell^{2n+r}} \left( \; \omega( (F-1)s,t) - \omega( (F-1)t,s ) - (q-1) \omega(s,t) \; \right) \\
&= \frac{1}{\ell^{2n+r}} \left( \; \omega( (F-1)s,t) - \omega( (F-1)t,s ) - \omega(Fs,Ft) + \omega(s,t) \; \right) \\
&= \frac{1}{\ell^{2n+r}} \omega( (F-1)s, (1-F)t ) \\
&= 0 \mod \mathbb{Z}_\ell ,
\end{align*}
where the last line follows because $(F-1)s, (1-F)t \in \ell^{n+r} T.$  
\end{proof}

Finally, we introduce a scaling factor so as to make the compability relation in the above lemma match up with that in equation \ref{psiomegacomp}. Namely, we define
$\omega_H=\frac{q-1}{2\ell^n} \omega^o_H$.

\subsection{The $\omega$ and $\psi$ invariants for a principally polarized abelian varieties and curves over a finite field}\label{abomegapsi}

Let $A / \mathbb{F}_q$ be a principally polarized abelian variety.  Let $\zeta$ be a generator for $\mu_{\ell^n} \subset \mathbb{F}_q^\times.$  Let $T = T_{\ell}(A)$ denote the $\ell$-adic Tate module of $A$ and let $V = T_{\mathbb{Q}_\ell}.$  Let $\mu = \varprojlim \mu_{\ell^m},$ where $\mu_{\ell^m}$ are the $\ell^m$th roots of unity.  This is a free $\mathbb{Z}_\ell$-module of rank 1.  We will identify $\mu$ with $\mathbb{Z}_\ell$ by choosing a basis vector for $\mu.$     

Via the principal polarization, the Weil pairing defines a natural alternating non-degenerate symplectic pairing $\omega: T \times T \rightarrow \mathbb{Z}_\ell.$  The Frobenius endomorphism $F: A \rightarrow A$ induces $F: T_\ell(A) \rightarrow T_\ell(A).$  It acts as a symplectic similitude in $\mathrm{GSp}^{(q)}(T,\omega)$:
$$\omega(Fx,Fy) = q \cdot \omega(x,y) \text{ for all } x,y \in T.$$  

Let $H = T / (1-F)T.$  Note that $V / T$ is naturally isomorphic to $A[\ell^\infty].$  By the snake lemma for the diagram  
$$\begin{CD}
0 @>>> T @>>>  V @>>> V/T @>>> 0\\
@. @V{1-F}VV @V{1-F}VV @V{1-F}VV @. \\
0 @>>> T @>>> V @>>> V/T @>>> 0
\end{CD},$$ 
the group $H$ is isomorphic to $\ker(1-F| \; A[\ell^\infty]) = A(\mathbb{F}_q)_{\ell}.$  By the construction described in \S \ref{omegapsisymplecticsimilitude}, these data induce the triple $(H, \omega_H, \psi_H).$\footnote{The invariants $\omega_H$ and $\psi_H$ only depend on $b \mod \ell^n \mu,$ which is equivalent to a choice of generator $\zeta \in \mu_{\ell^n}.$}  The resulting triple is the \emph{bilinearly enhanced group associated to $A$}.  We denote the invariants $\omega_H,\psi_H$ by $\omega_A, \psi_A.$

\begin{remark}
It is not hard to check that the inverse $\psi_A^{-1}$ of $\psi_A$ is the Tate-Lichtenbaum pairing \cite[XI.9]{Silverman}, but we will not use this fact in our arguments. We use $\psi$ because $\psi$, unlike its inverse, descends to quotients of the group of rational points of the abelian variety. 
\end{remark}

\bigskip

For every smooth projective curve $C / \mathbb{F}_q,$ the Jacobian $\mathrm{Pic}^0(C)$ is an abelian variety with a canonical principal polarization.  By the above discussion, we can attach bilinearly enhanced group $(\mathrm{Pic}^0(C)(\mathbb{F}_q)_\ell, \omega_{\mathrm{Pic}^0(C)}, \psi_{\mathrm{Pic}^0(C)})$ to the curve $C.$  To ease notation, we denote this triple by $(G_C, \omega_C, \psi_C).$  For the principal polarization implicitly used to define these invariants, we will always use the canonical one.

\subsection{Equidistribution Conjecture}

We fix $q,\ell^n$ as before. For each positive integer $g$, we may consider the set $\cH_{g, \mathbb{F}_q}$ of smooth, projective hyperelliptic curves over $\F_q$ of genus $g$. 
To a hyperelliptic curve $C,$ we associate the $\ell^n$-BEG  $(G_C,\omega_C,\psi_C)$. We thus obtain a corresponding counting measure $\mu^q_g$ on $\cC_{\ell,n}$. We conjecture that for $q$ fixed and $g \to \infty,$ the measures $\mu^q_g$ converge to the measure $\mu$ referred to in Theorem \S \ref{measure-exists-unique} and formally defined in \S \ref{universalmeasuredefinition}; this refines the analogue in this setting of the conjecture \cite{Malle10} and also generalizes some conjectures from \cite{LT}.

\begin{conj}\label{eqconjff}

As $g\ra\infty$, the measures $\mu^q_g$ converge to $\mu$ in the weak-* topology.

\end{conj}

\subsection{Statements of Results}

While we cannot prove Conjecture \ref{eqconjff}, we may make partial progress towards it in the style of \cite{EVW}, by building on their work. Informally, we prove 
 the moments of $\mu_g^q$ get close to those of $\mu$ for large $g$. Moreover, the error gets smaller as $q$ gets bigger. More precisely, we prove
 
 \begin{thm}\label{ffmain}
 
 Fix an element $G^\bullet=(G,\omega_G,\psi_G)\in\cC_{\ell,n}$, and suppose $q$ is sufficiently large wrt $|G|$. Let $\E_G^+,\E_G^-$ be the limsup,liminf respectively
 of $\E_{\mu^q_g}\#\sur(*,G^\bullet)$ as $g\ra\infty$. Then $$\mid\E_G^{\pm}-\E_\mu\#\sur(*,G^\bullet)\mid=O_G(q^{-1/2}).$$
 
Moreover, if $g,q$ both tend to infinity then $\mu^q_g$ converges to $\mu$ in the weak-* topology.
 
 \end{thm}
 
  \subsection{Generalizations for Conjecture \ref{eqconjff}} \label{relativeclassgrouppuncturedcurves}
We motivate in this section a generalization of Conjecture \ref{eqconjff}. Specifically, we work in the more general setting where the base curve
is not necessarily $\PP^1$, and is not necessarily proper. This will be useful later on when we motivate our conjecture in the number field setting.

To that end, let $C$ be a smooth, projective curve over $\F_q$, and let $S\subset C$ be a {reduced effective} divisor over $\F_q$. We consider double covers $D \xrightarrow{\pi} C$ where $D$
is a smooth projective curve, and $\pi$ is unramified over $S$. We set $T = \pi^{-1}(S)$, which is also reduced.  

We are interested in studying the $\ell$-part of the Picard group of $D-T$. However, since this will be split by the action of the non-trivial automorphism of $\pi$, it is better to consider the relative class group \[\Pic(D-T/C-S):=\frac{\Pic(D-T)}{\Pic(C-S)}.\]

Let $\Div_S, \Div_T$ denote the divisors on $C$ (resp. $D$) supported on $S$ (resp. $T$).  

\begin{lemma}
The natural restriction map induces right exact sequence
$$\Div_T / \pi^\ast \Div_S\rightarrow\Pic(D/C) \xrightarrow{\sim} \Pic(D - T / C - S).$$
\end{lemma}

\begin{proof}
There are maps of right exact sequences 
$$\begin{CD}
\Div_S @>>> \Pic(C) @>>> \Pic(C - S) @>>> 0 \\
@V{\pi^\ast}VV @V{\pi^\ast}VV @V{\pi^\ast}VV \\
\Div_T @>>> \Pic(D) @>>> \Pic(D - T) @>>> 0 \\
\end{CD}$$
with exact rows.  The result follows from the snake lemma.
\end{proof}

Since $\omega_D,\psi_D$ naturally push forward along quotient maps, we obtain elements $(\Pic(D/C)_\ell, \omega_{D/C}, \psi_{D/C})$ of $\cC_{\ell,n}$. 
If $S$ consists of $s$ closed points and $T$ consists of $t$ closed points, then $\Div_T / \pi^\ast \Div_S$ is a free abelian group on $u = t-s$ generators.
As such, it seems reasonable to model the image of $\Div_T / \pi^\ast \Div_S$ as a random $u$-generated subgroup of $\Pic(D/C).$ 

We define $\mu^g_{C,S}$ to be the counting measures corresponding to the elements $(\Pic(D/C)_\ell, \omega_{D/C}, \psi_{D/C})$ obtained as $\pi$ varies along
genus $g$ double covers of $C$ which are unramified over $S$.

The operator $Q^u$ defined in Definition \ref{Qtmu} exactly models the operation of quotienting out by $u$-random elements, which motivates the following conjecture:

\begin{conj}\label{eqconjffaffine}

Let $C,S,\mu^g_{C,S}$ be as above. As $g\ra\infty$, the measures $\mu^g_{C,S}$ converge to $Q^u\mu$ in the weak-* topology.

\end{conj}

\section{The invariants beyond the function field setting}\label{number-field-invariants}

Let $K$ be a number field containing the $\ell^n$th roots of unity for some odd prime $\ell$ and positive integer $n$, but not the $\ell^{n+1}$st roots of unit, and \emph{fix a generator $\zeta$ of $\mu_{\ell^n}(K)$}. We will define invariants $\psi_K$ and $\omega_K$ on $\mathrm{Cl}(K) [\ell^\infty]$ that mimic those defined in the function field setting.

To motivate these definitions, we can compare them to the function field case. If we replace every occurrence of $\spec \Oo_K$ in these definitions with the projective curve $C$ whose function field $\mathbb F_q(C)$ is $K$, we will define invariants $\psi_K$ and $\phi_K$ on the Picard group of $C$. In Section \ref{s-NT-FF-compatible}, we will see that $\psi_K= \psi_C $ and $\phi_K=\phi_C$, so these definitions agree with our earlier ones.

In Section \ref{s-internal-consistency} we will check that $\psi_K$ and $\omega_K$ satisfy the compatibility condition \eqref{psiomegacomp} making $ ( \mathrm{Cl}(K) [\ell^\infty], \omega_K, \psi_K )$ a bilinearly enhanced group.
\bigskip

\subsection{Definition of $\psi_K$}

\begin{definition}\label{psi-nt-defi}
 Working in the fppf site of $\spec\Oo_K$, recall the Kummer sequence $1\rightarrow\mu_{\ell^n}\rightarrow \mathbb{G}_m\rightarrow \mathbb{G}_m\rightarrow 1$. From the Kummer sequence and the fact
that $H^1(\spec\Oo_K,\mathbb{G}_m)\cong \mathrm{Cl}(K)$ we get the exact sequence
$$1\rightarrow \Oo_K^\times\otimes\Z/{\ell^n}\Z\xrightarrow{\delta} H^1(\spec\Oo_K,\mu_{\ell^n})\rightarrow \mathrm{Cl}(K)[{\ell^n}]\rightarrow 1.$$ 

Now, for any scheme $X$ we have that $H^1(X,\Z/{\ell^n}\Z) \cong\Hom(\pi_{1,et}(X)^{\textrm{ab}},\Z/{\ell^n}\Z)$ which yields $H^1(\spec\Oo_K,\Z/{\ell^n}\Z)\cong \mathrm{Cl}(K)^\vee[{\ell^n}]$ (from class field theory). 

We thus get a map $$\mathrm{Cl}(K)^\vee[{\ell^n}] \cong H^1(\spec\Oo_K,\Z/{\ell^n}\Z) \rightarrow H^1(\spec\Oo_K,\mu_{\ell^n}) \rightarrow \mathrm{Cl}(K)[{\ell^n}]$$ where the middle map is induced by cup product with $\zeta$. We define this map to be  $\psi_K$. 
\end{definition}

Note that, for $K$ the function field of a curve $C$ over a field $k$, we define $\mathrm{Cl}(K)$ as the group $\Pic(C)$ of line bundles on $C / k$, not its subgroup $\Pic^0(C)$ of degree zero line bundles. On the other hand, our definition of $\psi_C$ is in terms of $\Pic^0(C)$.  In Section \ref{s-NT-FF-compatible}, we will see that $\psi_C$ and $\psi_K$ are equal up to the inclusion $\Pic^0(C) \to \Pic(C)$.

\subsection{Definition of $\omega_K$}

The definition of $\omega_K$ is more involved - we do not construct it directly. We construct $\omega_K$ indirectly via pairings $\omega_{m,K}$ defined in \S \ref{omega-pairings} and discuss motivation for the definition of $\omega_K$ afterwards in \S{motivation-omega-K}.

We first state the group-theoretic Lemma \ref{wedge-from-pairing} that gives a criterion for $\omega$ to be determined uniquely by a system of pairings $\omega_m$. We then construct our pairings $\omega_{m,K}$ in \S \ref{omega-pairings}, verifying the hypotheses of Lemma \ref{wedge-from-pairing}; this defines an element of $(\wedge^2\mathrm{Cl}(K)) [\ell^n]$ in the number field case. In \S \ref{s-NT-FF-compatible}, we verify that the analogous definition in the function field case is compatible with our earlier definition by the Weil pairing.

\subsubsection{Lemma relating $\wedge^2 G$ and systems of alternating bilinear pairings on $G^\vee[\ell^m]$}

For $a, b \in  G^\vee [\ell^m]$, we view $a$ and $b$ as functions from $G$ to $\ell^{-m} \mathbb Z / \mathbb Z$, which gives a map $a \otimes b: G \otimes G \to  ( \ell^{-m} \mathbb Z / \mathbb Z ) \otimes (\ell^{-m} \mathbb Z / \mathbb Z)  = \ell^{-2m} \mathbb Z / \ell^{-m } \mathbb Z$. By embedding $\wedge^2 G$ into $G$ via $x\wedge y \mapsto x\otimes y - y\otimes x$, we have a map $\wedge^2 G \to \ell^{-2m} \mathbb Z / \ell^{-m } \mathbb Z$, which we also call $a \otimes b$.

\begin{lemma}\label{wedge-from-pairing} Let $G$ be a finite abelian $\ell$-group. Suppose we are given, for each $m$, a symplectic bilinear form $\omega_m : G^\vee [\ell^m]  \times G^\vee [\ell^m] \to \mathbb Q_\ell/ \mathbb Z_\ell $. 

Suppose that for all $a \in G^\vee [\ell^m] , b \in G^\vee [\ell^{m+1} ]  $ we have \begin{equation}\label{omega-m-compatibility} \omega_m (  a, \ell  b ) =  \omega_{m+1} ( a, b ) \end{equation}

Then there exists a unique $\omega \in \wedge^2 G$ such that for all natural numbers $m$ and for all $a, b \in  G^\vee [\ell^m] $,we have  \begin{equation}\label{omega-omega-m} \ell^m (a \otimes b ) (\omega) = \omega_m (a,b).\end{equation}

\end{lemma}

\begin{proof} We can express $G \cong \bigoplus_{i=1}^r \mathbb Z/\ell^{e_i}$ as a direct sum of cyclic groups $ \mathbb Z/\ell^{e_i}$ with generators $x_i$.

Any element of $G^\vee [\ell^m]$ is a linear combination of the forms $f_{i,m}$ for $i$ from $1$ to $r$, where $f_{i,m}$ sends $x_i$ to $\ell^{ - \min(e_i,m)}$ and $x_j$ to $0$ for $j \neq i$.

Thus, we have $\omega_m ( a,b) =\ell^m (a\otimes b) (\omega)$ for all $a,b$ if and only if we have \begin{equation}\label{wwm-basis} \omega_m( f_{i,m}, f_{j,m}) = \ell^m (f_{i,m} \otimes f_{j,m} ) (\omega) \end{equation} for all $1 \leq i< j \leq r$.

For any $\omega \in \wedge^2 G$, we can write $\omega = \sum_{i,j} c_{i,j} (x_i \wedge x_j )$ for some $c_{i,j}$.  We have  \[  (f_{i,m} \otimes f_{j,m} ) (x_i \wedge x_j) = f_{i,m} (x_i)  f_{i,j}(x_j) = \ell^{- \min(m,e_i) - \min(m,e_j) }  \] 

and thus \begin{equation} \label{basis-evaluation-formula} (f_{i,m} \otimes f_{j,m} ) (\omega ) = \ell^{ - \min(m,e_i) - \min(m,e_j) } c_{i,j} .\end{equation}

If $\omega$ satisfies \eqref{wwm-basis} for all $m, i, j$, then taking $m = \min(e_i,e_j)$, which gives $\ell^{  - \min(m,e_i) - \min(m,e_j) } =\ell^{-2\min(e_i,e_j) } $, we have
\[ c_{i,j} \equiv  \ell^{ \min (e_i,e_j)} \omega_{\min(e_i,e_j)} (f_{i,\min(e_i,e_j)}, f_{j,\min(e_i,e_j)} ) \mod \ell^{\min(e_i,e_j)} .\]
Since $(x_i \wedge x_j)$ is $\ell^{\min(e_i,e_j)}$-torsion, this implies that

\[\omega = \sum_{i,j} c_{i,j} (x_i \wedge x_j )  \sum_{i,j}  \ell^{\min (e_i,e_j)} \omega_{\min(e_i,e_j)} (f_{i,\min(e_i,e_j)}, f_{j,\min(e_i,e_j)} )( x_i \wedge x_j) .\]

If we prove the converse, that this value of $\omega$ satisfies \eqref{wwm-basis} for all $m,i,j$, then we will have established  the existence and uniqueness of a solution.
%
%
%
%
%
%
%
%
%
%

To do this, applying \eqref{basis-evaluation-formula}, it suffices to check that \[ \omega_m ( f_{i,m} , f_{j,m}) = \ell^{ m - \min(m,e_i) - \min(m,e_j)} \omega_{\min(e_i,e_j)} (f_{i,\min(e_i,e_j)}, f_{j,\min(e_i,e_j)} ).\]

For $m= \min(e_i,e_j)$ this is trivial and so we prove it by descending and ascending induction on $m$. For the descending induction, we observe that as long as $m \leq \min(e_i,e_j) $, $ f_{j, m-1} =\ell { f}_{j,m}$. By \eqref{omega-m-compatibility},

\begin{align*}
\omega_{m-1}(f_{i,m-1} , f_{j,m-1} ) &= \omega_{m-1}( f_{i,m-1}, \ell f_{j,m} ) \\
&= \omega_m (  f_{i,m-1}, f_{j,m} ) \\
&= \omega ( \ell f_{i,m}, f_{j,m} ) \\ 
&= \ell \omega( f_{i,m}, f_{j,m}),
\end{align*} which is handles the induction step because the exponent $m - \min(m,e_i) - \min(m,e_j)$ increases by $1$ when $m$ decreases by $1$. 

For the ascending induction, assume without loss of generality that $e_i \leq e_j$. For $m \geq e_i$, we have $f_{i,m+1} = f_{i,m}$, and $ f_{j,m} = \ell f_{j,m+1}  $ if $m<e_j$ and  $ f_{j,m+1}$ if $m \geq e_j$. Thus by \eqref{omega-m-compatibility},
\begin{align*}
\omega_{m+1} (f_{i,m+1}, f_{j,m+1}) &= \omega_{m+1} (  f_{i,m} ,f_{j,m+1}) \\
&= \omega_m ( f_{i,m} , \ell f_{j,m+1} ) \\
&= \begin{cases}  \omega_m (f_{i,m}, f_{j,m} ) & m< e_j \\ \ell \omega_m (f_{i,m}, f_{j,m})  & m \geq e_j \end{cases}, 
\end{align*}
which handles the inductions step because the exponent $m - \min(m,e_i)- \min(m,e_j)$ is constant if $e_i \leq m < e_j$ and increases by $1$ if $e_i,e_j \leq m$. 

\end{proof}

\begin{cor} In Lemma \ref{wedge-from-pairing}, if the pairings $\omega_m$ all take $\ell^n$-torsion values for some $n$, then $\omega$ is $\ell^n$-torsion. \end{cor}

\begin{proof} If $\omega_m$ are the bilinear forms defined by $\omega$, then $\ell^n \omega_m$ are the bilinear forms defined by $\ell^n \omega$. Thus if $\ell^n \omega_m =0$, by the uniqueness of $\omega$, it follows that $\ell^n \omega=0$. \end{proof}

\subsubsection{Construction of the pairings $\omega_{m,K}$} \label{omega-pairings}
Recall that $K$ is a global field containing $\ell^n$th roots of unity with generator $\zeta.$  Let $m$ be a natural number. 

Artin-Verdier duality defines (among other things) a pairing 
\[ ()_{\textrm{AV} } :  H^1 ( \spec \Oo_K, \mu_{\ell^m}) \times H^2 ( \spec \Oo_K, \mathbb Z/\ell^m) \to \ell^{-m}\mathbb Z/ \mathbb Z \]

Let $\zeta_m \in H^1 ( \spec \Oo_K, \mu_{\ell^m})$ be given by the $\mu_{\ell^m}$-torsor consisting of the $\ell^m$th roots of our fixed generator $\zeta$ of $\mu_{\ell^n}.$

Define a bilinear form $\omega_{m,K}: H^1 ( \spec \Oo_K, \mathbb Z/\ell^m) \times H^1 ( \spec \Oo_K, \mathbb Z/\ell^m) \to \frac{1}{\ell^m} \mathbb{Z} / \mathbb{Z}$ by \[ \omega_{m,K}(a, b) = -\frac{1}{2}( \zeta_m, a \cup b)_{\textrm{AV}}.\] By the class field theory isomorphism $H^1 ( \spec \Oo_K, \mathbb Z/\ell^m) \cong \mathrm{Cl}(K)^\vee [\ell^m] $ we can equivalently view this as a bilinear form on $\mathrm{Cl}(K)^\vee [\ell^m] $. 

\begin{lemma}
Let $K$ be a global field.  \begin{enumerate}

\item $\omega_{m,K}$ is a symplectic bilinear form.

\item $\omega_{m,K} (a,b)$ has order dividing $\ell^n$.

\item $\omega_{m,K}$ and $\omega_{m+1}$ satisfy the compatibility \eqref{omega-m-compatibility}. \end{enumerate}\end{lemma}

\begin{proof} (1) is clear because the cup product in degree 1 is symplectic bilinear. (2) is clear because $\zeta_m$ is $\ell^n$-torsion. 

(3) takes more work. First note that the class field theory isomorphism sends the inclusion map
\[ \mathrm{Cl}(K)^\vee [\ell^m] \to \mathrm{Cl}(K)^\vee [\ell^{m+1} ] \]
to the multiplication by $\ell$ map
 \[ H^1 ( \spec \Oo_K, \mathbb Z/\ell^m) \to  H^1 ( \spec \Oo_K, \mathbb Z/\ell^{m+1})\]
 and the multiplication by $\ell$ map
 \[ \mathrm{Cl}(K)^\vee [\ell^{m+1}] \to \mathrm{Cl}(K)^\vee [\ell^{m} ] \]
to reduction mod $\ell^m$ map
 \[ H^1 ( \spec \Oo_K, \mathbb Z/\ell^{m+1} ) \to  H^1 ( \spec \Oo_K, \mathbb Z/\ell^{m}).\]
We will denote the multiplication-by-$\ell$ map on cohomology classes as $a \mapsto \ell a$ and the reduction mod $\ell^m$ map as $b \mapsto \overline{b}$. Using this notation, and our definition of $\omega_m$, Equation \eqref{omega-m-compatibility} can be stated as
\[ ( \zeta_m, (\ell a) \cup b)_{\textrm{AV}}  =( \zeta_{m+1}, a\cup \overline{b})_{\textrm{AV}} .\]

To verify this, let us check the formula 
 \[ (\ell a) \cup b = \ell (a \cup \overline{b}).\] The cup product map is induced on cohomology by the multiplication $\mathbb Z/\ell^m \times \mathbb Z/\ell^m \rightarrow \mathbb Z/\ell^m$. The maps $a \to \ell a$ and $b \to \overline{b}$ are induced in cohomology by the maps $\mathbb Z / \ell^m \to \mathbb Z/\ell^{m+1}$ and $\mathbb Z/\ell^{m+1} \to \mathbb Z/\ell^m$, respectively. So any composition of these is induced on cohomology by a composition of maps of groups. To check the identity on cohomology, it suffices to check on the level of groups, where we must determine that multiplying one element by $\ell$ and then by another is equivalent to first multiplying by the other element and then by $\ell$, which is obvious.
 
Using this, we obtain 
\begin{align*} 
( \zeta_m, (\ell a) \cup b)_{\textrm{AV}} &= (\zeta_m, \ell (a \cup \overline{b} ))_{\textrm{AV} } \\
&= \ell  (  \overline{\zeta}_m, a\cup \overline{b})_{\textrm{AV}} \\
&= \ell ( \zeta_{m+1}, (a\cup \overline{b}))_{\textrm{AV}},
\end{align*} 
where the first identity follows from what we have described, the second from the same argument applied to the cup product in the definition of Artin-Verdier duality, and the third from the fact that $\overline{\zeta}_m = \zeta_{m+1}$, which is clear from the definition of $\zeta_m$ and $\zeta_{m+1}$. Here $\overline{\zeta}_{m+1}$ is defined using the map $H^1 ( \spec \Oo_K, \mu_{\ell^{m+1}} ) \to  H^1 ( \spec \Oo_K, \mu_{\ell^{m}})$, analogously to the $\mathbb Z/\ell^m$ case.

\end{proof} 

\begin{definition}\label{omega-nt-defi} For $K$ a global field, let $\omega_K \in (\wedge^2 \mathrm{Cl}(K) ) [\ell^n] $  be the unique $\omega$ such that \[\ell^m ( a \otimes b )(\omega_K) =\omega_{m,K} (a,b)\] for all $m.$  Here, $\omega_{m,K}$ are the alternating pairings defined at the beginning of Section \ref{omega-pairings}. \end{definition} 

\subsection{Motivation for the definition of $\omega_K$}\label{motivation-omega-K}
We offer some motivation for the definition of $\omega_K.$  First note that, as a general matter, it is more common in mathematics to define bilinear forms first and then to define elements of $\wedge^2$ or $\operatorname{Sym}^2$ in terms of them. Thus it is reasonable to first attempt to understand $\wedge^2 \mathrm{Cl}(K) [\ell^n]$ dually, in terms of bilinear forms. Because we are working with abelian groups and not vector spaces, it is not obvious which bilinear forms correspond to elements of $\wedge^2 \mathrm{Cl}(K) [\ell^n]$, but Lemma \ref{wedge-from-pairing} gives the answer. Lemma \ref{wedge-from-pairing} tells us to look for symplectic forms on the group of maps from the class group to $\mathbb Z/\ell^m$, which we recognize by class field theory as $H^1 ( \spec \Oo_K, \mathbb Z/\ell^m)$. A typical source of symplectic pairings on $H^1$ is the cup product, and it is especially natural to use the cup product here because of the relationship between the cup product and the Weil pairing, which we discuss later in Lemma \ref{omega-comparison}. To obtain such a pairing, we need a linear form on $H^2( \spec \Oo_K, \mathbb Z/\ell^m)$, which by Artin-Verdier duality is equivalent to an element in $H^1(\spec \Oo_K, \mu_{\ell^m})$. Finding the correct definition is then a matter of finding the correct torsor. Because our definition in the function field setting works equally well for any curve, it should correspond to a torsor that can be defined for any curve over $\mathbb F_q$. These would be the torsors that are pulled back from $\spec \mathbb F_q$, or, equivalently, split over $\overline{\mathbb F}_q (X)$. The analogue of the special field extension $\overline{\mathbb F}_q (X)$ in the number field setting is probably the cyclotomic field, since $\overline{\mathbb F}_q (X)$ is generated over $\mathbb{F}_q(X)$ by the roots of unity. To make a $\mu_{\ell^m}$-torsor that splits over a cyclotomic field, the simplest choice is to take the $\ell^m$th roots of a root of unity, and since we want our torsor to be $\ell^n$-torsion, and we have an $\ell^n$th root of unity available, that is a natural choice. The scalar constant of $-\frac{1}{2}$ is to make the compatibilty relation between $\psi_K$ and $\omega_K$ match up with our definitions, as well as with the function field case. Of course, if we scaled both definitions of $\omega$ as well as our compatibility relation
\eqref{psiomegacomp} by the same element of $\Z_\ell^\times$, all our results would remain essentially unchanged.

\subsection{Bilinear invariants for relative class groups}
The simplest case of the Cohen-Lenstra heuristics describes the class groups of quadratic extensions of $\mathbb Q$. Because these almost never contain $\ell^n$th roots of unity for $\ell>2$, we focus instead on varying quadratic extensions $L$ of a fixed global field $K$, where $K$ contains the $\ell^n$th roots of unity.

However, when doing this, $\mathrm{Cl}(L)_{\ell}$ will always contain $\mathrm{Cl}(K)_{\ell}$ as a subgroup. Because this subgroup does not vary, its distribution is uninteresting, so we quotient out by it. This motivates our use of the relative class group. Similarly, we have relative versions of $\psi$ and $\omega$.

\begin{definition} \label{relativepsiomega}
For an extension $L/K$, the {\emph relative class group} $\mathrm{Cl}(L/K)$ is the quotient $\mathrm{Cl}(L)/ \mathrm{Cl}(K)$, where we view ideal classes on $K$ as ideal classes on $K$ by tensoring with $\mathcal O_L$.

We define $\psi_{L/K}$ and $\omega_{L/K}$ to be the pushforwards of $\psi_L$ and $\omega_L$ from $\mathrm{Cl}(L)$ to $\mathrm{Cl}(L/K)$.
\end{definition}

As long as the degree $[L:K]$ is prime to $\ell$, the natural map $\mathrm{Cl}(L)_\ell$ to $\mathrm{Cl}(K)_\ell $ is split by the norm map $\mathrm{Cl}(K)_\ell \to \mathrm{Cl}(L)_\ell$ and so $\mathrm{Cl}(L/K)_\ell$ is a summand of $\mathrm{Cl}(K)_\ell$.

\subsection{Conjecture for the distribution of triples $(\mathrm{Cl}(L/K)_\ell, \psi_{L/K},\omega_{L/K})$ for quadratic extensions of number fields containing $\ell^n$-roots of unity}

When we transfer from the function field setting to the number field setting, we conjecture that the distribution of triples $(\mathrm{Cl}(L/K)_\ell, \psi_{L/K},\omega_{L/K})$ is governed by a modification of the measure $\mu$ characterized by Theorem \ref{measure-exists-unique}.  A modification of $\mu$ is necessary because the places of $L$ lying over $\infty$ must be thought of as analogous to punctures in the curve appearing in the function field case. In Conjecture \ref{eqconjffaffine}, we conjectured that for punctured curves the distribution of the Picard group is controlled by the measures $Q^{u} \mu$ defined in Definition \ref{Qtmu} by quotienting out by $u$ random elements. In the number field case, our conjecture is analogous:
%
%
%
%
%
%
%
%
%

\begin{conj}[Conjecture \ref{ntmain}]\label{ntmain-2} Let $\ell$ be an odd prime and $n$ a natural number. Let $K$ be a number field which contains the $\ell^n$th roots of unity but not the $\ell^{n+1}$th roots of unity. Let $t$ be half the degree of $K$. Then as $L$ varies over quadratic extensions of $K$, the BEG $( \mathrm{Cl}(L/K)_{\ell} , \psi_{L/K}, \omega_{L/K})$ is equidistributed according to the measure $Q^t \mu$.

 More formally, let $S_{K, X}$ be the set of quadratic extensions $L/K$ of discriminant less than $X$.  We conjecture that the counting measures of $( \mathrm{Cl}(L/K)_{\ell} , \psi_{L/K}, \omega_{L/K})$ averaged over $S_{K,X}$, converge to $Q^t \mu$ in the weak-* topology as $X \to\infty$.\end{conj}

In Section \ref{s-internal-consistency}, we will check, in Proposition \ref{class-group-support}, that $(\mathrm{Cl}(L/K)_{\ell}, \psi_{L/K} , \omega_{L/K} )  $ is always contained in the support of $Q^t \mu$, which is a basic sanity check on Conjecture \ref{ntmain-2}.

The motivation for quotienting by exactly $\frac{ [K:\mathbb Q]}{2}$ random elements comes primarily from the function field case, where we conjectured that the distribution $\Pic(D/C)_\ell$, where $D$ is a double cover of a punctured curve $C$, was $Q^u \mu$, where $u$ was the number of punctures of $D$ minus the number of punctures of $C$ (equivalently, the number of punctures of $C$ that are split in $u$). For an extension $L/K$ of number fields, because $K$ contains $\mathbb Q(\mu_{\ell^n})$, it is totally complex, and so has $\frac{[K:\mathbb Q]}{2}$ infinite places. Because these are all complex places, they all split in $L/K$, and hence the analogue of $u$ is $\frac{[K:\mathbb Q]}{2}$.

Alternatively, if one thinks of the elements we quotient by as coming from the unit group, the same logic shows that the rank of the unit group of $L$ modulo the unit group of $K$ is $\frac{[K:\mathbb Q]}{2}$.

\section{Alternate definitions of $\psi_K$}\label{s-alternate-psi}

We present some equivalent definitions of the invariant $\psi_K$.  Throughout this section, we fix a generator $\zeta \in \mu_{\ell^n}(\mathcal O_K).$

\subsection{Definition by Restricting Torsors}

We can express the composition $H^1(\spec\Oo_K,\Z/{\ell^n}\Z) \xrightarrow{\cup \zeta} H^1(\spec\Oo_K,\mu_{\ell^n}) \rightarrow \mathrm{Cl}(K)[{\ell^n}]$ in a different way, using torsors. It is not possible to directly express $H^1(\spec O_K,\Z/{\ell^n}\Z) \cong\Hom(\mathrm{Cl}(K),\Z/{\ell^n}\Z)$ this way, as that map is not defined using torsors but rather using class field theory. 

\begin{prop}\label{torsor-definition} \begin{enumerate}

\item Given a $\mathbb Z/{\ell^n}$-torsor $Y$ over $\spec \Oo_K$, viewed as a scheme, the $\mathbb G_m$-torsor associated to it by the map $$H^1(\spec\Oo_K,\Z/{\ell^n}\Z) \xrightarrow{\cup \zeta} H^1(\spec\Oo_K,\mu_{\ell^n}) \rightarrow H^1(\spec \Oo_K, \mathbb G_m)$$ is the inverse of the space of invertible functions on $Y$ where the canonical $\mathbb Z/{\ell^n}$-action on $Y$ multiplies the functions by the fixed generator of $\mu_{\ell^n}$ in $\mathbb G_m$. 

\item Given a $\mathbb Z/{\ell^n}$-torsor  over $\spec \Oo_K$, viewed as an \'{e}tale algebra $R$ over $\mathcal O_K$ with an automorphism of order ${\ell^n}$, the locally free module associated to it by $H^1(\spec\Oo_K,\Z/{\ell^n}\Z)\xrightarrow{\cup \zeta} H^1(\spec\Oo_K,\mu_{\ell^n}) \rightarrow H^1(\spec \Oo_K, \mathbb G_m)= \mathrm{Cl}(K)$ is the dual to the subset of $R$ where the automorphism of order ${\ell^n}$ of acts by the fixed generator of $\mu_{\ell^n}$.

\end{enumerate}

\end{prop}

\begin{proof}\begin{enumerate}

\item In general, given groups $H$ and $G$, a map $i: H \to G$, and a left $H$-torsor $Y$ representing an element in $H^1(X,H)$, the induced image in $H^1(X,G)$ is given by the left $G$-torsor of maps $ f: X \to G$  such that for $h \in H, y\in Y$, $ f(hx) = f(x) h^{-1}$, with the action of $G$ given by left multiplication. This can be checked immediately with the cocycle definition of functoriality of $H^1$.  In the case $H = \mathbb Z/{\ell^n}$, $G = \mathbb G_m$,  $i$ sending a generator of $H$ to a generator of $\mu_n$, this is exactly the stated construction.

\item This follows from the previous part and the observation that a $\mathbb G_m$-torsor is associated to the unique invertible sheaf whose invertible sections over each open set are equal to the $\mathbb G_m$-torsor.  The space of invertible functions on $Y$ such that the canonical $\mathbb Z/{\ell^n}$-action on $Y$ multiplies the functions by the fixed generator of $\mu_{\ell^n}$ in $\mathbb G_m$ is simply the invertible elements of the module of elements of $R$ where the automorphism of order $n$ acts by the fixed generator of $\mu_{\ell^n}$, and then dualizing this module corresponds to inverting the torsor.

\end{enumerate}

\end{proof}

\subsection{ Definition by Hilbert Symbols} \label{psihilbertsymbol}

In this section we construct an alternative definition for $\psi_K$ using Hilbert symbols. This will be convenient for explicit calculations, and is what we used in our numerical experiments.

To understand the map $H^1(K, \mathbb Z/{\ell^n}) \xrightarrow{\cup \zeta} H^1(K, \mu_{\ell^n})$ more explicitly, we use the following commutative diagram:

$$\begin{CD}
 H^1(\spec \Oo_K, \mathbb Z/{\ell^n}) @>>> H^1(K, \mathbb Z/{\ell^n}) \\
@V{\cup \zeta}VV @V{\cup \zeta}VV \\
 H^1( \spec \Oo_K, \mu_{\ell^n}) @>>> H^1(K, \mu_{\ell^n}).
\end{CD}$$


Because $\mu_{\ell^n}= \mathbb Z/{\ell^n}$ over $K$, the rightmost map is an isomorphism. Since $\spec \Oo_K$ is a normal scheme of dimension 1, the top map is an 
injective. it follows that the map $H^1(\spec \Oo_K, \mathbb Z/{\ell^n})\xrightarrow{\cup \zeta} H^1(\spec \Oo_K, \mu_{\ell^n})$ is also injective.

Consider the pairing $\gamma:\Aa_K^{\times}\times \Aa_K^{\times} \rightarrow \mu_{\ell^n}$ sending $\gamma(a,b)=\sum_v \langle a,b\rangle_{{\ell^n},v}$ where $\langle ,\rangle_{{\ell^n},v}$ denotes the ${\ell^n}$-Hilbert symbol pairing at $v$. By Hilbert reciprocity, $K^{\times}$ is isotropic for $\gamma$, so $\gamma$ descends to a pairing on $K^{\times}\backslash \Aa_K^{\times} \times K^{\times}/(K^{\times})^{\ell^n},$ which we also denote by $\gamma.$

Class field theory gives an isomorphism $H^1 ( K, \mathbb Z/{\ell^n}) = \Hom( K^{\times} \backslash \mathbb A_K^\times , \mathbb Z/{\ell^n})$. Furthermore $H^1 ( K, \mu_{\ell^n}) = K^\times / (K^\times)^{\ell^n}$ by Kummer theory.

\begin{lemma}\label{class-kummer-isomorphism} The isomorphism $K^\times / (K^\times)^{\ell^n} \cong \Hom( K^{\times} \backslash \mathbb A_K^\times , \mathbb Z/{\ell^n})$ obtained by composing the Kummer theory isomorphism $K^\times / (K^\times)^{\ell^n} \cong H^1(K,\mu_{\ell^n}),$ the isomorphism $H^1(K, \mathbb Z/{\ell^n}) \cong H^1(K,\mu_{\ell^n})$ induced by a fixed choice of generator $\zeta \in H^0(K,\mu_{\ell^n}),$ and the class field theory isomorphism is induced by the pairing $\gamma$:
\begin{align*}
K^\times / (K^\times)^{\ell^n} &\rightarrow \Hom( K^{\times} \backslash \mathbb A_K^\times , \mathbb Z/{\ell^n}) \\
b &\mapsto \iota \circ \gamma(\cdot, b),
\end{align*}
where $\iota: \mu_{\ell^n} \xrightarrow{\sim} \mathbb{Z} / \ell^n \mathbb{Z}$ is the isomorphism induced by the choice of $\zeta.$
 \end{lemma}

\begin{proof} Let $b$ in $K^\times$ be an element. Its associated class in $H^1(K,\mu_{\ell^n})= H^1(K,\mathbb Z/{\ell^n})$ corresponds to the degree ${\ell^n}$ abelian extension $K(\sqrt[{\ell^n}]{b})$. We must check that this Galois character, viewed as a character of the idele class group, is given by $ a \mapsto  \sum_v \langle a,b\rangle_{{\ell^n},v}$. Because the ideles are contained in a product of local fields, it is sufficient to  check that the character of the Galois group of $K_v$ defined by $K_v(\sqrt[{\ell^n}]{b})$ is given by $a \mapsto \langle a,b \rangle_{{\ell^n},v}$. This is one definition of the Hilbert symbol. \end{proof}

\begin{lemma}\label{adelic-dual-description} Under the identification of $H^1(K,\mathbb{Z} / \ell^n \mathbb{Z})$ with $\Hom( K^{\times} \backslash \mathbb A_K^\times , \mathbb Z/{\ell^n})$ via the class field theory isomorphism, the image of $H^1(\spec \Oo_K, \mathbb Z/{\ell^n})$ inside  $H^1(K, \mathbb Z/{\ell^n})$ is the subset of $ \Hom( K^{\times} \backslash \mathbb A_K^\times , \mathbb Z/{\ell^n})$ that is trivial on $\Oo_{K_v} ^\times$ for all finite places $v$ of $K$. \end{lemma}

\begin{proof} $H^1(\spec \Oo_K, \mathbb Z_{\ell^n}) = \Hom( \left( \pi_1^{et} (\spec \Oo_K)\right)^{\mathrm{ab}} , \mathbb Z/{\ell^n})$ is the subset of $\Hom (\left( \Gal(K)\right)^{\mathrm{ab}}, \mathbb Z/{\ell^n})$ that is trivial on the kernel of the natural map $\left( \Gal(K)\right)^{\mathrm{ab}} \to \left( \pi_1^{et} (\spec \Oo_K)\right)^{\mathrm{ab}})$, which is the natural map from the Galois group of the maximal abelian extension to the maximal abelian unramified extension, which in the language of class field theory is precisely the profinite completion of the map $K^{\times} \backslash \mathbb A_K^\times \to K^{\times} \backslash \mathbb A_K^\times/\prod_v \Oo_{K_v}^\times$, hence elements trivial on the kernel of this map are precisely elements trivial on $ \Oo_{K_v}^\times$. (The profinite completion may be ignored because we are working with finite order characters of these groups.) \end{proof}

\begin{lemma}\label{kummer-torsor-description} The image of $H^1(\spec \Oo_K, \mu_{\ell^n})$ inside $H^1(K, \mu_{\ell^n})$ is the subset of $K^\times / (K^\times)^{\ell^n}$ consisting of elements whose valuation at each finite place is a multiple of ${\ell^n}$.  \end{lemma}

\begin{proof} A $\mu_{\ell^n}$-torsor over $K$ necessarily extends to a $\mu_{\ell^n}$ torsor over an open subset of $\spec \Oo_K$. To check it extends to the whole ring, by Beauville-Laszlo, it is necessary and sufficient to check that it extends to each complete local ring. To do this, we compare the Kummer sequences for $K, K_v,$ and $\spec \Oo_{K_v}$. 

\[ \begin{tikzcd} K^\times \arrow[r,"{\ell^n}"] \arrow[d] & K^\times \arrow[r] \arrow[d] & H^1(K,\mu_{\ell^n}) \arrow[r]\arrow[d] & 0 \\
 K_v^\times \arrow[r,"{\ell^n}"] & K_v^\times \arrow[r] & H^1(K_v,\mu_{\ell^n}) \arrow[r] & 0 \\
  \Oo_{K_v}^\times \arrow[r,"{\ell^n}"]\arrow[u] &  \Oo_{K_v}^\times \arrow[r]\arrow[u] & H^1 (\spec \Oo_{K_v},\mu_{\ell^n}) \arrow[r]\arrow[u] & 0 \\ \end{tikzcd} \]
 
Hence an element of $H^1(K, \mu_{\ell^n})$, when restricted to $H^1(K_v,\mu_{\ell^n})$, lies in the image of $H^1(\spec \Oo_{K_v},\mu_{\ell^n})$, if and only if the corresponding element of $K^\times / (K^\times)^{\ell^n}$, when restricted to $K_v^\times / (K_v^\times)^{\ell^n}$, lies inside the image of $ \Oo_{K_v}^\times / ( \Oo_{K_v}^\times)^{\ell^n}$. The image of $ \Oo_{K_v}^\times / ( \Oo_{K_v}^\times)^{\ell^n}$ is equal to $ \Oo_{K_v}^\times (K_v^\times)^{\ell^n}$, which consists precisely of elements whose $v$-adic valuation is a multiple of ${\ell^n}$.
\end{proof}

\begin{lemma}\label{kummer-map-description} The natural map to  $H^1(\spec \Oo_K, \mu_{\ell^n})  \to H^1(\spec \Oo_K, \mathbb G_m)[{\ell^n}] = Cl(K)[{\ell^n}]$ coming from the Kummer exact sequence is given by sending an element of $f\in K^\times/(K^\times)^n$ whose valuation is a multiple of ${\ell^n}$ at all finite places to $\prod_{\mathfrak p \in \spec \Oo_K} \mathfrak p^{ v_p(f)/{\ell^n}}$.  \end{lemma}

Note that this is well-defined as an ideal class $f$ as multiplying by the ${\ell^n}$th power of any element simply multiplies $\prod_{\mathfrak p \in \spec \Oo_K} \mathfrak p^{ v_p(f)/{\ell^n}}$ by that element's principal ideal, and raising $\prod_{\mathfrak p \in \spec \Oo_K} \mathfrak p^{ v_p(f)/{\ell^n}}$ to the ${\ell^n}$th power produces the principal ideal generated by $f$.

\begin{proof} By Proposition \ref{torsor-definition}, given a $\mu_{\ell^n}$-torsor, the associated $\mathbb G_m$-torsor can be defined as the inverse of the torsor  of $\mathbb G_m$-valued functions on the torsor that transform by multiplication by $\mu_{\ell^n}$ under the action of $\mu_{\ell^n}$.  In particular any meromorphic such function gives us a map from the torsor to $\mathbb G_m$, and hence lets us write it as a fractional ideal. 

For the torsor $\sqrt[{\ell^n}]{f}$, for $f \in K$, $\sqrt[{\ell^n}]{f}$ is such a function. Over any point $\mathfrak p$, the order to which $\mathfrak p$ appears in this fractional ideal is the highest power of $\mathfrak p$ that divides elements (locally) in the image of this function. Because all elements in the image are multiples of $\sqrt[{\ell^n}]{f}$ by local units, the highest power of $\mathfrak p$ that divides them is $v_{\mathfrak p}(f) /{\ell^n}$.\end{proof} 

Putting it all together, we get the following description of the map $\psi_K : \mathrm{Cl}(K)^\vee [n]  \to \mathrm{Cl} (K)[n]$, previously defined cohomologically.

\begin{prop} \label{hilberysymboldescription} There is a natural identification

\[ \mathrm{Cl}(K)^\vee [{\ell^n}]   = \Hom \left( K^\times \backslash \mathbb A_K^\times / \prod_v  \Oo_{K_v}^\times ,\mathbb Z/{\ell^n} \right). \]

Any such homomorphism on the adeles can be written as $a \mapsto \sum_{v}\langle a,b\rangle_{{\ell^n},v}$ for some $b \in K^\times$ such that:
\begin{enumerate}
\item The valuation of $b$ at each place is a multiple of ${\ell^n}$,
\item For every place $v$ the element $b_v$ pairs trivially with all of $\Oo_{K_v}^\times$.
\end{enumerate}
The element $b$ is unique up to multiplication by elements of $(K^\times)^{\ell^n}$. 

Moreover, the ideal class $\prod_{\mathfrak p \in \spec \Oo_K}  \mathfrak p ^{v_{\mathfrak p}(b)/{\ell^n}}$ is the image of the original element of $\mathrm{Cl}(K)^\vee [{\ell^n}]  $ under the map $\psi_K: \mathrm{Cl}(K)^\vee [{\ell^n}] \to \mathrm{Cl}(K)[{\ell^n}]$ defined in Definition \ref{psi-nt-defi}. \end{prop}

\begin{proof} This follows by combining all the lemmas in this subsection. The description of $ \mathrm{Cl}(K)^\vee [{\ell^n}] $ is Lemma \ref{adelic-dual-description}.  The description of elements in terms of Hilbert symbols is Lemma \ref{class-kummer-isomorphism}. The fact that $b$ has a valuation at each place a multiple of ${\ell^n}$ is the commutativity of the diagram combined with Lemma \ref{kummer-torsor-description}. The description of $\psi_K$ follows from the commutativity of the diagram and Lemma \ref{kummer-map-description}. \end{proof}

\section{Compatibility between $\psi$ and $\omega$ in the number field case}\label{s-internal-consistency}

The goal of this section is to show that $(\mathrm{Cl} (L/K)_{\ell}, \omega_{L/K}, \psi_{L/K})$  is always contained in the support of the measure $Q^t \mu,$ defined in Definition \ref{Qtmu} and conjectured in Conjecture \ref{ntmain-2} to govern the distribution of the triples $(Cl (L/K)_{\ell}, \omega_{L/K}, \psi_{L/K}).$ Because $Q^t \mu$ is a measure on the set $\mathcal C_{\ell,n}$ of isomorphism classes of BEGs, we first, by a series of lemmas, show the compatibility condition \eqref{psiomegacomp}, which verifies that $(Cl (L/K)_{\ell}, \omega_{L/K}, \psi_{L/K})$ is a BEG. We next show that $\operatorname{ker} \psi_{L/K}$ has rank $t= \frac{ [K:\mathbb Q]}{2}$, which is sufficient, by the measure calculation in the next section, to imply that it lies in the support.

\subsection{Equivalence between Artin-Verdier and Class Field Theory}

We begin by establishing the compatibility between two separate pairings on $H^1$.  As $\omega_{L/K} $ is defined using Artin-Verdier Duality, 
this will ultimately allow us to relate $\omega_{L/K}$ and $\psi_{L/K}$. 

As $\omega$ is defined as a series of pairings indexed by a natural number $m$, we will consider throughout a natural number $m$, not necessarily equal to $n$, and the $\ell^m$th roots of unity $\mu_{\ell^m}$.

\begin{lemma}\label{lem-Brauer-invariant}  Let $\mathfrak p$ be a prime of $\mathcal O_K$, $K_{\mathfrak p}$ the corresponding local field, $\pi$ a uniformizer, $\kappa(\pi) \in H^1 ( K_{\mathfrak p} , \mu_{\ell^m})$ the image of $\pi$ under the connecting map from the Kummer sequence. Let $\alpha \in H^1 ( \mathcal O_{K_{\mathfrak p} } , \mathbb Z/\ell^m)$ be a torsor which, viewed as a map from the fundamental group of $\mathcal O_{K_{\mathfrak p}}$ to $\mathbb Z/\ell^m$, sends $\operatorname{Frob}_{\mathfrak p}$ to $k \in \mathbb Z/\ell^m$. 

Regard $\kappa(\pi) \cup (\alpha) \in H^2 (K_{\mathfrak p}, \mu_{\ell^m})$ as an element of the Brauer group. Then we have the formula for the invariant of the Brauer class \[ \operatorname{inv} ( \kappa(\pi) \cup (\alpha)  ) = -\frac{k}{\ell^m} .\]

\end{lemma}

\begin{proof}  It suffices to check this in the case $k=1$, as the torsor with $k=1$ generates the group of torsors.  This we now do by an explicit calculation with Brauer groups. 

Let $n=\ell^m$.  We let $\phi$ be the 2-cocycle  $\kappa(\pi)\cup\alpha$. Letting $\pi_0^n=\pi$ and computing explicitly, we see that 
$$\phi(\sigma,\tau) = \left(\frac{\sigma(\pi_0)}{\pi_0}\right)^{n_\tau}$$ where $n_\tau$ acts on the residue field by the $n_\tau$th power
of Frobenius. 	

 Let $K_{n}$ be the unramified extension of $K$ of degree $n,$ and $L=K_n(\pi_0)$.  To an element $g\in G=\Gal(L/K)$ we assign $(\mu_g,n_g)$ such that
$g\mid_{K_n}=F^{n_g}$ and $\frac{g\pi_0}{\pi_0} = \mu_g$. This identifies $G$ with the semi-direct product of 
$\mu_n$ and $\Z/n\Z$.

Now we may rewrite our cocycle $\phi$  as $\phi(g,h):=\mu_g^{n_h}$. By the standard dictionary between $H^2(K,\mu_n)$ and $n$-torsion in the Brauer group $\mathrm{Br}(K)$ \cite[IV,Cor 3.15]{MilneClassFieldTheory}, this gives rise to the central simple algebra $A_\phi$ given by
$$A_\phi:=\displaystyle\bigoplus_{g\in G} Le_g$$ with multiplication defined by  
\begin{align*}
e_g \ell &= g(\ell)e_g, \\ 
e_ge_h &= \phi(g,h)e_{gh}.
\end{align*}

We'll exhibit a simple subalgebra $B$ of $A_\phi$ satisfying:
\begin{itemize}
\item
$Z_{A_\phi}(B)$ is isomorphic to a matrix algebra over $K,$ say $M_d(K).$

\item
$\mathrm{inv}_K(B) = -\frac{1}{n}.$
\end{itemize}	

By the Centralizer Theorem, $A_\phi \cong B\otimes_K Z_{A_\phi}(B).$  So by the above two items, it follows that 
\begin{align*}
\mathrm{inv}(A_\phi) &= \mathrm{inv}_K(B) + \mathrm{inv}_K(Z_{A_\phi}(B)) \\
&= -\frac{1}{n} + \mathrm{inv}_K( M_m(K) ) \\
&= - \frac{1}{n}.
\end{align*}

Because $\mu_n \subset K_n,$ we can diagonalize the ($K$-linear) action of $F$ on $K_n$: $K_n = \bigoplus_{\mu \in \mu_n} K_{\mu},$ where
$K_{\mu} = \{\alpha \in K_n: F(\alpha) = \mu \cdot \alpha\}.$  Let $C :=\displaystyle\sum_\mu K_\mu e_{\mu,0}$.  Note that $K_n \cong C$ via the isomorphism $\sum_{\mu \in \mu_n} k_{\mu} \mapsto \sum_{\mu \in \mu_n} k_{\mu} e_{\mu,0}$; ; via this isomorphism, we see that Frobenius on $C$ equals $k e_{\mu,0} \mapsto F(k) e_{F(\mu),0}$ for $k \in K, \mu \in \mu_n.$  Define $B := C[\pi_0^{-1}e_{1,1}].$ \newline

We claim that conjugation by $\pi_0^{-1} e_{1,1}$ induces Frobenius on $C.$  Indeed, for $k \in K_n$, we have

\begin{align}\label{comm}
(\pi_0^{-1}e_{1,1})k e_{\mu,0}(\pi_0^{-1} e_{1,1})^{-1} &= \pi_0^{-1}e_{1,1}k e_{\mu,0}\pi_0e_{1,-1} \nonumber \\
&=\pi_0^{-1}F(k)e_{1,1}\mu\pi_0 e_{\mu,0}e_{1,-1} \nonumber \\
&=F(k)F(\mu) e_{1,1}e_{\mu,0}e_{1,-1} \nonumber \\
&=F(k)F(\mu)  e_{F(\mu),1}e_{1,-1} \nonumber \\
&=F(k)F(\mu)  F(\mu)^{-1} e_{F(\mu),0} \nonumber \\
&= F(k)e_{F(\mu),0}. 
\end{align}

It follows from the above that $C$ is its own centralizer in $B$. Indeed, let $c\in C$ be a generator over $K$. Then for a general element $\sum_
{k=0}^{n-1} c_k (\pi_0^{-1} e_{1,1})^k\in B$ commuting with $c$, we see that $$c\sum c_k (\pi_0^{-1} e_{1,1})^k = \sum c_kF^k(c) (\pi_0^{-1} e_{1,1})^k$$ from which it follows that $c_k=0$  for $k\neq 0$. Also, considering the centralizer of $\pi_0^{-1}e_{1,1},$ it follows from \eqref{comm} that the center of $B$ equals $K$. 

We next verify that $B$ is simple.  
Let $J$ be a non-zero 2-sided ideal. Then $J$ is a vector space over $C$ under left mutiplication. Conjugation by $C^\times$ breaks $B$ up into the $1$-dimensional eigenspaces over $C$ with distinct characters, namely $C(\pi_0^{-1} e_{1,1})^k$ has the character $c\ra \frac{c}{F^k(c)}$.  Because $J$ is invariant under $C^\times$ conjugation, it must contain at least one of these eigenspaces.  Each of these eigenspaces contains a unit, and therefore $J=B$.

From the above computation and since the valuation of $\pi_0^{-1}e_{1,1}$ is $-\frac{1}{n}$
it follows that $\inv_K(B)=-\frac{1}{n}$ (see \cite[IV.4]{MilneClassFieldTheory}).

\medskip

It remains to show that $Z_{A_\phi}(B)$ is isomorphic to a matrix algebra over $K.$ To do this, let $R=K_n[e_{\mu,0}, \mu \in \mu_n].$  The action of Frobenius on $R$ is given by Frobenius on $K_n$ and $F(e_{\mu,0}):=e_{F(\mu),0}$. Note that the fixed algebra $R^F$ is contained in $Z_{A_\phi}(B)$ by equation
 \eqref{comm}. Now we have an isomorphism $\phi:R\ra K_n^n$ sending $y=\sum_\mu a_\mu e_{\mu,0}$ to 
 $\phi(y) = (\sum_\mu a_\mu \mu^k)_{k = 0,\ldots,n}$ which is $F$-equivariant for the component-wise action of $F$ on $K_n^n$. 
 
 Thus we see that $R^F\cong K^n$. It follows that $Z_{A_\phi}(B)$ is a central simple algebra of dimension $n^2$ containing $n$
mutually orthogonal idempotents, and thus $Z_{A_\phi}(B)\cong M_n(K)$, completing the proof.
\end{proof}

\begin{lem}\label{lem-giant-diagram} For any $\alpha \in H^1 (\mathcal O_K, \mathbb Z/\ell^m)$, the following diagram commutes:
\[ \begin{tikzcd}
H^0(K_{\mathfrak p}, \mathbb G_m) \arrow[r] \arrow[d, "\kappa"]& H^{1}_c(U, \mathbb G_m) \arrow[r] \arrow[d, "\kappa"]&  H^1_c( \mathcal O_K, \mathbb G_m) \arrow[r] \arrow[d, "\kappa"]& H^1 ( \mathcal O_K, \mathbb G_m)\arrow[d, "\kappa"] \\
H^1(K_{\mathfrak p}, \mathbb \mu_{\ell^m} ) \arrow[r]\arrow[d,"\cup \alpha"] & H^{2}_c(U,  \mu_{\ell^m} ) \arrow[r] \arrow[d,"\cup \alpha"]&  H^2_c( \mathcal O_K,  \mu_{\ell^m} ) \arrow[r] \arrow[d,"\cup \alpha"]& H^2 ( \mathcal O_K,  \mu_{\ell^m} )\arrow[d,"\cup \alpha"] \\
H^2(K_{\mathfrak p}, \mathbb \mu_{\ell^m} ) \arrow[r] & H^{3}_c(U,  \mu_{\ell^m} ) \arrow[r] &  H^3_c( \mathcal O_K,  \mu_{\ell^m} ) \arrow[r] & H^3 ( \mathcal O_K,  \mu_{\ell^m} ) \\ \end{tikzcd}\]
where the maps $\kappa$ arise from the Kummer sequence, the horizontal arrows in the left and right columns arise from the exact sequence of compactly supported cohomology \cite[III, Proposition 0.4(a)]{MilneArithmeticDuality}, and the horizontal arrows in the middle column arise from \cite[III, Proposition 0.4(c)]{MilneArithmeticDuality}. \end{lem}
\begin{proof} Let $U = \operatorname{Spec} \mathcal O_K - \{ \mathfrak p\}$  .
Milne defines the compactly supported cohomology groups of a sheaf $\mathcal F$ on $U$ as the shifted mapping cone of the natural map of complexes \[ \Gamma ( U, I^* ( \mathcal F)) \to \Gamma (K_\mathfrak p, I^*(\mathcal F)  ) \times \prod_{v |\infty} \Gamma ( K_v, I^* (\mathcal F) )\] where $I^*( \mathcal F)$ is an injective resolution of $\mathcal F$ on the flat site of $U$. Note that the restriction of $I^*(\mathcal F)$ to $\spec K_\mathfrak v$ is an injective resolution of the restriction of $\mathcal F$ to $K_{v}$. 

The compactly supported cohomology groups of $\mathcal F$ on $\mathcal O_K$ are defined similarly, as the shifted mapping cone of \[ \Gamma ( U, I^* ( \mathcal F)) \to  \prod_{v |\infty} \Gamma ( K_v, I^* (\mathcal F) ).\]

This induces natural maps \[H^i ( K_{\mathfrak p}, \mathcal F) \to H^{i+1}_c (U, \mathcal F) \textrm{ and } H^{i+1}_c ( \mathcal O_K, \mathcal F) \to H^{i+1} (\mathcal O_K, \mathcal F)\] arising directly from the construction of $H^*_c$ as a mapping cone. These calculate the left and right horizontal arrows of the diagram.

The middle arrow is constructed by first mapping $H^{i}_c (U, \mathcal F) $ to the mapping cone
\[ \Gamma ( U, I^i ( \mathcal F)) \to \Gamma_{\mathfrak p} (K_\mathfrak p, I^{i+1} (\mathcal F)  ) \times \prod_{v |\infty} \Gamma ( K_v, I^i (\mathcal F) )\] 
and then identifying this mapping cone as $H^i( \mathcal O_K, \mathcal F)$.

The first vertical arrow, the Kummer sequence, involves choosing a triple of injective resolutions of $\mu_{\ell^m}, \mathbb G_m,$ and $\mathbb G_m$ that themselves form a short exact sequence \cite[III, Proposiition 0.4(b)]{MilneArithmeticDuality}. The commutativity of the squares can be checked straightforwardly on cochains, because the inverse image along a map of sheaves, differential, and inverse image along another map of sheaves we use to define the connecting homomorphism commute with the various pullbacks of sections to different spaces we use to define the horizontal maps.

For the second vertical arrow, the cup product, after choosing an injective resolution of $\mu_{\ell^m}$, we choose a complex of flat sheaves, isomorphic to $\mathbb Z/\ell^m$, where our chosen class $\alpha$ appears as a cocycle. This can be done by choosing a finite \'{e}tale covering where $\alpha$ splits and taking the Cech complex or that covering, or more simply by choosing an extension $\mathbb Z/\ell^m \to \mathcal F \to \mathbb Z/\ell^m$ representing $\alpha$ and using the complex $\mathcal F \to \mathbb Z/\ell^m$. We then choose a further injective resolution of $\mu_{\ell^m}$ that resolves the tensor product of these two complexes. The cup product is then expressed as multiplication of cochains. Because multiplication of cochains commutes with pullback, it commutes with the horizontal maps.

Hence all the squares are commutative.

\end{proof}

\begin{prop}\label{prop-cft-av} The following two pairings  $H^1(\mathcal O_K, \mathbb G_m) \times H^1(\mathcal O_K, \mathbb Z/\ell^m) \to \mathbb Z / \ell^m$ are equal:
\begin{enumerate}
\item
Identify $H^1(\mathcal O_K, \mathbb Z/\ell^m)$ with $\mathrm{Hom}(\pi_1( \spec \mathcal{O}_K)^{\mathrm{ab}}, \mathbb{Z} / \ell^m)$ with $\mathrm{Hom}(\mathrm{Cl}(K), \mathbb{Z} / \ell^m)$ (the latter identification by class field theory).  Identify $H^1(\mathcal{O}_K, \mathbb{G}_m)$ with $\mathrm{Cl}(K).$  Then pair $\mathrm{Hom}(\mathrm{Cl}(K), \mathbb{Z} / \ell^m)$ with  $\mathrm{Cl}(K)$ by evaluation.  
\item
Map $H^1(\mathcal O_K, \mathbb G_m)$ to $H^2(\mathcal O_K, \mu_{\ell^m})$ by the connecting homomorphism from the Kummer sequence.  Then take cup product of $H^2(\mathcal O_K, \mu_{\ell^m})$ with $H^1(\mathcal{O}_K, \mathbb{Z} / \ell^m)$ which lands in $H^3(\mathcal{O}_K, \mu_{\ell^m}).$  Then apply the Artin-Verdier duality trace map $H^3(\mathcal{O}_K, \mu_{\ell^m}) \rightarrow \mathbb{Z} / \ell^m.$
\end{enumerate}
\end{prop}

\begin{proof} It suffices to check that, for all sufficiently large primes $\mathfrak p$, and all torsors $\alpha \in H^1(\mathcal O_K, \mathbb Z/\ell^m)$, the image of the class $[\mathfrak{p}]$ of the ideal sheaf $\mathfrak p$ under the pairings with $\alpha$ defined by (1) and (2) are equal; this suffices because all elements of the class group arise from infinitely many primes. By Artin reciprocity, $[\mathfrak{p}]$ corresponds to $\mathrm{Frob}_{\mathfrak{p}}$ under class field theory.  So the pairing of $\alpha$ with $[\mathfrak{p}]$ defined by (1) is simply the action of $\Frob_{\mathfrak p}$ on the $\overline{K}$-points of the torsor $\alpha.$ In particular, it depends only on the restriction of $\alpha$ to $K_{\mathfrak p}.$

We will now calculate the pairing of $\alpha$ with $[\mathfrak{p}]$ defined by (2).  By definition, this is the image of $ (\kappa[\mathfrak p] \cup \alpha) \in H^3(\mathcal O_K, \mu_{\ell^m})$ under the identification $H^3(\mathcal O_K, \mu_{\ell^m}) \cong \mathbb Z/\ell^m$ from Artin-Verdier duality.

Let $\pi$ be a uniformer of $K_{\mathfrak p}$. Let us first check that $[\mathfrak p] \in H^1 (\mathcal O_K, \mathbb G_m)$ is the image of $\pi$ under the three arrows in the top row of the commutative diagram of Lemma \ref{lem-giant-diagram}.

We can view $H^{1}_c(U, \mathbb G_m)$ as consisting of line bundles on $U$ with trivializations on the punctured formal neighborhood of the points in $S$. Under this identification, the image of $\pi$ in $H^1_c(U, \mathbb G_m)$ is the trivial line bundle on $U$ with the identity trivialization at all infinite places and with the trivialization at the place $\mathfrak p$ twisted by $\pi$.  Given a line bundle $L$ on $U$ with a trivialization on the punctured formal neighborhood of $\mathfrak p$, there is a unique way of extending it to a line bundle on $\spec \mathcal O_K$ with a trivialization on the formal neighborhood of $\mathfrak p$. This is the line bundle whose sections are sections of $L$ on $U$ whose image under the trivialization does not have a pole at $\mathfrak p$. The map $H^1_c(U, \mathbb G_m) \to H^1_c(\mathcal O_K , \mathbb G_m)$ sends a line bundle with trivialization to the extended line bundle.  Applying this to our chosen line bundle, the sections of $\mathcal O_U$ whose image under the trivialization dividing by $\pi$ do not have a pole at $\mathfrak p$ are exactly the sections in the ideal $\mathfrak p$, so the extended line bundle is the ideal sheaf $\mathfrak p$. Finally, the natural map $H^1_c(\mathcal O_K , \mathbb G_m) \to H^1(\mathcal O_K , \mathbb G_m)$ forgets the trivialization at $\infty$. Thus, $\pi$ is sent to the ideal class $[\mathfrak p]$.

The groups on the bottom row of the diagram are all isomorphic to $\mathbb Z/\ell^m$ in a standard way. For instance this follows from \cite[Proposition 2.6 on page 169]{MilneArithmeticDuality}, which shows that $H^2( K_\mathfrak p, \mathbb G_m) = H^3_c(U, \mathbb G_m) = H^3_c(\mathcal O_K, \mathbb G_m) = H^3(\mathcal O_K, \mathbb G_m)= \mathbb Q/\mathbb Z$ and $H^1( K_\mathfrak p, \mathbb G_m) = H^2_c(U, \mathbb G_m) = H^2_c(\mathcal O_K, \mathbb G_m) = H^2(\mathcal O_K, \mathbb G_m)= 0$.  The standard isomorphism on the bottom-left is the invariant map of the Brauer group. The standard map on the bottom-right is used to define the Artin-Verdier pairing. Hence we have
\begin{align*}
&( \alpha, \kappa [\mathfrak{p}] )_{\mathrm{AV}} \\
&= \operatorname{inv} ( \kappa( \pi ) \cup \alpha) \\
&= \alpha ( \operatorname{Frob}_{\mathfrak p})
\end{align*}


by Lemma \ref{lem-Brauer-invariant}.  This matches the pairing of $\alpha$ with $[\mathfrak p]$ under the pairing (1), as desired. 


\end{proof}

\subsection{Checking compatibility between $\psi$ and $\omega$}

In this section we proceed to use Proposition \ref{prop-cft-av} to establish relation (1), the compatibility relation between $\psi$ and $\omega.$  To do this, we find it convenient to study the connecting homomorphisms associated to certain
explicit group scheme extensions between $\mathbb G_m$ and $\mu_{\ell^N}$ for various $N.$

\begin{definition} Let $\langle \alpha, \beta \rangle:  \mathrm{Cl}(K)^\vee [\ell^n] \times \mathrm{Cl}(K)^\vee \to \mathbb Z/\ell^n$ be the pairing defined by applying the class field theory isomorphism $\mathrm{Cl}(K)^\vee [\ell^n]  \cong H^1( \mathcal O_K , \mathbb Z/\ell^n)$, the map $\mathbb Z/\ell^n \to \mu_{\ell^n}$ defined by our fixed generator of $\mu_{\ell^n},$ and the Kummer map to map $\alpha$ to $\mathrm{Cl}(K)$ and then contracting with $\beta$.  \end{definition}

Throughout this subsection we fix $m=n+r$ .

\begin{definition} Let $G$ be the group scheme that sits in the middle of the exact sequence  $0 \to  \mu_{\ell^{n+r} } \to G\to \mathbb Z/\ell^{n+r} \rightarrow 0$ obtained by pulling back the Kummer exact sequence $0 \to \mu_{\ell^{n+r} } \to \mathbb G_m \to \mathbb G_m \rightarrow 0$ along the series of maps $\mathbb Z/\ell^{n+r} \to \mathbb Z/\ell^n\to \mu_{\ell^n} \to \mathbb G_m,$ where the second map is defined via our fixed choice of generator $\zeta$ for $\mu_{\ell^n}$.\end{definition}

\begin{lemma} The group scheme $G$ consists of pairs $(a,x)$ with $a \in \mathbb Z/\ell^{n+r}$, $x \in \mu_{\ell^{2n+r}}$ such that $x^{\ell^{n+r}} = \zeta^a$ for $\zeta$ our chosen generate of $\mu_{\ell^n}$ \end{lemma}

\begin{proof} By definition, $G$ is the fiber product of $\mathbb G_m$ and $\mathbb Z/\ell^{n+r}$ over $\mathbb G_m$ under the $\ell^{n+r}$ power and $a \mapsto \zeta^a$ maps respectively, so it consists of pairs $x \in \mathbb G_m, a\in \mathbb Z/\ell^{n+r}$ with $x^{\ell^{n+r}} = \zeta^a$, which because $\zeta^{l^n}=0$ forces $x^{\ell^{2n+r}}=1$ so $x \in \mu_{\ell^{2n+r}}$. \end{proof}

\begin{definition} Let $B: H^i ( \mathcal O_K, \mathbb Z/\ell^{n+r}) \to H^{i+1} (\mathcal O_K, \mu_{\ell^{n+r}})$ be the connecting homomorphism associated to this exact sequence $0 \to \mu_{\ell^{n+r} } \to G \to \mathbb Z/\ell^{n+r} \rightarrow 0$. \end{definition}

\begin{lemma}\label{langle-rangle-B} For $\alpha, \beta  \in \mathrm{Cl}(K)^\vee [\ell^{n+r}]$, viewed as elements of $H^1 (\mathcal O_K, \mathbb Z/\ell^{n+r}),$  we have \[\langle \ell^r \alpha, \beta \rangle = (B\alpha,\beta)_{AV} =\operatorname{tr} ( B\alpha \cup \beta)  \] where $\operatorname{tr} :H^3 (\mathcal O_K, \mu_{\ell^{n+r}} ) \to \mathbb Q/\mathbb Z$ is the Artin-Verdier trace.

\end{lemma}

\begin{proof} By definition,  $\langle \alpha, \beta \rangle$ is obtained by taking $\ell^r\alpha$, viewing it as an element of $H^1( \mathcal O_K , \mathbb Z/\ell^n)$, mapping to $H^1(\mathcal O_K, \mu_{\ell^n})$ (via our fixed choice of generator $\zeta$ for $\mu_{\ell^n}$) and then to $H^1(\mathcal O_K, \mathbb G_m)$, and then contracting with $\beta.$  By Proposition \ref{prop-cft-av}, this is equivalent to taking $\ell^r \alpha$, mapping to $H^1(\mathcal O_K, \mathbb G_m)$ along that series of maps, applying the connecting homomorphism from the Kummer sequence to map into $H^2(\mathcal O_K, \mu_{\ell^m})$, and then taking an Artin-Verdier pairing of the result with $\beta$.

Viewing $\ell^r \alpha$ as an element of  $H^1( \mathcal O_K , \mathbb Z/\ell^n)$ is equivalent to viewing $\alpha$ as an element of  $H^1( \mathcal O_K , \mathbb Z/\ell^{n+r} )$ and mapping to $H^1( \mathcal O_K , \mathbb Z/\ell^n)$ by reduction mod $\ell^n$. So all told this is equivalent to sending $\alpha$ from $H^1( \mathcal O_K , \mathbb Z/\ell^n)$ to $H^1(\mathcal O_K, \mathbb G_m)$ by the composed map $\mathbb Z/\ell^{n+r} \to \mathbb G_m$, which is $a\mapsto \zeta^a$, applying the Kummer exact sequence connecting map, and then Artin-Verdier duality. 

Because $G$ is the pullback of the Kummer exact sequence on that series of maps, $B$ is the composition of that series of maps with the Kummer exact sequence. 

Finally, the relation between Artin-Verdier duality and the Artin-Verdier trace is simply the definition of the Artin-Verdier duality map.
\end{proof}

\begin{definition} Let $G'$ be the group scheme consisting of pairs $a \in \mathbb Z/\ell^{2n+r}, x \in \mu_{\ell^{2n+r}}$ such that $x^{n+r} = \zeta^{2a}$, modulo the subscheme of pairs $(\ell ^{n+r} t, \zeta^{t})$. \end{definition}

 Then $G'$ has a map to $\mathbb Z/\ell^{n+r}$ given by taking $a$ modulo $\ell^{n+r}$, whose kernel is isomorphic to $\mu_{\ell^{n+r}}$ under the map $(a,x) \mapsto x \zeta^{- a / \ell^{n+r}}$.
 
 \begin{definition} Let $B': H^i ( \mathcal O_K, \mathbb Z/\ell^{n+r}) \to H^{i+1} (\mathcal O_K, \mu_{\ell^{n+r}})$ be the connecting homomorphism associated with the exact sequence $0\to \mu_{\ell^{n+r}} \to G' \to \mathbb Z/\ell^{n+r} \to 0 $. \end{definition}
 
 \begin{definition} Let $f: G \times G \to G'$ be the bilinear map of group schemes that sends $(a_1, x_1) \times (a_2, x_2)$ to $(\tilde{a}_1\tilde{a}_2, x_1^{\tilde{a}_2} x_2^{\tilde{a}_1})$, where $\tilde{a}_1$ and $\tilde{a}_2$ are lifts of $a_1$ and $a_2$ respectively from $\mathbb Z/\ell^{n+r}$ to $\mathbb Z/\ell^{2n+r}$. \end{definition}

 \begin{lemma} The map $f$ is well-defined. \end{lemma}
 
 \begin{proof} Adding $\ell^{n+r}$, say to $a_1$, has the effect of adding $\ell^{n+r} \tilde{a}_2$ to the first coordinate and multiplying the second coordinate by $x_2^{ \ell^{n+r} } = \zeta^{a_2}$. \end{proof}
 
 \begin{lemma} \begin{enumerate}
 
 \item The map $f$ is compatible with the projections onto $\mu_{\ell^{n+r}}$.
 
 \item The maps $\mu_{\ell^{n+r} } \times G \to  \mu_{\ell^{n+r}} \subset G'$ and $ G\times \mu_{\ell^{n+r} }  \to  \mu_{\ell^{n+r}} \subset G'$ induced by $f$ are the same as those obtained by projecting from $G$ to $\mathbb Z/\ell ^{n+r}$ and taking the obvious multiplication. 
 
 \end{enumerate}
 
 \end{lemma}
 
 \begin{proof} Both can be checked immediately.\end{proof}

\begin{lemma}\label{B-B'}  For $\alpha, \beta  \in H^1 (\mathcal O_K, \mathbb Z/\ell^{n+r}),$  we have \[ B \alpha \cup \beta -  \alpha \cup B\beta = B'(\alpha \cup \beta).\]\end{lemma}

\begin{proof}  To do this, we use the facts that there is a natural map from Cech cohomology of a sheaf to usual cohomology, compatible with cup products and connecting homomorphisms, and that it is an isomorphism in degree $1$. Thus we can represent $\alpha$ and $\beta$ by Cech cocycles. We can calculate the connecting homomorphism explicitly as, first, an arbitrary lift of those cocycles from $\mathbb Z/\ell^{n+r}$ to $G$ (possibly after refinement), second applying the Cech differential, and third recognizing the result as a cocycle for $\mu_{\ell^{n+r}}$.

The bilinear map $G \times G \to G'$ induces a cup product where we cup $G$-cochains with $G$-cochains to obtain $G'$-cochains, and in particular for $\tilde{\alpha}$ a lift of $\alpha$ to $G$ and $\tilde{\beta}$ a lift of $\beta$ to $G$, $\tilde{\alpha} \cup \tilde{\beta}$ is a lift of $\alpha \cup \beta$ to $G'$.  Then we have

\[ B' (\alpha \cup \beta) = d_{G'} (\tilde{\alpha} \cup \tilde{\beta}) = d_G \tilde{\alpha} \cup \tilde{\beta} + \alpha \cup d_G \tilde{\beta} \] and $d_G \tilde{\alpha}$ is a cochain for $G$ such that $\pi ( d\tilde{\alpha} )= d \pi(\tilde{\alpha})=  d\alpha=0$, hence is a class in $H^2( \mathcal O_K, \mu_{\ell^{n+r}})$ and in fact is $B\alpha$, so its cup product with $\tilde{\beta}$ is the same as the cup product with $\pi(\tilde{\beta})= \beta$, so \[ d_G \tilde{\alpha} \cup \tilde{\beta}  = B \alpha \cup \beta\] and similarly the other term is $\alpha \cup B\beta$.\end{proof}

\begin{lemma}\label{langle-rangle-B'} For $\alpha, \beta  \in Cl(K)^\vee [\ell^{n+r}]$, viewed as elements of $H^1 (\mathcal O_K, \mathbb Z/\ell^{n+r}),$  we have \[ \langle \ell^r \alpha,\beta\rangle - \langle \ell^r \beta, \alpha \rangle =\tr ( B' ( \alpha \cup \beta)).\] \end{lemma}

\begin{proof} This follows on combining Lemmas \ref{langle-rangle-B} and \ref{B-B'}, upon remembering that $\alpha \cup B \beta = B\beta\cup \alpha$ because $B \beta$ is in degree $2$, which is even. \end{proof}

\begin{definition} Let $G^*$ be the Cartier dual of $G'$. Let $B^*: H^i ( \mathcal O_K, \mathbb Z/\ell^{n+r}) \to H^{i+1} (\mathcal O_K, \mu_{\ell^{n+r}})$ be the connecting homomorphism associated to $G^*$.  \end{definition}

\begin{lemma}\label{B'-B*} We have \[  \tr( B' ( \alpha \cup \beta))=- \tr( \alpha \cup \beta \cup B^* (1)) \]

 
  \end{lemma}

\begin{proof}  Because the trace map factors through $H^3(\mathcal O_K, \mathbb G_m$, it suffices to show that \[ B' ( \alpha \cup \beta)+  \alpha \cup \beta \cup B^* (1) =0 \in H^3(\mathcal O_K, \mathbb G_m).\]

As in the proof of Lemma \ref{B-B'}, we may assume that $\alpha$ and $\beta$ are Cech cocycles, and perform the calculations in cohomology. We can lift $\alpha \cup \beta \in C^2( \mathcal O_K, \mathbb Z/\ell^{n+r}) $ to a cochain $\tilde{\alpha} \cup \tilde{\beta} \in C^2(\mathcal O_K, G')$ and lift $1\in H^0 (\mathcal O_K, \mathbb Z/\ell^{n+r})$ to a cochain  $\tilde{1} C^0(\mathcal O_K, G^*$. By definition, $B' (\alpha \cup \beta)$ consists of applying the differential to obtain a cochain in $C^3 (\mathcal O_K, G')$, pulling back to $C^3(\mathcal O_K, \mu_{\ell^{n+r})}$, and then mapping to $C^3 (\mathcal O_K, \mathbb G_m)$. 

This last step is equivalent to cupping with $1$ under the Cartier duality pairing  $\mu_{\ell^{ n+r}} \times \mathbb Z/\ell^{n+r} \to \mathbb G_m$. So the pulling back to $\mu_{\ell^{n+r}}$ and then mapping to $\mathbb G_m$ is equivalent to cupping with $\tilde{1}$ under the Cartier duality pairing $G' \times G^* \to \mathbb G_m$, which extends it. So

\[ B' (\alpha \cup \beta) = d( \tilde{\alpha} \cup \tilde{\beta} ) \cup \tilde{1} .\]

Similarly, we have \[(\alpha \cup \beta) \cup B^*(1) = (\tilde{\alpha} \cup \tilde{\beta} ) \cup d \tilde{1} ,\] because $ d \tilde{1}$ is the image of $B*(1)$ under the map $C^1 ( \mathcal O_K, \mu_{\ell^{n+r}}) \to C^1 (\mathcal O_K, G^*)$ and $\alpha \cup \beta$ is the image of $\tilde{\alpha} \cup \tilde{\beta}$ under the dual map $C^2( \mathcal O_K, G' ) \to C^1 (\mathcal O_K, \mathbb Z/\ell^{n+r} )$. Thus, 

\[  B' ( \alpha \cup \beta)+  \alpha \cup \beta \cup B^* (1) = d ( \tilde{\alpha} \cup \tilde{\beta} \cup \tilde{1}) \] which is a coboundary and thus vanishes in cohomology, as desired. \end{proof}

\begin{lemma}\label{langle-rangle-B*} For $\alpha, \beta  \in Cl(K)^\vee [\ell^{n+r}]$, viewed as elements of $H^1 (\mathcal O_K, \mathbb Z/\ell^{n+r}),$  we have \[ \langle \ell^r \alpha,\beta\rangle - \langle \ell^r \beta, \alpha \rangle =( B^* 1, \alpha\cup \beta)_{AV} .\]\end{lemma}

\begin{proof} By Lemmas \ref{langle-rangle-B'} and \ref{B'-B*}, we have 
\begin{align*}
\langle \ell^r \alpha,\beta\rangle - \langle \ell^r \beta, \alpha \rangle &= \tr (  B' (\alpha \cup \beta) ) \\
&= \tr ( B'(\alpha \cup \beta) \cup 1 ) \\
&= \tr( \alpha \cup \beta \cup B^* 1) \\ 
&= \tr ( B^*1 \cup \alpha \cup \beta ) \\
&= (B^*1, \alpha\cup \beta)_{AV}.
\end{align*}
\end{proof}

\begin{lemma}\label{B^*-alpha} We have $B^* 1 = - \alpha_{n+r}$ where $\alpha_{n+r} \in H^1( \mathcal O_K, \mu_{\ell^{n+r}})$ is identified with the torsor of $\ell^{n+r}$th roots of our fixed generator $\zeta$ of $\mu_{\ell^n}$ (defined in \S\ref{number-field-invariants}). \end{lemma}

\begin{proof} By definition, $B^*1$ is the torsor defined as the inverse image of $1\in \mathbb Z/\ell^{n+r}$ under the exact sequence $0 \to \mu_{\ell}^{n+r} \to G^* \to \mathbb Z/\ell^{n+r} \to 1$. 

To calculate this, observe by dualizing the definition of $G'$ that $G^*$ is the group consisting of pairs of $(a^* \in \mu_{\ell^{2n+r}}, x^* \in \mathbb Z/\ell^{2n+r}$ such that $\left(a^*\right)^{\ell^{n+r}}  \zeta^{ x^*}=1$, modulo the subgroup generated by $(\zeta^2, \ell^{n+r} )$. 

To calculate the torsor, we look at the set of elements sent to $1 \in \mathbb Z/\ell^{n+r}$ by the projection onto $x^*$ mod $\ell^{n+r}$, which is isomorphic under the map $(a^*, x^*) \to a^* \zeta^{- 2  \frac{x^* -1}{\ell^{n+r}}} $ to the set of $\ell^{n+r}$th roots of $\zeta^{-1}$. 

On the other hand, $\alpha_{n+r}$ is defined as the torsor of $\ell^{n+r}$th roots of $\zeta$. Since they are inverse torsors, their cohomology classes are negatives of each other. \end{proof}


%
%
%
%
%
%
%
%
%

\begin{lemma}\label{psi-omega-comp-verify} Let $m=n+r$ and let $\alpha$ and $\beta$ be elements of $Cl(K)^\vee [\ell^m].$  There is an equality
\[ \langle \ell^r \alpha,\beta\rangle - \langle \ell^r \beta, \alpha \rangle =2 \cdot \omega_{n+r,K} (\alpha,\beta) \] 
\end{lemma}

\begin{proof} This follows from Lemmas \ref{langle-rangle-B*} and \ref{B^*-alpha} as well as Definition \ref{omega-nt-defi}. \end{proof} 

\subsection{Non-degeneracy of $\psi_{L/K}$}

Finally, we check the non-degeneracy condition for $(\mathrm{Cl}(L/K)_{\ell}, \omega_{L/K}, \psi_{L/K})$ to be contained in the support of $Q^t \mu$.

\begin{lemma}\label{number-field-rank-bound} The $\ell$-rank of the kernel of $\psi_{L/K}$ is at most $t=\frac{[K:\mathbb Q]}{2} $.\end{lemma}

\begin{proof} 
We first focus on $\psi_K$ for a single field $K$. By the definition of $\psi_K$, the kernel of $\psi_K$ consists
of those elements of $H^1(\spec\Oo_K,\Z/\ell^n\Z)$ which map to those elements of $H^1(\spec\Oo_K,\mu_{\ell^n})$ which have trivial image in the class group. We have the commutative diagram 

\[ \begin{tikzcd} H^1(\spec \Oo_K, \mathbb Z/{\ell^n}) \arrow[r]\arrow[d] & H^1(K, \mathbb Z/{\ell^n}) \arrow[d] \\ H^1( \spec \Oo_K, \mu_{\ell^n}) \arrow[r] & H^1(K, \mu_{\ell^n}) \\ \end{tikzcd}\] in which the top arrow is injective because $\spec \Oo_K$ is normal, and the right arrow is in isomorphism, so the map $H^1(\spec \Oo_K, \mathbb Z/{\ell^n}) \to H^1(K, \mu_{\ell^n})$ is injective, hence the map $H^1(\spec\Oo_K,\Z/\ell^n\Z) \to H^1(\spec\Oo_K,\mu_{\ell^n})$ is injective.

Thus we may identify $\ker \psi_K$ with a subgroup of the kernel of the natural map $H^1(\spec\Oo_K,\mu_{\ell^n})\to H^1( \spec \Oo_K, \mathbb G_m)$, which by the Kummer exact sequence is $\Oo_K^{\times}\otimes\Z/\ell^n\Z$.

Since $K$ contains the $\ell^n$th roots of unity, it is totally complex and has unit rank $\frac{ [K:\mathbb Q]}{2}$. Thus by Dirichlet's unit theorem, $\Oo_K^{\times}\otimes\Z/\ell^n\Z \cong (\Z/\ell^n\Z)^{\frac{[K:\Q]}{2} +1}$, where the $+1$ comes from torsion.  Thus, the $\ell$-rank of its subgroup $\ker\psi_K$ is at most $\frac{[K:\Q]}{2}+1$.

Now, we return to our setting. Note that since $L$ is odd, the order 2 automorphism of $L/K$ gives a canonical splitting 
$$\Cl(L)[\ell^n]\cong \Cl(K)[\ell^n]\oplus\Cl(L/K)[\ell^n]$$ which is respected by the map
$\psi_L$, and such that $\psi_L$ restricts to $\psi_K$. Following the above, and letting $\Oo_{L/K}^\times\subset\Oo_L^\times$ denote the kernel of the norm map to $K$, we
may identify $\ker\psi_{L/K}$ with a subgroup of $\Oo_{L/K}^\times \otimes \Z/\ell^n\Z$, and conclude that it has $\ell$-rank at most $\frac{[L:\Q]-[K:\Q]}{2}$ which is equal to $t_{L/K}$ since both $L$ and $K$ are totally complex. 
\end{proof}

\begin{prop}\label{class-group-support} The triple $(\mathrm{Cl}(L/K)_{\ell}, \omega_{L/K}, \psi_{L/K})$ is contained in the support of $Q^t \mu$ where $t = \frac{ [K:\Q]}{2}$. \end{prop}

\begin{proof} 

By Lemma \ref{psi-omega-comp-verify} and $\omega_L$ and $\psi_L$ satisfy \eqref{psiomegacomp}. Thus their projections $\omega_{L/K}$ and $\psi_{L/K}$ satisfy \eqref{psiomegacomp}. Thus by Definition \ref{beg}, $(\mathrm{Cl} (L/K)_{\ell}, \omega_{L/K}, \psi_{L/K})$ is a BEG.

Because the product in Theorem \ref{quotient-measure-formula} is manifestly nonzero, it follows from that theorem that the support of $Q^t \mu$ consists of all BEGs where the kernel of $\psi_G$ has rank at most $t$. By Lemma \ref{number-field-rank-bound}, the kernel of $\psi_{L/K}$ has rank at most $t$ and thus $(\mathrm{Cl} (L/K)_{\ell}, \omega_{L/K}, \psi_{L/K})$ is indeed contained in the support.  \end{proof}

\section{Compatibility of the general definitions of $\psi$ and $\omega$ with Frobenius}  \label{s-NT-FF-compatible}

Let $C$ be a smooth projective geometrically irreducible curve over a finite field $k$ containing the $\ell^n$th roots of unity.  Let $A = \mathrm{Pic}^0(C).$ Because $A$ is a principally polarized abelian variety, the construction in \S3.1 gives $A(k)_\ell$ the structure of a bilinearly enhanced group, which we denote $( \mathrm{Pic}^0(C)(k)_\ell, \omega_C, \psi_C)$. This is the same notation we used in the special case where $C$ is hyperelliptic in \S3.2.

The  exact sequence $0 \to \Pic^0(C)(k) \to \Pic(C) (k) \to \mathbb Z \to 0$ (which is exact by Lang's theorem) induces an isomorphism $\mathrm{zer}: \Pic^0(C)(k) [\ell^n] \to  \Pic(C) (k)[\ell^n] $ and a surjection $\mathrm{zer}^\vee:  \Pic(C) (k)^\vee [\ell^n] \to \Pic^0(C)(k)^\vee [\ell^n]$ with kernel $\mathbb Z/\ell^n$.

The aim of this section is to prove the commutative diagram \[ \mathrm{zer} \circ \psi_C \circ \mathrm{zer}^\vee  = \psi_K \] where $K$ is the function field of $C$.

\begin{lemma}\label{psi-psi-Weil} Let $B$ be an abelian variety. Let $\tau: \pi_1(B_{\overline{k}} ) \to \mu_{\ell^n}$ be a homomorphism. This induces a class $[\tau]$ in $H^1 ( B_{\overline{k}} , \mu_{\ell^n})$ and hence\footnote{$B^\vee[\ell^n] = \mathrm{Pic}^0 (B)[\ell^n] = \mathrm{Pic}(B) [\ell^n]  =H^1 ( B_{\overline{k}} , \mu_{\ell^n})$ because the component group of $\mathrm{Pic}(B)$ is torsion-free.}  an $\ell^n$-torsion class $[\tau]$  in $H^1(B_{\overline{k}}  , \mathbb G_m) [\ell^n]  = B^\vee [\ell^n]$. 

Because $\tau$ is an $\ell^n$-torsion character, we can also view it as a map  $ \tau': B[\ell^n]  \cong \pi_1(B_{\overline{k}} )  \otimes \mathbb Z/\ell^n  \to \mu_{\ell^n} $

Pairing against  $[\tau] \in B^\vee[\ell^n]$ in the Weil pairing $B[\ell^n]  \times B^\vee [\ell^n] \to \mu_{\ell^n}$ recovers the homomorphism $\tau': B[\ell^n] \to \mu_{\ell^n}$. \end{lemma}

\begin{proof} 
This follows from the definition of the Weil pairing in \cite[\S20, p.183]{Mu}. To see the equivalence between Mumford's definition and ours, recall how one obtains an $\ell^n$-torsion line bundle
from $\tau$. First, one considers the multiplication by $\ell^n$ \'{e}tale cover $m:B\ra B$. To descend the trivial line bundle $\Oo_B$ amounts to giving isomorphisms 
$r_g:\Oo_B\ra \Oo_B$ for each $g\in B[\ell^n]$, and so one simply takes $r_g=\tau'(g)$. In Mumford's notation, his $\chi$ is our $\tau'$.

\end{proof}

\begin{lemma}\label{psi-psi-Lang} Let $C$ be a curve over a finite field $k$ and let $A$ be its Jacobian, a principally polarized abelian variety. Viewing $G \subset \mathrm{Cl}(K)_\ell$ as a subset of $\pi_1(C)^{\mathrm{ab}}_\ell$ via class field theory, the natural homomorphism $\pi_1(A_{\overline{k}})_\ell = \pi_1(C_{\overline{k}})^{\mathrm{ab}}_\ell \to \pi_1(C)^{\mathrm{ab}}_\ell= \pi_1(A)^{\mathrm{ab}}_\ell$ factors through $G.$  The map to $G$ is exactly the map $\phi: A[\ell^r] \to G$ of \S3.1 composed with the natural map $\pi_1(A_{\overline{k}})_\ell  = T_\ell(A) \to A[\ell^r]$. \end{lemma}

\begin{proof}
Lang gave a beautiful construction of class field theory over function fields using the map $1-F$ and the Jacobian. He did not explain this in detail in his paper, because class field theory over function fields was already known. We explain how class field theory gives this identity. 

Fix a divisor $D$ of degree 1 on $C$, giving an identification $\mathrm{Cl}(K)_\ell = G \times \mathbb Z_\ell$ and an Abel-Jacobi map $C \to A$. Consider the abelian cover $C'$ of $C$ defined by the fiber product over $A$ of $C$ with the map $(1-F) :A \to A$, which has Galois group $A(k)$. Consider further the base change $C^*$ of $C'$ to $\overline{k}$, which has Galois group $G \times \widehat{\mathbb Z} = \widehat{ \mathrm{Cl}(K)}$.  We claim the natural map from $\pi_1^{\mathrm{ab}}(C)$ to the Galois group of this cover is the same as the identification of this Galois group with the profinite completion of $\mathrm{Cl}(K)$ under class field theory.

To do this, it suffices to check that every Frobenius element is sent to the same element of $\widehat{\mathrm{Cl}(K)}$ under these two definitions, because the Frobenius elements are dense in the Galois group. Under class field theory, the Frobenius element at a closed point $v$ is sent to the class of the line bundle $\mathcal O(-v)$, or, writing $D$ for our degree $1$ divisor, $( \mathcal O(-v + \deg v \cdot D), \deg v)$ in $G \times \widehat{ \mathbb{Z} }$. On the other hand, we can check how $\operatorname{Frob}_v$ acts on the fiber of $C^*$ over $C$ at $v$. First, we calculate the action on $C'$. Let $x$ be a geometric point of $C$ lying over 
$v.$ The points in $C'$ lying over $x$ can be expressed as pairs $(x,y)$ with $y \in A(\overline{k})$ such that $(1 - F)(y) = \mathrm{AJ}(x)$ where $\mathrm{AJ}$ is the Abel-Jacobi map. 
The action of $\operatorname{Frob}_v$ on the fiber of $C'$ at $x$ is given by $F^{\deg v}$. We have

\begin{align*}
F^{\deg v} (y) - y &=   \sum_{i=0}^{\deg v-1} F^{i+1}(y) - F^i(y) \\
&= \sum_{i=0}^{\deg v-1} F^{i} ( F(y) -y ) \\
&= - \sum_{i=0}^{\deg v-1} F^{i} (\mathrm{AJ}(x)).
\end{align*}

Now $\mathrm{AJ}(x)$ is the class of the line bundle $\mathcal O(x - D)$. So $F^i(\mathrm{AJ}(x))$ is the class of $\mathcal O( F^i(x) - D)$. Thus $\sum_{i=0}^{\deg v-1} F^{i} (\mathrm{AJ}(x))$ is the class of $\mathcal O \left( \left(\sum_{i=0}^{\deg v-1}   F^i(x)  \right)  - \deg v \cdot D \right) $. Since $\mathcal O (v) =\mathcal O  \left(\sum_{i=0}^{\deg v-1}   F^i(x)  \right)$, we conclude that $ F^{\deg v}(y) - y $ is the class of $\mathcal O ( -v + \deg v \cdot D )$. In other words, $F^{\deg v}$ acts on this fiber by translation by $\mathcal O ( -v + \deg v \cdot D ) \in G$.

Finally, $F^{\deg v}$ acts on $\overline{k}$ by $F^{\deg v}$, which corresponds to the element $\deg v \in \widehat{\mathbb Z}$. So indeed these two homomorphisms are the same on each Frobenius element, and thus equal.

Now, to understand the map $\pi_1 (A _{\overline{k}})_\ell = \pi_1(C_{\overline{k}})^{\mathrm{ab}}_\ell  \to \pi_1(C)^{\mathrm{ab}}_\ell $, it suffices to see how elements of $\pi_1 (A _{\overline{k}}) $ act on $C^*$. These elements fix $\overline{\mathbb F_q}$, so their action on $C^*$ depends only on their action on $C'$ and factors through $G$. The isomorphism $\pi_1 (A _{\overline{k}})= \pi_1(C_{\overline{k}})^{\mathrm{ab}}   $ is defined via the embedding of $C$ into $A$, so their action on $C'$ is equal to their action on the covering $(1-F)$ of $A$ by $A$, which has Galois group $A(k)$. We identify $\pi_1 (A _{\overline{k}})_\ell $ with the inverse limit of $A[\ell^n]$ for natural numbers $n$ via its actions on the coverings of $A$ by $A$ defined by the multiplication by $\ell^n$ map. Choose $m$ such that the $(1-F)$ covering factors through the multiplication by $m$ map. Let $n$ be the $\ell$-adic valuation of $m$ and let $m'=m/\ell^n$. 

There exists a homomorphism $M: A \to A$ such that $(1-F) M = m = M(1-F) $.  The image of an $\ell^n$-torsion point $x$ in $A(k)$ is then given by $M( (m')^{-1} x)$. Letting $y$ be any inverse image in $A[\ell^{\infty}]$ of $x$ under $1-F$, we have $m y = M ((1-F) y )= M( x)$ so that $\ell^n y = M ( (m')^{-1} x)$ . Thus, $x$ is obtained from $M ( (m')^{-1} x)$ by the snake lemma map of the diagram of \S3.1 - we take $M ( (m')^{-1} x)$ in $A[\ell^{\infty}]$, pull back to $A[\ell^{\infty}]$ under $\times \ell^n$ to obtain $y$, pushforward under $1-F$ to obtain $x$, and then recognize it as an element of $A[\ell^r]$. Because $\phi$ is by definition the inverse of this snake lemma map, this shows that the composition of $\phi$ with the projection from $\pi_1(A _{\overline{k}})_\ell$ matches the action of $\pi_1(A _{\overline{k}})$ on $C'$ and thus, by our earlier discussion, agrees with class field theory. 

\end{proof}

\begin{thm}

\[ \mathrm{zer} \circ \psi_C \circ \mathrm{zer}^\vee  = \psi_K .\]

\end{thm}

\begin{proof} Fix $\alpha \in H^1 ( C, \mathbb Z/\ell^n),$ which is naturally identified with $\mathrm{Cl}(K)^\vee [\ell^n]$ by class field theory.  We wish to show $\psi_K (\alpha) = \psi_C ( \mathrm{zer}^\vee(\alpha))$ as elements of $\mathrm{Cl}(K)[\ell^n]$. Because the natural map $\mathrm{Cl}(K)[\ell^n] = \Pic(C) [\ell^n] \to \Pic(C_{\overline{k}} ) [\ell^n]$ is injective, it suffices to check that the pullbacks of $\psi_K (\alpha) $and $\psi_C ( \mathrm{zer}^\vee(\alpha))$ to $C_{\overline{k}}$ are equal.

 Let $\overline{\alpha}$ be the pullback of $\alpha$ to $C_{\overline{k}}$. 
 
 Recall, from Definition \ref{psi-nt-defi}, that $\psi_K(\alpha)$ is defined via the composition $H^1 ( C, \mathbb Z/\ell^n)  \to H^1(C, \mu_{\ell^n} ) \to H^1(C, \mathbb G_m)$. Let $\overline{\psi}_K$ be defined by the analogous composition $H^1 ( C_{\overline{k}} , \mathbb Z/\ell^n)  \to H^1(C_{\overline{k}} , \mu_{\ell^n} ) \to H^1(C_{\overline{k}} , \mathbb G_m)$. 
 Because these maps are compatible with the pullback to $C_{\overline{k}}$, the pullback of $\psi_K(\alpha)$ to $C_{\overline{k}}$ is $\overline{\psi}_K(\overline{\alpha})$.

We identify $H^1(C_{\overline{k}}, \mathbb Z/\ell^n)$ with the set of homomorphisms from $T_\ell(A)$ to $\mathbb Z/\ell^n$.  Viewing $\alpha$ as a homomorphism $\Pic(C) (k) \to \mathbb Z/\ell^n$, the pullback $\overline{\alpha}$ of $\alpha$ is obtained by composing $\alpha$ with the projection $T_\ell(A) \to \Pic(C)(k)_\ell$.  The latter composition factors through  $T_\ell(A) / \ell^n,$ which is naturally identified with $A[\ell^n].$ By Lemma \ref{psi-psi-Weil}, the Weil pairing with $\overline{\psi}_K(\overline{\alpha})$ equals this homomorphism $A [\ell^n] \to \mathbb Z/\ell^n$, composed with the map $\mathbb Z/\ell^n \to \mu_{\ell^n}$ defined by $\zeta$. 

Because the Weil pairing is a perfect pairing, it suffices to check that the Weil pairing with $\overline{\psi}_K(\overline{\alpha})$ equals the Weil pairing with the pullback of $\psi_C( \mathrm{zer}^\vee (\alpha))$, which we now compute, following the construction of \S\ref{abomegapsi}.

We have the map $\phi: T_\ell(A)  \to \mathrm{coker}(1-F | \; T_\ell(A) )  \cong \Pic^0(C) (k)_\ell$ and the Cartier dual map $\phi^\vee:  \Pic^0(C) (k)_\ell^\vee \to A[\ell^\infty].$

  The $\ell^n$-torsion element $\phi^\vee ( \mathrm{zer}^\vee (\alpha))$  of $A[\ell^\infty]$  is Cartier dual to a map $T_\ell(A) \to \mu_{\ell^n}$. This dual map is obtained by first applying the projection $\phi: A[\ell^m] \to \Pic^0(C) (k)$, then $\mathrm{zer}: \Pic^0(C)(k) \to \Pic(C) (k)$, and then applying $\alpha$. (Because $\alpha$ is an element of the dual $\mathrm{Cl}(K)^\vee [\ell^n]$, i.e. the space of linear forms  $\mathrm {Cl}(K) \to \mathbb Z/\ell^n$, this map is referred to as $\alpha$ and not $\alpha^\vee$.)

   Since the duality between $T_\ell(A)$ and $A[\ell^{\infty}]$ comes from the Weil pairing, $\phi^\vee ( \mathrm{zer}^\vee (\alpha))$ is the element of $A[\ell^m]$ whose Weil pairing with an element of $A[\ell^m]$ is this composition $\alpha^\vee \circ \mathrm{zer} \circ \phi$. On the other hand, we just saw that the Weil pairing of $\overline{\psi}_K(\overline{\alpha})$ with an element of $A[\ell^n]$ is obtained by lifting to an element of $T_\ell(A)$, projecting to $\Pic^0(C)(k)_\ell$, including into $\Pic(C)(k)_\ell$ and then applying $\alpha$. We chose $m$ so that the map $T_\ell(A) \to  \Pic^0(C)(k)_\ell$ factors through $A[\ell^m]$ and then defined $\phi$ to be this factorization. Thus, we see that given any element of $A[\ell^m]$, taking the Weil pairing with $\phi^\vee ( \mathrm{zer}^\vee (\alpha))$ is equivalent to multiplying by $\ell^{m-n}$ to obtain an element of $A[\ell^n]$ and taking the Weil pairing with $\overline{\psi}_K(\overline{\alpha})$. In other words, embedding $A[\ell^n]$ in $A[\ell^m]$ the usual way, we have $ \phi^\vee ( \mathrm{zer}^\vee (\alpha))= \overline{\psi}_K(\overline{\alpha}) $

In other words,  $\phi^\vee ( \mathrm{zer}^\vee (\alpha))$ is the image of $\overline{\psi}_K ( \overline{\alpha})$ under the inclusion map $A[\ell^n] \to A[\ell^m]$. By the definition of $\psi$ (Definition \ref{def-psi-G}), it suffices to show that for $\alpha_0 \in T_\ell(A) \otimes \mathbb Q_\ell $ such that $ \alpha_0 \mod T_\ell(A) = \phi^\vee ( \mathrm{zer}^\vee (\alpha)) \psi_C(\alpha)=  \overline{\psi}_K(\overline{\alpha})$, we have $\phi ( (1-F)  \alpha_0 )  =  \overline{\psi}_K(\overline{\alpha})$.  But $\phi$ was defined by the snake lemma as the inverse of the map obtained by taking a lift along the map $T_\ell(A) \otimes \mathbb Q_\ell \to A[\ell^{\infty} ]$ and then applying $1-F$, so this identity follows.

\end{proof}

\begin{lemma}\label{P-1-case}  Let $c \in H^2 (\mathbb P^1_{k}, \mathbb Z/\ell^m )$ be a class.  The Artin-Verdier pairing of $c$ with $\zeta_m$  is equal to $\frac{1-q}{\ell^n}$ times the pullback of $c$ to $H^2(\mathbb P^1_{\overline{k} }, \mathbb Z/\ell^m)$ followed by the trace map on $H^2(\mathbb P^1_{\overline{k} }, \mathbb Z/\ell^m)$. \end{lemma}

\begin{proof} There is a natural map $H^2 ( \mathbb P^1_k , \mathbb Z/\ell^{\min(n,m)}) \to H^2 (\mathbb P^1_k, \mathbb Z/\ell^m)$ given by the inclusion.

Let us check that this map is an isomorphism. To do this, observe that $ H^2 (\mathbb P^1_k, \mathbb Z/\ell^m)$ is, by a spectral sequence, the $\Frob_q$-invariants in $H^2(\mathbb P^1_{\overline{k}}, \mathbb Z/\ell^m)$, and similarly with $H^2 ( \mathbb P^1/k , \mathbb Z/\ell^{\min(n,m)})$. The natural map $H^2 ( \mathbb P^1_{\overline{k}} , \mathbb Z/\ell^{\min(n,m)}) \to H^2 (\mathbb P^1_{\overline{k}}, \mathbb Z/\ell^m)$ is not necessarily an isomorphism, but it is an isomorphism on $\Frob_q$-invariants, because $\operatorname{Frob}_q$ acts by multiplication by $q$ on both groups so the $\operatorname{Frob}_q$-invariants are exactly the $(q-1)$-torsion elements.

Furthermore, we have an isomorphism $H^2 ( \mathbb P^1_k , \mathbb Z/\ell^{\min(n,m)})  \cong H^2 ( \mathbb P^1_k , \mu_{\ell^{\min(n,m)}) })$ using our chosen generator for $\mu_{\ell^n}$.

Thus, we can assume $c$ lies in the image of $H^2 ( \mathbb P^1_k , \mu_{\ell^{\min(n,m)}) })$.

Furthermore, it suffices to take $c$ a generator of this group. By Kummer theory, the Kummer class applied to a degree $1$ line bundle on $C$ is sufficient.

The trace map on $H^2 (\mathbb P^1_{\overline{k}}, \mathbb Z/\ell^m)$, composed with the projection from $H^2 ( \mathbb P^1_{\overline{k} } , \mu_{\ell^{\min(n,m)}) })$, is equal to the trace map on $H^2 (\mathbb P^1_{\overline{k}}, \mathbb Z/\ell^m)$, since the trace map is compatible with inclusions of cyclic groups. By definition \cite[XVII, (1.1.3.2) and (1.1.3.3)]{sga4-3}, the trace map of the Kummer class of a degree $1$ line bundle is $1$. Multiplied by $1$, we get $\frac{q-1}{\ell^n}$.

The Artin-Verdier pairing between the projection of this Kummer class along  $H^2 ( \mathbb P^1_{k } , \mu_{\ell^{\min(n,m)}) }) \to H^2 (\mathbb P^1_k, \mathbb Z/\ell^{\min(n,m)}) \to H^2(\mathbb P^1_k, \mathbb Z/\ell^m)$ and $\zeta_m \in H^1( \mathbb P^1_k, \mu_\ell^m)$ is the Artin-Verdier pairing of this Kummer class and the image of $\zeta_m$ along the Cartier dual maps  $H^1( \mathbb P^1_k, \mu_\ell^m) \to H^1 (\mathbb P^1_k , \mu_{\ell^{\min(n,m)}} ) \to H^1 (\mathbb P^1_k, \mathbb Z/\ell^{\min n, m})$. The image of $\zeta_m$ in $H^1 (\mathbb P^1_k , \mu_{\ell^{\min(n,m)}} ) $ is the torsor of $\ell^{\min(n,m)}$th roots of our fixed generator $\zeta\in \mu_{\ell^n}$. The action of $\Frob_q$ on this torsor is by multiplication by $\zeta^{ \frac{q-1}{ \ell^{\min (n,m)} }}$ The map $\mu_{\ell^{\min(n,m)}}\to \mathbb Z/\ell^{\min (n, m)})$ sends $\zeta^{ \ell^{n - \min(n,m) x} } $ to $x$, so the action of $\Frob_q$ on the image of $\zeta_m$ in $H^1 (\mathbb P^1_k, \mathbb Z/\ell^{\min (n, m)})$ is by adding $\frac{q-1}{ \ell^{\min(n,m)} \cdot \ell^{n - \min(n,m)}} = \frac{q-1}{\ell^n} \in \mathbb Z/\ell^n$.

It follows from \label{prop-cft-av} that the Artin-Verdier pairing of the Kummer class of a degree $1$ line bundle with this torsor is equal to the action of the Galois group element corresponding under class field theory to a degree $1$ line bundle with this torsor. Every degree $1$ line bundle is the inverse of the ideal sheaf at a point, and thus is sent by class field theory to a Galois element acting on $\overline{k}$ as $\Frob_q^{-1} $, so acting on this torsor as $\frac{q^{-1} -1}{\ell^n} \equiv \frac{1-q}{\ell^n} \mod q$, and so the Artin-Verdier pairing is $\frac{1-q}{\ell^n}$, as desired.

\end{proof}

\begin{lemma}\label{omega-comparison} For each finite field $\mathbb F_q$, prime $\ell$, and natural number $n$ such that $q \equiv 1 \mod \ell^n$ but $q\not\equiv 1\mod \ell^{n+1}$, and curve $C$ over $\mathbb F_q$, we have \[\omega_K =  \omega_C\] where $K =\mathbb F_q(C)$ is the function field of $C.$ \end{lemma}

\begin{proof} By the uniqueness statement in Lemma \ref{wedge-from-pairing}, it suffices to check that the bilinear forms $\omega_{m,K}$ defined in \S \ref{omega-pairings} is equal to  $  \omega_{C,m}=  \ell^m (a \otimes b) (\omega_{C} ) $.
%
%

By definition, the bilinear form $\omega_{m,K}$ takes classes $a,b$ to $ - \frac{1}{2} (\zeta_m, a \cup b)_{AV}$. On the other hand, the bilnear form $\omega_{m,C}$ can be obtained by pulling back to $A[\ell^m]$, equivalently, $H^1 (C_{\overline{k}}, \mathbb Z/\ell^m)$, and then taking $\frac{ 1- q}{ 2 \ell^n}$ times the Weil pairing. The Weil pairing is equivalent to the cup product $H^1 (C_{\overline{k}}, \mathbb Z/\ell^m) \times H^1 (C_{\overline{k}}, \mathbb Z/\ell^m) \to H^2 (C_{\overline{k}}, \mathbb Z/\ell^m)$ followed by the trace map on  $H^2 (C_{\overline{k}}, \mathbb Z/\ell^m)$. 
See \cite[Chapter 5, Prop 3.4]{SGA45} and \cite[V.2, Rmk 2.4(f)]{MilneEtaleCohomology}.

Because cup product is compatible with pullback from $C_k$ to $C_{\overline{k}}$, it suffices to check that the two linear forms on $H^2(C_k, \mathbb Z/\ell^m)$, the first one being the Artin-Verdier pairing with $\zeta_m$ times $\frac{-1}{2}$, and the second being the pullback to $C_{\overline{k}}$ followed by the trace times $\frac{q-1 }{2\ell^n}$, are equal. 
\bigskip

We can first handle the case $C = \mathbb P^1$, which is exactly Lemma \ref{P-1-case}.

We will now use the case of $\mathbb P^1$ to handle the general case. To do this, we fix a map $ f:C \to \mathbb P^1$ and check separately that both our linear forms are compatible with the projection to $\mathbb P^1$ in the sense that evaluating the form on a given class in  $H^2(C_k, \mathbb Z/\ell^m)$ and pushing forward the form to $H^2(\mathbb P^1_k ,\mathbb Z/\ell^m)$ and then evaluating the linear form give the same result. 

For the trace map in \'{e}tale cohomology, this compatibility with pushforward is \cite[XVIII, Lemma 1.1.5]{sga4-3}. This is stated over an algebraically closed field, so we must in addition use the fact that pushforward along $C \to \mathbb P^1$ commutes with pullback to an algebraically closed field.

For Artin-Verdier duality, recall that the pairing of $\beta\in H^2(C_k, \mathbb Z/\ell^m)$ with $\zeta_m$ proceeds in 2 steps. We first consider $\zeta_m \cup\beta\in H^3(C_k,\Z/\ell^m)$
and then apply the isomorphism $H^3(C_k,\mathbb G_m)\cong \Q/\Z$ and $H^3(C_k,\mathbb G_m)[\ell^m]\cong H^3(C_k,\Z/\ell^m)$. By the push-pull formula, it follows
that $$f_*\beta\cup \zeta_m = f_*(\beta\cup \zeta_m)$$ where we abuse notation slightly by using $\zeta_m$ to denote cohomology classes on $C$ and on $\mathbb P^1$.
It is therefore sufficient to check that the isomorphism $H^3(C_k,\mathbb G_m)\cong \Q/\Z$ commutes with pushforward. But this follows from its definition via the Brauer group
in \cite[II.2.1]{MilneArithmeticDuality}, together with the following commutative diagram:

\[
\xymatrix{
0\ar[r]&\mathbb G_{m,C}\ar[r]\ar[d]& g_*\mathbb G_{m,\eta}\ar[r]\ar[d]& \mathrm{Div}_C\ar[r]\ar[d]& 0\\
0\ar[r]&\mathbb G_{m,\mathbb P^1}\ar[r]& g_*\mathbb G_{m,\eta}\ar[r]& \mathrm{Div}_{\mathbb P^1}\ar[r]& 0\\
}
\]

where the vertical maps are the natural norm and pushforward maps respectively.

Because the linear forms are equal on $\mathbb P^1$, and preserved by projection to $\mathbb P^1$, they are equal in general.

\end{proof}

\section{Random Matrix Theory} \label{linearrandommodels}

For the entirety of this section, we fix a prime number $\ell$ and a positive integer $n$.

\subsection{The Linear Model: Motivation} \label{linearrandommodelmotivation}
If we draw our intuition from the function field setting, it is natural to want to model the distribution of $\coker(F - 1)$ where $F$ is sampled randomly from 
$\GSP^{(q)}(\ZZ_\ell)$. However,  the non-linearity of this model makes it hard to work with. Instead, we linearize in the following way: If $F \in \GSP_{2g}^{(q)}(\ZZ_\ell)$ is close to 1, then 

$$\coker(\log(F)) = \coker(\log(1 + (F-1))) = \coker(F-1).$$

It is thus plausible that $\coker(F-1)$ and $\coker(\log(F))$ are distributed in the same way.  The Lie algebra of $\Sp_{2g}(\ZZ_\ell)$ consists of skew-symplectic matrices\footnote{That is, matrices $M$ such that $\langle Mv, w\rangle= -\langle v, Mw \rangle$ for $\langle \cdot, \cdot \rangle$ the symplectic pairing.}.  Therefore, the logarithm of a general element in $\GSP_{2g}^{(q)}(\ZZ_\ell)$ should take the shape $\frac{1}{2} \log(q) + M,$ where $M$ is skew-symplectic.  Assuming $\ell$ is odd, the assumption $\ell^n || q - 1$ implies that $\frac{1}{2} \log q$ has valuation $n.$  
Therefore, the cokernel of $M + \frac{1}{2} \log q$ has the same distribution as the cokernel of $M + \ell^n.$  This motivates our linear random model:
we consider the cokernel of $M + \ell^n,$ where $M$ is randomly sampled from the Lie algebra of $\Sp_{2g}(\ZZ_\ell)$ with respect to its additive Haar measure.

\subsection{The Linear Model} \label{linearrandommodelcalculations}
Let $\ell$ be an odd prime and consider the following model. Let $\Z_{\ell}^{2g},\omega = \sum_{i=1}^g e_i \wedge f_i$ be the standard symplectic space, and $M$ be a skew-symplectic endomorphism. Let $G=G_M$ be the cokernel of $M+\ell^n.$ We further wish to endow $G_M$ with additional data so as to obtain an element of 
$\cC_{\ell,n}$: 

The first structure is the pushforward $\overline{\omega}=\omega_M\in\wedge^2 G$ of the standard symplectic form, where $\overline{\cdot}$ denotes reduction mod $(M + \ell^n)(\ZZ_\ell^{2g})$. Note that $\omega_G$ is $\ell^n$-torsion, 
since 


\begin{align*}
0 &=\sum \overline{Me_i} \wedge \overline{f_i} + \overline{e_i} \wedge \overline{Mf_i} \\
&= -2\ell^n \sum \overline{e_i} \wedge \overline{f_i}. \\
&= -2\ell^n \omega_M.
 \end{align*}

The second structure is the isomorphism $\psi_M:G^{\vee}[\ell^n]\ra G[\ell^n]$ 
stemming from the snake lemma, as in \S \ref{omegapsisymplecticsimilitude}. Explicitly, if we dualize we get an identification of $G^{\vee}$ with the kernel of $\ell^n-M$ on $(\Q_{\ell}/\Z_{\ell})^{2g}$,
and we define $\psi_G$ by sending $\alpha$ to $(M+\ell^n)\alpha'$, where $\alpha'$ is any lift of $\alpha$ to $\Q_{\ell}^{2g}$. 
This provides us with a bilinear form on $G^{\vee}[\ell^n]\times G^{\vee}$
given by $\langle \alpha,\beta\rangle_G:= \beta(\psi(\alpha))$ and like in Lemma \ref{compatibilitypsiomega}, the triple $(G_M,\omega_M,\psi_M)$ satisfies the compatibility relation \eqref{psiomegacomp}.

\begin{definition}\label{linearmeasuredef}
We define the linear measure $\mu_g$ on $\cC_{\ell,n}$ as the pushforward of the Haar measure on skew-symplectic matrices $M$ under the map
$M\ra (G_M,\omega_M,\psi_M)$.

\end{definition}

We shall show that  as $g\ra\infty$ the $\mu_g$  converge to a natural probability measure $\mu$, and we shall begin by computing its moments.

\subsection{Moments of $\mu_g$} First, we need a couple preliminary lemmas:

\begin{lemma}\label{allsur}

Fix a finite abelian group $G$. Fix $\omega_G \in \wedge^2 G$ satisfying $\ell^n \omega_G = 0.$  If $g$ is large enough then there exists a surjection $f :\Z_\ell^{2g}\ra G$ such that $f \omega = \omega_G$, and for any such surjection $f$ there exists a skew-symplectic $M$ with $f\circ(M+\ell^n)=0$. 

\end{lemma}

\begin{proof}

By Witt's extension theorem the set of all $f$ satisfying $f\omega=\omega_G$ forms a single orbit under $\Sp_{2g}\Z_\ell$. For the second claim, it is therefore enough to find $M$ satisfying the conditions of the Lemma relative to a single $f$ for which $f\omega=\omega_G$. For every $k>0,$ any $2\times 2$ anti-diagonal matrices whose off-diagonal entries are $- x,2\ell^n+x$ with valuations $k,n$ resp. defines a surjection  
$f:\Z_\ell^2\ra H_k= \Z/\ell^n \oplus \Z/\ell^r$ for which $f\omega = \omega_k$ generates $\wedge^2H_k[\ell^n]$. An arbitrary $\omega_G$ can be obtained by pushing forward 
$\oplus \omega_k$ under a surjective map $\oplus_k H_k \ra G$. This verifies the first claim, and we can now define $M$ for the second claim. Let $M$ be the direct sum of the transformations defined by
\begin{align*}
(M + \ell^n)e_1 &= (2\ell^n + x) e_1 \\
(M + \ell^n)e_2 &= -x \cdot e_2.
\end{align*}
Then $M$ has trace 0 and so is skew-symplectic with integral entries and $\im(M + \ell^n) = \ker f.$  The claim follows.

\end{proof}

\begin{lemma}\label{eqdist}

Fix $G,\omega_G\in\wedge^2G$. As $g \to \infty,$ the proportion of surjections $f:\Z_\ell^{2g}\ra G$ for which $f\omega = \omega_G$ approaches $\frac{1}{\mid\wedge^2G\mid}.$  

\end{lemma}

\begin{proof}

First, note that since almost all homomorphisms are surjections, we can and do use a random homomorphism $f$ instead of a random surjection. Let $\nu_g$ be the pushforward
measure on $\wedge^2G$ thus obtained for a given $g$. Note that $\nu_g$ is the $g\,$th convolution of $\nu_1$. Since $\nu_g$ has full support by Lemma \ref{allsur}, the
claim follows.  
\end{proof}

We now compute the asymptotics of the $\mu_g$ and their moments.

\begin{thm}\label{totmommeas}

Fix $(G,\omega_G,\psi_G)\in\cC_{\ell,n}$. As $g\ra\infty$, we have $$\E_{\mu_g}\mid\sur(*,(G,\omega_G,\psi_G))\mid \ra \frac{1}{\left| \sym^2G[\ell^n] \right|}.$$ 

Moreover, if $\psi_G$ is an isomorphism, then  $$\mu_g((G,\omega_G,\psi_G)) \ra \frac{c_\ell}{\left| \Aut(G,\omega_G,\psi_G) \right| \cdot \left| \sym^2G[\ell^n] \right|}$$ where $c_\ell=\prod_{i=0}^{\infty} (1-\ell^{-(2i+1)}).$ 
And if $\psi_G$ is not an isomorphism, then
$$\mu_g((G,\omega_G,\psi_G)) \ra 0.$$

\end{thm}

\begin{proof}

Fix a surjection $f$ with $f\omega = \omega_G$. By Lemma \ref{eqdist}, there are  $(1-o_g(1))\cdot\frac{|G|^{2g}}{\#\wedge^2G}$ of these.  By Lemma \ref{allsur}, for large enough $g$ there is at least one skew-symplectic $M$ satisfying $f \circ (M + \ell^n) = 0.$  The set of all such $M$ form a coset for the group of those skew-symplectic matrices $N$ with $f\circ N=0$. 

Identifying $L=\Z_\ell^{2g}$ with $L^\ast$ via the symplectic pairing, $M$ can be viewed as a self-dual map from
$L^\ast$ to $L.$  Thus we are looking for the measure of symmetric matrices with image in the kernel of $f$. 
Write $$G=\oplus_{i=1}^r \left( \Z/\ell^{m_i}\right) \cdot b_i.$$ Pick a basis $e_i$ for $L$ so that $f(e_i)=b_i$ with $b_i=0$ for $i>r$ by convention
Let $b_i^\vee$ denote the dual basis for $G^\vee.$  Let $A$ denote the matrix of $M+\ell^n$ with respect to the bases $e_i$ 
and $e_i^*$, so that $Ae_i^*=\sum_j A_{i,j}e_j$. Then the condition that $f\circ (M+\ell^n)=0$ is equivalent to $A_{i,j}$ being divisible by $\ell^{m_j}$. Moreover, if we let
$r_i:=\max(0,m_i-n)$ then $\ell^{r_i}b_i$ is a basis for $G[\ell^n]$ and the the map $\psi_G$ is given by
$$\psi_G(\ell^{r_i}b_i^\vee)=\sum_j \frac{A_{i,j}}{\ell^{m_i-r_i+r_j}}\ell^{r_j}b_j.$$ 

Now, changing $M$ amounts to adding a symmetric matrix $B$ to $A$. Clearly, $B$ must satisfy $\ell^{\max(m_i,m_j)} \mid B_{i,j}$. The Haar measure
of all such $B$ is easily computed to be $\frac{\mid\wedge^2G\mid}{|G|^{2g}}$. To finish the proof of the first part of the lemma, we must show that any $\psi_G$ that is compatible
with $\omega_G$ can occur as above with an appropriate choice of $A'$. Now for any two such $\psi^1_G,\psi^2_G$ consider the difference $\psi^3_G=\psi^1_G-\psi^2_G$. 
We may write $$\psi^3_G(\ell^{r_i}b_i^\vee)=\sum_j \frac{a_{i,j}}{\ell^{m_i-r_i+r_j}}\ell^{r_j}b_j.$$  The element $a_{i,j}$ is well defined modulo 
$\ell^{\min(m_j,n)+m_i+r_j-r_i} = \ell^{m_j+\min(m_i,n)}$.
 Applying the compatibility relation \ref{psiomegacomp}  with $b_i^\vee,b_j^\vee$ and $r=\max(r_i,r_j)$ implies that $a_{i,j}$ and $a_{j,i}$ are equal modulo $\ell^{\min(m_i+m_j,n+m_i,n+m_j)}$. Since $$\min(m_j+\min(m_i,n), m_i+\min(m_j,n)) = \min(m_i+m_j,n+m_i,n+m_j)$$ it follows that we may pick a single integer which represents
$a_{i,j}$ and $a_{j,i}$ simultaneously. Setting $A'_{i,j}$ to be this integer completes the proof.

\medskip

Now, for the second part of the lemma, note that we are now looking for matrices $M$ such that the image of $M+\ell^n$ is equal to the kernel of $f$. This is equivalent to the two conditions
\begin{itemize}
\item
$\ell^{m_j} | A_{i,j},$ the condition from earlier which guarantees that $\im(M + \ell^n) \subset \ker f$
\item 
$\ell^{\sum m_i}\mid\mid\det A,$ which combined with the above bullet point implies that $\im(M + \ell^n) = \ker f.$ 
\end{itemize}
After taking out a factor of $\ell^{m_j}$ from the $j \,$th column in $A$, we see that the resulting matrix is block upper-triangular, consisting of an
$r\times r$ matrix $C$ and a $2g-r\times 2g-r$ matrix $D$, and we are looking for the probability that both of these are invertible. 

Since $k\geq 1$, the reduction of $D$ mod $\ell$ is just a symmetric matrix, which is invertible with probability tending to $c_\ell$(this is the $t=0$ case of Lemma \ref{symmetricprob}). 

We claim that $C$ is invertible iff $\psi$ is invertible, which would complete the proof. Note that $C$ is block upper triangular, as is $\psi$, so to check invertibility we only have to restrict
to the blocks where $m_i$ is fixed. The claim is now immediate, as the matrices for $C$ and $\psi$ are both the appropriate submatrices of $\left(\frac{A_{i,j}}{\ell^{m_i}}\right)_{i,j}$.

\end{proof}

We shall also need a uniform upper bound for the intermediate measures $\mu_g$, so as to apply Fatou's Lemma when we study the limiting measure

\begin{lemma}\label{upperbound}

There exists an absolute constant $c$ such that 
$$\sum_{\omega_G\in \wedge^2G[\ell^n]}\#\sur\left((\Z_\ell^{2g},\omega),(G,\omega_G)\right)\leq c|G|^{2g}\cdot \frac{|\wedge^2G[\ell^n]|}{|\wedge^2 G|}.$$
\end{lemma}

\begin{proof}

First, we claim that for such a surjection to exist, the group $G'=\ell^nG$ must have $\ell$-rank at most $g$. To see this, note that since $\omega_G$ is $\ell^n$-torsion it pushes
forward to 0 under the natural surjection from $G$ to $G'$. (This implication can be checked using a basis.) Since $\Z_\ell^{2g}$ then surjects onto $G'$ and maps $\omega$ to $0$, dualizing we see that
 $G'^{\vee}[\ell]$ embeds into $\F_p^{2g}$ as an isotropic subspace and thus has rank at most $g$. The claim follows. We may thus assume that $\ell^nG$ has rank at
 most $g$ and is thus a quotient of $\Z_\ell^g$.
 
 Now, we prove the lemma by using Fourier analysis. We in fact bound the total number of homomorphisms $f:\Z_\ell^{2g}\rightarrow G$ which take $\omega$ to $\omega_G$.  In this proof, let $H^\vee$ denote $\Hom(H,\mathbb{S}^1),$ where $\mathbb{S}^1 := \{z \in \C: |z| = 1 \}.$  We compute 

 \begin{align*}
 \sum_{f:\Z_\ell^{2g}\rightarrow G} \delta_{f_*\omega\in\wedge^2G[\ell^n]}&=\sum_{f:\Z_\ell^{2g}\rightarrow G} \E_{\chi\in\ell^n(\wedge^2G)^\vee} \left( \chi(f_*\omega)\right)\\
 &=\E_{\chi\in\ell^n(\wedge^2G)^\vee}\sum_{f:\Z_\ell^{2g}\rightarrow G}\chi(f_*\omega)\\
 &=\E_{\chi\in\ell^n(\wedge^2G)^\vee}\left(\sum_{f:\Z_\ell^{2}\rightarrow G}\chi(f_*\omega)\right)^g\\
 \end{align*}
 
 Now each $\chi\in \left(\wedge^2 G \right)^\vee$ is naturally an alternating bilinear form on $G$ valued in $\mathbb{S}^1,$ so we denote $\ker\chi$ as those elements of $G$ that pair to $0$ with every
 other element of $G$. It is then easy to see that $\sum_{f:\Z_\ell^{2}\rightarrow G}\chi(f_*\omega) = |\ker\chi|\cdot |G|$ so that the above sum becomes

$$ \frac{|G|^{2g}\cdot \left|\wedge^2G[\ell^n]\right|}{|\wedge^2G|}\sum_{\chi\in\ell^n(\wedge^2G)^\vee} \frac{1}{[G:\ker\chi]^g}.$$

Now,  $\chi\in\ell^n(\wedge^2G)^\vee$ is equivalent to  $\ker\chi \supset G[\ell^n].$  Let $G' = G / G[\ell^n].$  There is a surjection 
\begin{align*}
&\{ (f,\omega): f: G' \twoheadrightarrow G'/N \text{ surjective with } \ker f = N, \omega \text{ non-degenerate, alternating on } G/N \} \\
&\twoheadrightarrow \{ \text{alternating bilinear forms on } G' \text{ with kernel } N \}\\ 
\end{align*}
defined by
$$(f,\omega) \mapsto \left[ (g,g') \mapsto \omega(f(g),f(g')) \right].$$
Every fiber of the above mapping consists of a single $\mathrm{Aut}(G'/N)$ orbit.  The number of such orbits is at most 
$$\frac{\# \mathrm{Surj}(G',G'/N) \cdot |\wedge^2 G'/N|}{\mathrm{Aut}(G'/N)},$$
because $\mathrm{Aut}(G'/N)$ acts freely on the left side.  Thus, 
\begin{align*}
\sum_{\chi\in\ell^n(\wedge^2G)^\vee} \frac{1}{[G:\ker\chi]^g} &\leq \sum_{N \subset G'} \frac{\#\mathrm{Surj}(G',G'/N) \cdot | \wedge^2 G'/N | }{\#\Aut \left( G'/N \right)\cdot |G'/N|^g}\\
&\leq \sum_H \frac{|\wedge^2\!\!H|}{\#\Aut H}\\
\end{align*}

where $H$ varies over all finite abelian $\ell$-groups and the last inequality follows since $$\#\sur\left({G/G[\ell^n],H}\right)\leq \#\sur\left(\Z_\ell^g,H\right) = |H|^g.$$

It thus remains to show that $ \sum_H \frac{|\wedge^2\!H|}{\#\Aut H}$ is finite. This is an easy calculation with partitions, which we carry out in Lemma \ref{part}.

 \end{proof}
 
 \begin{lemma}\label{part}
 
  $ \sum_H \frac{|\wedge^2\!H|}{\#\Aut H}$ is finite, where the sum is over all finite abelian $\ell$-groups $H$.
 
 \end{lemma}
 
 \begin{proof}
 
 We may parametrize $H$ with sequences $(a_i)_{i\in\N}$ of non-negative integers only finitely many of which are non-zero, identifying a sequence with
 $\oplus_i (\Z/\ell^i)^{a_i}$. We denote by $n_H$ the maximum integer such that $a_n>0$. Let $c=\prod_{i>0}(1-\ell^{-1})$.
 
 Then $$|\wedge^2\!H| = \ell^{\sum_{i<j}ia_ia_j+\sum_i ia_i(a_i-1)/2}=\ell^{\sum_{i\leq j} ia_ia_j- \sum_i ia_i(a_i+1)/2}$$ and
 $$\#\Aut H= \ell^{\sum_{i\leq j} iaia_j} \prod_i\prod_{1\leq k\leq n_H}(1-\ell^{-k})^{-1}\geq  \ell^{\sum_{i\leq j} iaia_j} c^n$$
 
 and so 
 
 \begin{align*}
  \sum_H \frac{|\wedge^2\!H|}{\#\Aut H}&\leq \sum_H c^{-n_H}\ell^{\sum_i ia_i(a_i+1)/2}\\
  &\leq \sum_{n_H,a}c^{-n_H-1}\ell^{-n_Ha(a+1)/2}\\
  &\leq 2c^{-1}\sum_{n_H}(c\ell)^{-n_H}\\
  &\leq\frac{2c^{-1}}{1-c^{-1}\ell^{-1}}\\  
  \end{align*}

 \end{proof}
 
 \subsection{The universal measure $\mu$} \label{universalmeasuredefinition}
 
 \begin{thm}\label{univmeas}
 
 As $g\ra\infty$ the measures $\mu_g$ weak-* converge to  a probability measure $\mu$, with the moments from Theorem \ref{totmommeas}.  
 
 \end{thm}

\begin{proof}

First, note that by Theorem \ref{totmommeas} $\mu(G) = \frac{c_\ell\cdot \mid\wedge^2G[\ell^n]\mid h_G}{\Aut(G)}$ where $h_G$ is the fraction of pairs $\omega_G,\psi_G$ such that $\psi_G$ is invertible. By writing
$G=\oplus_{j=1}^n (\Z/\ell^j)^{r_j(G)}$ it is easy to see that 
$$h_G = \prod_{j=1}^{\lfloor n/2\rfloor} (\ell^{-1};\ell^{-1})_{r_j(G[\ell^n])} \prod_{j=\lfloor n/2\rfloor +1}^{n} (\ell^{-1};\ell^{-2})_{\lceil r_j(G[\ell^n])/2\rceil}.$$

In particular, $\mu(G) \asymp \frac{\mid \wedge^2G[\ell^n]\mid }{\Aut(G)}$.

Now, for any group $G$, it follows from Theorem \ref{totmommeas} together with the calculations in Lemma \ref{upperbound} that

$$\E_{\mu_g}\#\sur\left(*,G\right)\leq c\cdot |\wedge^2G[\ell^n]|,$$

so that in particular it follows that \begin{equation}\label{eq-star} \mu_g(G)\ll \mu(G).\end{equation}  By Fatou's Lemma, it follows that $\mu$ is a probability measure.

It remains to show that $\mu$ has the predicted moments. By Fatou's Lemma again, it follows from \eqref{eq-star} that for any group $H$, we have
$\lim_g\E_{\mu_g}\#\sur\left(*,H\right) = \E_{\mu}\#\sur\left(*,H)\right)$. This implies
\begin{align*}
\limsup_g\sum_{\omega_H,\psi_H}\E_{\mu_g}\#\sur\left(*,(H,\omega_H,\psi_H)\right)&= \limsup_g\E_{\mu_g}\#\sur\left(*,H\right)\\
&=\E_{\mu} \#\sur\left(*,H\right)\\
&=\sum_{\omega_H,\psi_H}\E_{\mu}\#\sur\left(*,(H,\omega_H,\psi_H)\right)\\
\end{align*}

On the other hand, by Fatou's Lemma, for each pair $(\omega_H,\psi_H)$ we have 

$$\liminf_g\E_{\mu_g}\#\sur\left(*,(H,\omega_H,\psi_H)\right)\leq \E_{\mu}\#\sur\left(*,(H,\omega_H,\psi_H)\right).$$

Since $\liminf \leq \limsup$ we must have equality, and so the claim follows.

\end{proof}

\begin{lemma}\label{momimpmeas}

The measure $\mu$ is determined by its moments.
\end{lemma}

\begin{proof}

Suppose $\mu'$ is another measure with the same moments. Note that we immediately get the inequality 
$$\mu'(G,\omega_G,\psi_G)\leq \frac{1}{\left| \Aut(G,\omega_G,\psi_G) \right| \cdot |\sym^2G[\ell^n]|}.$$

Let $V$ be the vector space of functions on triples $(G,\omega_G,\psi_G)$ where $\psi_G$ is invertible. Write $U$ for the 
column vector whose components  are $\mu'(G,\omega_G,\psi_G)$, and let $M$ be the matrix whose components are $$\#\sur\left((G,\omega_G,\psi_G),(H,\omega_G,\psi_H)\right).$$ 
Then $MU=T$ by assumption, where $T$ has components $|\sym^2G[\ell^n]|^{-1}$.  Now let $D$ be the diagonal matrix with entries $\#\Aut(G,\omega_G,\psi_G)|\sym^2G[\ell^n]|$ 
Then $MD^{-1} \cdot DU = T$. Now $DU$ and $T$ are both in $L^\infty(V)$, and the rows of $MD^{-1}-I$ have sum $c_{\ell}^{-1}-1<1$ so $I-MD^{-1}$ has operator norm less than 1 and
$MD^{-1}$ is therefore invertible as an operator on $L^\infty(V)$. Thus, $DU= (MD^{-1})^{-1} T $ uniquely determines $DU$, and thus $U$.  

\end{proof}

\subsection{The non-linear model}

We define a non-linear model along the lines of \cite{LT}, but taking the pairings $\psi_G$ into account. We then use our work on the linear model above to prove that the non-linear model converges to the same measure $\mu$, resolving in particular \cite[Conjecture 3.1]{LT}.

So, let $\nu^n_g$ be defined as follows. Take $q\in\Z_{\ell}$ to be any element such that $\ell^n|| q-1$. Now we define the measure $\nu^n_g$ to be the pushforward of the Haar
measure on $F\in \GSP_{2g}^{(q)}$ to our category $\mathcal{C}_n$ under the map $F\ra \coker(1-F)$. 

We begin by showing show that $\nu^n_g$ has the ``right"  moments. First, we recall the following

\begin{lemma}\cite[Theorem 3.1]{LT}\label{nonlinearbound}

Fix $H^\cdot=(H,\omega_H) \in(\wedge^2 H)[\ell^n]$. If we forget the $\psi_G$ factor, the $\nu^n_g$ -expected number of surjections $G^\cdot=(G,\omega_G)$ to any lift of $H^\cdot$ 
is equal to $0$ for $g\leq g(H)$ and $1$  for $g>g(H)$.

\end{lemma}

We now need to know that when we throw in the $\psi_H$ then we eventually get the right moments, for large enough $g$.

\begin{lemma}\label{nonlinearmoments}

Fix $H^\cdot=(H,\omega_H,\psi_H)\in\mathcal{C}_g$. The $\nu^n_g$ -expected number of surjections $G^\cdot=(G,\omega_G,\psi_G)\ra H^\cdot$ is equal to $1/|\sym^2 H[\ell^n]|$ for 
$g\gg_H 1$.

\end{lemma}

\begin{proof}

We follow closely the proof of \cite[Theorem 3.1]{LT}. Following that proof, for $g>g(H)$ we fix a surjection $f:\Z_\ell^{2g}\ra H$ such that $f_*\omega = \omega_H$. Let $V:=\ker f$. 
Then the set of all surjections  onto $H$ pushing $\omega$ to $\omega_H$ forms a single orbit under pre-composition by $\Sp_{2g}(\Z_\ell)$. Let $S_f$ denote the set of elements 
$F\in\GSP^{(q)}_{2g}(\Z_\ell)$ such that $\im(1-F)\subset \ker f$, or equivalently $f=f\circ F$. Note that $S_f$ is a left torsor for 
$\stab(f)\subset \Sp_{2g}(\Z_\ell)$. We fix $F_0\in S_f$.
It will be useful for us also to recall that $s\in \stab(f)$ iff $(1-s)V^*=0$. Now we consider what happens to the $\psi$-pairing under an element $F_0s$. 

Dualizing, we have $$f^\vee: H^\vee\hookrightarrow (\Q_\ell/\Z_\ell)^{2g}$$ and $$\psi_{F_0s}(h^\vee)=f\circ(1-F_0s)(\ol{f^\vee(h^\vee)})$$ where $\ol{\cdot}$ denotes any lift to $\Q_\ell^g$. 
As $(1-sF_0) = (1-F_0) + (1-s)F_0$, so we see that 
$$(\psi_{sF_0}-\psi_{F_0})(h^\vee) = f\left((1-s)(F_0(\ol{f^{\vee}(h^\vee)})\right).$$
 
 Next, note that $$f^{\vee}(H^\vee[\ell^n])=\ker(q-F_0)\mid_{(\Q_\ell/\Z_\ell)^{2g}[\ell^n] }= V^*[\ell^n],$$ which is therefore killed by $(1-s)$.
 Letting $v=F_0(\ol{f^{\vee}(h^\vee)})$, it follows that $f((1-s_1s_2)(v)) = f((1-s_1))(v) +f\circ s_1( (1-s_2)v) = f((1-s_1)v)+f((1-s_2)v)$, and thus we see that the map 
\begin{align*}
R:\stab(f) &\ra \Hom(H^\vee[\ell^n],H[\ell^n]) \\ 
s &\mapsto \psi_{sF_0}-\psi_{F_0}
\end{align*}
is a group homomorphism. The claim will thus follow if we show that
for large enough $g$, the image of $R$ is of size $|\sym^2 H[\ell^n]|$.

Note that the map $R$ can be thought of as the composition of $R_0:\stab(f)\ra \Hom(V^*,\Z_{\ell}^{2g}/V)^\dagger$ given by $s\ra R_0(s)(v^*)=(1-s)v^*$ with the restriction to
$V^*[\ell^n]$. Here $^{\dagger}$ denotes the self-dual maps.  It is thus sufficient to prove that for large enough $g$,  $R_0$ is surjective.

Now, fix an element $\rho\in \Hom(V^*,\Z_{\ell}^{2g}/V)^{\dagger}$.   We claim that there exists a skew-symplectic element $M$ which is a multiple of $\ell$, such that $MV^*=0$
and moreover $M$ induces the map $\rho$. In fact, this follows from the proof of Theorem \ref{totmommeas}. Now we set $s=e^M$. Note that since $M\Z_\ell^{2g}\subset V$ it follows that
$s$ induces the same map as $M$, which completes the proof.

\end{proof}

\begin{thm} 

The measures $\nu^n_g$ converge to $\mu_g$ as $n\rightarrow\infty$.

\end{thm}

\begin{proof}

Let $\nu$ be in the weak-* closure of $\nu^n_g$. By Lemma $\ref{nonlinearbound}$ it follows that $\nu^n_g(G)\ll \mu_g(G)$ and so it follows that the moments of $\nu$ are the limits
of the moments of the $\nu^n_g$, which are the same as the moments of $\mu_g$ by Lemma \ref{nonlinearmoments}. The claim follows by
Lemma \ref{momimpmeas}.

\end{proof}

\subsection{ Quotienting out by an additional $t$ elements}~\\ \label{quotienting}
\medskip

Recall that in Definition \ref{Qtmu} we defined the operation $Q$ that takes a measure on $\mathcal C_{\ell,n}$ to another measure on $\mathcal C_{\ell,n}$ obtained by quotienting each group by a random element, and defined the measure $Q^t \mu$ by iteratively applying $Q$ to $\mu$. In this section, we study $Q^t \mu$.

%
%
%
%

We have the following general lemma:

\begin{lemma} \label{Qmom}
If $\mu$ is a measure with finite moments, then so is $Q\mu$ and $$\E_{Q^t\mu}\#\sur(*,(G,\omega_G,\psi_G)) = |G|^{-t} \E_{\mu}\#\sur(*,(G,\omega_G,\psi_G)).$$
\end{lemma}

\begin{proof}

It suffices to handle the case where $t=1$. Note $Q\mu$ is the quotient of a $\mu$-random group $H$ by the image of a random map $f$. So a a surjection from a $Q\mu$-random
group $(H,\omega_H,\psi_H)$ to $(G,\omega_G,\psi_G)$ is the same as a surjection $\psi$ from a $\mu$-random group $(H,\omega_H,\psi_G)$ together with a map 
$f:\Z_{\ell}\rightarrow\ker\phi$. Noting such a kernel has index $|G|$ we get 

\begin{align*}
&\E_{Q\mu}\#\sur(*,(G,\omega_G,\psi_G))\\
&=  \sum_{(H,\omega_H,\psi_H)} \mu(H,\omega_H,\psi_H)\frac{\#\{\phi:(H,\omega_H,\psi_H)\twoheadrightarrow(G,\omega_G,\psi_G),f:\Z_{\ell}\rightarrow\ker\phi\}}{|H|}\\
&= \sum_{(H,\omega_H,\psi_H)} \mu(H,\omega_H,\psi_H)\frac{\#\sur\left((H,\omega_H,\psi_H)\twoheadrightarrow(G,\omega_G,\psi_G)\right)}{|G|}\\
&=|G|^{-1} \E_{\mu}\#\sur(*,(G,\omega_G,\psi_G))
\end{align*}

as desired.

\end{proof}

Since the moments of $Q^\mu$ are smaller then the moments of $\mu$, the same proof for Lemma \ref{momimpmeas} applies in this setting and gives:

\begin{lemma}\label{Qmomimpmeas}

The $Q^t\mu$ is determined by its moments.

\end{lemma}

We now proceed to compute $Q^t\mu$. Note that while $\mu$ is supposed only where $\psi_G$ is an isomorphism, $Q^t\mu$ could potentially be supported where $\coker\psi_G$ has $\ell$-rank at most $t$.

\begin{thm}\label{quotient-measure-formula}

If $(G,\omega_G,\psi_G)$ is such that the image of $\psi_G(G^\vee[\ell])$ in $G[\ell]$ has codimension $s\leq t$, then

$$Q^t\mu(G,\omega_G,\psi_G) = \frac{1}{|\Aut(G,\omega_G,\psi_G)|\cdot |\sym^2G[\ell^n]|\cdot|G|^{t}}\frac{(\ell^{-1})_t}{(\ell^{-1})_{t-s}}\prod_{i=t+1}^{\infty}(1+\ell^{-i})^{-1}$$

\end{thm}

\begin{proof}

Recall that we are looking for the measure of pairs of maps $M:\Z_\ell^{2g}\rightarrow \Z_{\ell}^{2g}, B:\Z_{\ell}^t\rightarrow \Z_{\ell}^{2g}$ so that 
M is skew-symplectic, the cokernel of $(M+\ell^n)\oplus B$ is isomorphic to $G$ and the pushforward of $\psi$ under $M+\ell^n$ is $\psi_G$. 

Fix a surjection $f$ with $f\omega = \omega_G$. By Lemma \ref{eqdist}, there are  $(1-o_g(1))\cdot\frac{|G|^{2g}}{\#\wedge^2G}$ of these.  By Lemma \ref{allsur}, for large enough $g$
there is at least one $M$ satisfying $f \circ (M + \ell^n) = 0.$  The set of all such $M$ form a coset of those matrices $N$ with $f\circ N=0$. 

Identifying $L=\Z_\ell^{2g}$ with $L^\ast$ via the symplectic pairing, $M$ can be viewed as a self-dual map from
$L^\ast$ to $L.$ Upon identifying $L$ with $L^\ast$ via the symplectic form, we identify $\ell^n \cdot 1$ with $\ell^n \cdot$ the identification map.  To ease notation, we continue to refer to this map as $\ell^n$ below.   
Write $$G=\oplus_{i=1}^r \Z/\ell^{m_i} b_i.$$ Pick a basis $e_i$ for $L$ so that $f(e_i)=b_i$ with $b_i=0$ for $i>r$ by convention.   
Let $b_i^\vee$ denote the dual basis for $G^\vee.$  Let $A$ denote the matrix of $M+\ell^n$ with respect to the bases $e_i$ 
and $e_i^*$, so that $Ae_i^*=\sum_j A_{i,j}e_j$, and l. Then the condition that $f\circ (M+\ell^n)=0$ is the condition that $A_{i,j}$ is divisible by $\ell^{m_j}$. Moreover, if we let
$r_i:=\max(0,m_i-n)$ then $\ell^{r_i}b_i$ is a basis for $G[\ell^n]$ and the the map $\psi$ is given by
$$\psi(\ell^{r_i}b_i^\vee)=\sum_j \frac{A_{i,j}}{\ell^{m_i-r_i+r_j}}\ell^{r_j}b_j.$$ 

Now, changing $M$ amounts to adding a symmetric matrix $A'$ to $A$. Clearly, $A'$ must satisfy $\ell^{\max(m_i,m_j)} \mid A'_{i,j}$. The Haar measure
of all such $A'$ is easily computed to be $\frac{\mid\wedge^2G\mid}{|G|^{2g}}$. Moreover, exactly as in the proof of theorem \ref{totmommeas} we can evidently make any
$\psi_G$ which is compatible with $\omega_G$ occur by picking an appropriate $A'$. All such $\psi_G$ occur with equal measure as they are distinct cosets of allowable matrices $A'$. 

Let $C$ be the $2g\times 2g+t$ matrix of $(A+\ell^n,B)$ in the bases $e_i,e_i^\vee$ and an arbitrary basis for the domain $\Z_\ell^t$ of $B$. We need 
the Haar measure of all such $C$ which are surjective onto the kernel of $f$, and which induces the pairing $\psi_G$. The set of those $C$ which map into
the kerne of $f$ and induce $\psi_G$ has Haar meaure $\frac{|\wedge^2G|}{|\sym^2G[\ell^n]|\cdot|G|^{2g+t}}$  by the above. We restrict
to such $C$ from now on. 

Let $C'$ be the matrix $C$ with the $i$'th row divided by $\ell^{m_i}$ for $1\leq i\leq r$, and let $C_0$ be the reduction of $C'$ modulo $\ell$. Note that $C_0$ is of the shape $$\begin{pmatrix}\psi_* & D & B_1 \\ 0 & X=X^t & B_2\end{pmatrix}$$ with $\psi_*$ being determined by $\psi_G$, and $D,B_1,X,B_2$ being Haar-random with the only condtion that $X$ is symmetric.

 For $C$ to map surjectively onto $\ker f$ is equivalent to $C_0$ being surjective, or equivalently for the rank of $C_0$ to be $2g$.  This in particular requires $\begin{pmatrix} X B_2 \end{pmatrix}$ to be surjective. The lemma below follows an unpublished note of Robert Rhoades where he computes the number of symetric matrices having a fixed rank over a field.

\begin{lemma}\label{symmetricprob}

Let $E=\begin{pmatrix} X B_2 \end{pmatrix}$ be a random $m\times m+t$ matrix over $\F_{\ell}$ with $X=X^t$. Then as $m\rightarrow\infty$, the probability that $E$ is surjective tends to $\prod_{i=t+1}^{\infty}(1+\ell^{-i})^{-1}$.

\end{lemma}

\begin{proof}

Define $(x)_j:= \prod_{i=1}^j (1-x^i)$, and ${a\choose b}_x:=\frac{(x)_a}{(x)_b(x)_{a-b}}$. 
Let $c=\prod_{i=1}^{\infty} (1+\ell^i)^{-1}.$ Let $I(n,j)$ be the number of symmetric $n\times n$ matrices which have corank $j$. We claim that $I(n,j)=I(n-j,0)\ell^{j(n-j)}\cdot {n\choose j}_{\ell^{-1}}$.
To see this, note that a an $n\times n$ self-dual map $\phi:L\rightarrow L^\ast$ of rank $n-j$ has kernel a $j$-dimensional subspace $E$, and induces a self-dual map on  the quotient space $L/E$. The number of $j$-dimensional subspaces is $\ell^{j(n-j)}\cdot {n\choose j}_{\ell^{-1}}$ and so the claim follows.
Letting $n\rightarrow\infty$ we see that the probability of a large symmetric matrix having corank $j$ is $\frac{c\ell^{-\frac{j^2+j}{2}}}{(\ell^{-1})_j}$. 

Now, if $X$ has corank $j\leq t$ the probability that $B_2$ surjects from $\F_\ell^t $ onto the cokernel of $X$ is the probability that a random $t\times j$ matrix is surjective which is $\prod_{i=t-j+1}^{t} (1-\ell^{-i}) = \frac{(\ell^{-1})_t}{(\ell^{-1})_{t-j}}$.

Thus, the probability that $E$ is surjective tends to 
$$c\sum_{j=0}^t \frac{\ell^{-\frac{j^2+j}{2}}(\ell^{-1})_t}{(\ell^{-1})_j(\ell^{-1})_{t-j}} = c\sum_{j=0}^t \ell^{-\frac{j^2+j}{2}}{t\choose j}_{\ell^{-1}}=c\prod_{i=0}^{t-1} (1+\ell^{i+1})$$ by the q-binomial theorem. This completes the proof.

\end{proof}

Assuming that $\begin{pmatrix} 0 &  X & B_2 \end{pmatrix}$ is surjective, let $K$ denote its kernel. Then write $F$ for the $r\times t$ matrix representing
the restriction of $\begin{pmatrix} D & B_1 \end{pmatrix}$ to $K$. Then for $C'$ to be surjective it is equivalent for the matrix 
$\begin{pmatrix} \psi_* & F \end{pmatrix}$ to be surjective. 

\begin{lemma}

The corank of $\psi_*$ is equal to the codimension of $\psi_G(G^\vee[\ell])$ in $G[\ell]$.

\end{lemma}

\begin{proof}

Notice that the matrix $\psi_*$ in the bases $b_i^{\vee}, b_i$ represents the reduction of the map $\psi_G^{\vee}:G^\vee/\ell^n\rightarrow G/\ell^n$. Thus its corank is equal to the dimension of $G/\ell\im\psi_G^{\vee}$. Dualizing back gives the result.
\end{proof}

Given the lemma, the only remaining condition is for $F$ to be surjective onto the cokernel of $\psi_*$. The probability that a random map from 
$\F_{\ell}^t$ to $\F_{\ell}^s$ is surjective is $\frac{(\ell^{-1})_t}{(\ell^{-1})_{t-s}}$, which completes the proof.
\end{proof}

\subsection{Proof of the stability theorem}

\begin{thm}[Stability of $\mu_{n,u}$]\label{CLStability}
Fix $n \in \mathbb{Z}_{> 0}$ and $u \in \mathbb{Z}_{\geq 0}.$  Fix $S \subset | \mathcal{C}_n |$ a finite subset.  Fix $\epsilon > 0.$  The measure $\mu_{n,u}$ enjoys the following stabiilty property: there is some $\delta = \delta(S,\epsilon) > 0$ and $T = T(S,\epsilon)$ satisfying: 
\begin{quote}
If $\mu$ is any probability measure on $|\mathcal{C}_n|$ satisfying 
$$\left| \mathbb{E}_{\mu} \left(\# \mathrm{Surj}(\bullet, \Gamma) \right) - \mathbb{E}_{\mu_{n,u}} \left(\# \mathrm{Surj}(\bullet, \Gamma) \right) \right| < \delta \text{ for all } \Gamma \in T,$$
then
$$\left| \mu(\Gamma) - \mu_{n,u}(\Gamma) \right| < \epsilon \text{ for all } \Gamma \in S.$$
\end{quote}
\end{thm}

\begin{proof}

The statement is equivalent to the following: 

\begin{quote}\label{limmeaus}
If $\nu_i$ is a sequence of measures such that for all $A^{\cdot}\in |\mathcal{C}_n|$ we have 
$$\mathbb{E}_{\nu_i} \left(\# \mathrm{Surj}(\bullet, A^\cdot) \right) \to \mathbb{E}_{\mu_{n,u}} \left(\# \mathrm{Surj}(\bullet, A^\cdot) \right)$$ then $\nu_i$ weak-* converges to $\mu_{n,u}$ (i.e. converges pointwise on elements of $\mathcal{C}_n$). 
\end{quote}
We set $\nu_i$ to be a sequence as in the statement above. We set $$f_{A^\cdot}(X^\cdot)=\#\sur(X^\cdot,A^\cdot).$$

For a positive integer $e$, we set $\mathcal{C}_{n,e}$ to be the subset of $\mathcal{C}_n$ consisting of triples whose underlying
abelian group is torsion of order $\ell^e$.  Note that there is a natural pushforward functor 
$\Phi_e:\mathcal{C}_n\ra\mathcal{C}_{n,e}$ induced by the map $A\ra A\otimes\Z/\ell^e\Z$. Note that for $X^\cdot\in\mathcal{C}_n$
and $A^\cdot\in\mathcal{C}_{n,e}$ we have $\sur(X^\cdot,A^\cdot)\cong\sur(\Phi_e(X^\cdot),A^\cdot)$. In particular
$$\int_{\mathcal{C}_n} f_{A^\cdot}(\Phi_e(X^\cdot)) d\nu(X^\cdot) =\int_{\mathcal{C}_{n,e}} f_{A^\cdot}(Y^\cdot) d\Phi_e(\nu)(Y^\cdot).$$

We set $\nu^e_i:=\Phi_e(\nu_i)$ for all $e>0$.  First, we shall prove the following proposition:

\begin{prop}\label{tailbound}

For all $e>0,\epsilon>0,A^\cdot\in \mathcal{C}_{n,e}$ there exists an integer $c$ such that 
$$\int_{|Y|>c} f_{A^\cdot}(Y^\cdot) d\nu^e_i(Y^\cdot) <\epsilon$$ for all $i$.

\end{prop}
Let $Ab_e$ be the category of finite abelian groups of exponent dividing  $\ell^e$. There is a natural forgetful map
$F:\mathcal{C}_{n,e}\ra Ab_e$ which satisfies $$\#\sur(A^\cdot,B^\cdot)\leq \#\sur(A,B).$$ It is therefore sufficient to prove the following statement:

\begin{prop}\label{tailboundab}

For all $e>0,A\in Ab_e,\epsilon>0$ there exists an integer $c$ such that 
$$\int_{|X|>c} \#\sur(X,A) dF(\nu^e_i)(X) <\epsilon$$ for all $i$.

\end{prop}

\begin{proof}

Consider $A':=A\oplus \Z/\ell\Z$. For $c>\ell^{eM}$, any $X$ with $|X|>c$ satisfies that the rank of $X[\ell]$ is larger then $M$.  We claim that $$\#\sur(X,A')\geq \#\sur(X,A)\ell^{M-\rk A[\ell]}.$$ To see this, note that for any surjection $f:X\ra A$ we may pick a subgroup $Y\subset X$ which surjects onto $A$ with at most
$\rk A[\ell]$ generators. Now the number of liftings of $f$ to a surjection onto $A'$ is at least $\#\sur(X/Y,\Z/\ell\Z)$ which is at 
least of size $\ell^{M-\rk A[\ell]} -1$, and this can be made arbitrarily large. Thus for any $\epsilon'$ we may fine a 
 sufficiently large $c$ such that we have 
 \begin{align*}
 \int_{|X|>c} \#\sur(X,A) dF(\nu^e_i)(X)&\leq \epsilon'\int_{|X|>c} \#\sur(X,A') dF(\nu^e_i)(X)\\
 &\leq\epsilon'\int_{Ab_e} \#\sur(X,A') dF(\nu^e_i)(X)\\
 &\leq\epsilon'\int_{\mathcal{C}_n}\sum_{A'^\cdot\in F^{-1}(A')}f_{A'^\cdot}(X)d\nu_i(X)\\
 \end{align*}
 
 By assumption $\int_{\mathcal{C}_n}\sum_{A'^\cdot\in F^{-1}(A')}f_{A'^\cdot}(X)d\nu_i(X)$ is absolutely bounded (in a manner depending only on $A$ and the sequence $\nu_i$), so by taking $\epsilon'$ sufficiently
 small we obtain our desired result.

\end{proof}

Let $\nu$ be a measure in the weak-* closure of $\nu_i$.  We will show that $\nu= \mu_{n,m}$. By passing to a subsequence, we may assume that $\nu_i $ converge weak-* to $\nu$.

\begin{lemma}  For any $e>0$,  $ \Phi_e (\nu)$ is in the weak-* closure of $\nu_i^e $ \end{lemma}

\begin{proof} It suffices to prove, for all $G \in \mathcal {C}_{n,e}$, that \[ \lim_{i \ra \infty} \nu_i ( \Phi_e^{-1} ( G)) = \nu ( \Phi_e^{-1}(G)).\] For a natural number $r\geq n$,  let $\mathcal C_{n,>r}$ be the set of elements of $\mathcal C_{n,>r}$ whose underlying finite abelian group is not $\ell^r$-torsion.

We will first prove that $\lim\sup_{i \ra \infty} \nu_i ( \mathcal C_{n,>r} ) = \frac{1}{ \ell^r (\ell-1) }$. To do this, note that the number of elements of $\mathcal C_n$ whose underling abelian group is isomorphic to $\mathbb Z/\ell^{r+1}$ is at most $\ell^n$, because there are at most $\ell^n$ choices of $\psi$ and $\omega$ must vanish since $\wedge^2 \mathbb Z/\ell^{r+1} = 0$. Furthermore each element of $\mathcal C_{n,>r}$, because it is not $\ell^r$-torsion, has at least $\ell^r (\ell-1)$ surjective maps to $\mathbb Z/\ell^{r+1}$. So  

\begin{align*}
\ell^r (\ell -1) \nu_i ( \mathcal C_{n,>r} ) &\leq \int_{ \mathcal C_{n,>r} } \sur (X, \mathbb Z/ \ell^{r+1} ) \nu_i(X) \\
&\leq \int_{ \mathcal C_{n} } \sur (X, \mathbb Z/ \ell^r) \nu_i(X) \\ 
&= \sum_{ \substack {Y \in \mathcal C_n \\ F(Y) \cong \mathbb Z/\ell^{r+1} }}\int_{\mathcal C_n} \sur(X,Y) d\nu_i(X)
\end{align*} 
which is a sum of at most $\ell^{2}$ terms, each of which converges as $i$ goes to $\infty$ to \[ \int_{\mathcal C_n} \sur(X,Y) d \mu_{n,u}(X) = \frac{1}{ \sym^2 \mathbb Z/\ell^{r+1} [\ell^n] }= \frac{1}{\ell^n}.\]  Hence the sum converges as $i$ goes to $\infty $ to $1$, giving the statement.

Using this,

\begin{align*}  
\lim_{i \ra \infty} \nu_i ( \Phi_e^{-1} ( G)) &\leq  \lim\sup_{i \ra \infty} \nu_i ( \Phi_e^{-1} ( G) \cap \mathcal C_{n,>r} ) + \lim\sup_{i \ra \infty} \nu_i ( \Phi_e^{-1} ( G) \cap \mathcal C_{n,r} ) \\
&\leq \lim\sup_{i \ra \infty} \nu_i (  \mathcal C_{n,>r} ) + \lim\sup_{i \ra \infty} \nu_i ( \Phi_e^{-1} ( G) \cap \mathcal C_{n,r} ) \\
&\leq \frac{1}{ \ell^r (\ell-1)} +  \lim\sup_{i \ra \infty} \nu_i ( \Phi_e^{-1} ( G) \cap \mathcal C_{n,r} \\ 
&=\frac{1}{ \ell^r (\ell-1)} + \nu ( \Phi_e^{-1} ( G) \cap \mathcal C_{n,r} ) \leq \frac{1}{ \ell^r (\ell-1) }+ \nu ( \Phi_e^{-1} (G) ),
\end{align*}
with the key step because $\Phi_e^{-1} ( G) \cap \mathcal C_{n,r} $ is finite since groups in it have rank at most the rank of $G$ and torsion bounded by $\ell^r$. \end{proof}

Then using the proposition and the lemma, for all $A^\cdot\in\mathcal{C}_{n,e},\epsilon>0$ we can find $c$ such that we 
have

\begin{align*}
\int_{\mathcal{C}_n} f_{A^\cdot}(X^\cdot) d\nu(X^\cdot)&=\int_{\mathcal{C}_n} f_{A^\cdot}(\Phi_e(X^\cdot)) d\nu(X^\cdot)\\
&=\int_{\mathcal{C}_{n,e}} f_{A^\cdot}(Y^\cdot) d\Phi_e(\nu)(Y^\cdot)\\
&\geq\int_{|Y|<c} f_{A^\cdot}(Y^\cdot) d\Phi_e(\nu)(Y^\cdot)\\
&= \lim_{i\ra\infty}\int_{|Y|<c} f_{A^\cdot}(Y^\cdot) d\nu^e_i(Y^\cdot)\\
&\geq \lim_{i\ra\infty}\int_{\mathcal{C}_{n,e}} f_{A^\cdot}(Y^\cdot) d\nu^e_i(Y^\cdot) - \epsilon\\
&= \lim_{i\ra\infty}\int_{\mathcal{C}_{n,e}} f_{A^\cdot}(\Phi_e(X^\cdot) ) d\nu_i(X^\cdot) - \epsilon\\
&= \int_{\mathcal{C}_n} f_{A^\cdot}(X^\cdot) d\mu_{n,u}(X^\cdot) - \epsilon\\
\end{align*}

Taking $\epsilon$ to 0 we see $\int_{X} f_{A^\cdot}(X^\cdot) d\nu(X^\cdot)\geq\int_{X^\cdot} f_{A^\cdot}(X^\cdot) d\mu_{n,u}(X^\cdot)$ and thus they are equal by Fatou's Lemma. The Theorem then follows by Lemma \ref{momimpmeas}.

\end{proof}

\subsection{Relating $\mu$ to Malle's conjecture for class groups of number fields with $\ell$-power roots of unity}
We define a map $\psi:G^\vee[\ell^n]\rightarrow G[\ell^n]$ to be \emph{allowable} if $\langle \psi(\alpha),\beta\rangle = \langle \psi(\beta), \alpha\rangle $ whenever
$\ell^r\alpha=\ell^s\beta=0, r+s\leq n.\newline$

Let $A_{s,n}(G)$ be the number of maps $\psi:G^\vee[\ell^n]\rightarrow G[\ell^n]$ such that the corank of $\psi(G^\vee[\ell])$ in $G[\ell]$ is $s$.
\begin{lemma}
The universal measure assigns to a group $G$ the value
$$Q^t\mu(G) = \frac{\prod_{i=t+1}^{\infty}(1+\ell^{-i})^{-1}\ell^{R_n(G)}}{|\Aut(G)||\sym^2G[\ell^n]|\cdot|G|^{t}}\sum_{s=0}^t A_{s,n}(G)\frac{(\ell^{-1})_t}{(\ell^{-1})_{t-s}}$$
where $R_n(G)$ is defined by \[ R_n\left( \prod_{i=1}^r (\mathbb Z/\ell^{e_i} \mathbb Z) \right) = \sum_{\substack{ 1\leq i < j \leq r\\ \max(e_i,e_j) \leq n }} \min( e_i, e_j, n-\max(e_i,e_j)).\]

\end{lemma}

\begin{proof}

By the orbit stabilizer theorem, $$Q^t\mu(G):=\sum_{(\omega_G,\psi_G)} \frac{\Aut(G)}{\Aut(G,\omega_G,\psi_G)}Q^t\mu(G,\omega_G,\psi_G)$$ where the sum is over all pairs of $\omega_G,\psi_G$ yielding an $\ell^n$-BEG.
Applying Theorem \ref{quotient-measure-formula} yields

$$Q^t\mu(G)=\frac{\prod_{i=t+1}^{\infty}(1+\ell^{-i})^{-1}}{|\Aut(G)||\sym^2G[\ell^n]|\cdot|G|^{t}}\sum_{s=0}^{\min(r,t)}B{s,n}(G)\frac{(\ell^{-1})_t}{(\ell^{-1})_{t-s}}$$

where $B_{s,n}(G)$ counts the number of pairs of $\psi_G,\omega_G$ which yield an $\ell^n$-BEG and such that the corank of $\psi(G^\vee[\ell])$ in $G[\ell]$ is $s$.

Now, every allowable $\psi_G$ has at least one compatible $\omega_G$, and since the condition on $\omega_G$ given $\psi_G$ is an additive coset condition, the number of 
compatible $\omega_G$ is the same. It therefore remains to prove that the number of $\omega_G$ compatible with $\psi_G=0$ is $\ell^{R_n(G)}$.

To see this, write $G=\oplus\Z/\ell^{e_i}\Z\cdot b_i$ and write $\omega_G=\sum_{i,j}c_{i,j} b_i\wedge b_j$, where $c_{i,j}\in\Z/\ell^{\min(e_i,e_j)}\Z$. Wlog $e_i\leq e_j$. 
\begin{itemize}

\item \emph{Case 1: $e_i\geq n$.} Then 
$$0=\omega_{G,e_i}(\ell^{e_j-e_i}b_j^\vee,b_i^\vee)=c_{i,j}\mod{\ell^{e_i}}$$ which implies $c_{i,j}=0$.

\item \emph{Case 2: $e_j\geq n>e_i$.} Then
$$0=\omega_{G,n}(\ell^{e_j-n}b_j^\vee,b_i^\vee) = c_{i,j}\ell^{n-e_i}\mod\ell^n$$ which implies $c_{i,j}=0$.

\item \emph{Case 3: $e_j<n$}. Then

$$0=\omega_{G,n}(b_j^\vee,b_i^\vee) = c_{i,j}\ell^{2n-e_i-e_j}\mod\ell^n.$$ This implies  $\ell^{e_i+e_j-n}\mid c_{i,j}$. Since $c_{i,j}$ is only defined modulo $\ell^{e_i}$,
this gives $\ell^{\min(e_i,n-e_j)}$ possibilities for $c_{i,j}$.
\end{itemize}

Multiplying over all pairs $(i,j)$ gives the result.
\end{proof}

For the case of $n=1$, Malle\cite{Malle10} conjectured that $G$ should occur with probability

$$\frac{\prod_{i=t+1}^{\infty}(1+\ell^{-i})^{-1}}{|\Aut(G)||G|^t}\cdot \frac{\ell^{{r\choose 2}}(\ell^{-1})_{r+t}}{(\ell^{-1})_{t}}.$$

\begin{lemma}

The above two quantities agree. In other words, our conjecture agrees with Malle's.

\end{lemma}

\begin{proof}

Since for $n=1$ all $\psi$ are allowable, it is sufficient to show that 

$$\frac{\ell^{{r\choose 2}}(\ell^{-1})_{r+t}}{(\ell^{-1})_{t}} =\frac{1}{\ell^{{r+1\choose 2}}} \sum_{s=0}^{\min(r,t)}A_{s,n}(G)\frac{(\ell^{-1})_t}{(\ell^{-1})_{t-s}}$$

or slightly more elegantly

$$\frac{\ell^{r^2}(\ell^{-1})_{r+t}}{(\ell^{-1})_{t}} = \sum_{s=0}^{\min(r,t)}A_{s,1}(G)\frac{(\ell^{-1})_t}{(\ell^{-1})_{t-s}}$$

Note that since $A_{s,1}(G)$ depends only on $G[\ell]$, it is sufficient to handle the case where $G=\F_{\ell}^r$, which we henceforth assume. 

The number $A_{s,1}(G)$ of maps from $\F_{\ell}^r\rightarrow \F_{\ell}^r$ with image a subspace of dimension $r-s$ is the number of such subspaces, which is 
$\ell^{s(r-s)}{r\choose s}_{\ell^{-1}}$ multiplied by the number of surjections which is $\frac{(\ell^{-1})_r}{(\ell^{-1})_s}$. Using this and dividing through, 
we see that the above identity reduces to
$${r+t\choose r}_{\ell^{-1}}=\sum_{s=0}^{\min(r,t)}{t\choose s}_{\ell^{-1}}{r\choose s}_{\ell^{-1}}$$ which is a q-vandermonde identity.

\end{proof}

Garton \cite{Garton} gave a very nice formula in the case of $t=0,n=1$ for $\mu(G)$. We can also recover and generalize Garton's result as follows:

\begin{prop}\label{Gartonbign}
Let $G$ be a finite abelian $\ell$-group satisfying $G[\ell^n] = \oplus_{i=1}^n (\Z/\ell^i\Z)^{m_i}$. Then

$$\mu(G)=\frac{\prod_{i=1}^{\infty}(1+\ell^{-i})^{-1}}{|\Aut(G)|}\cdot |\wedge^2G[\ell^n]|\cdot (\ell^{-1};\ell^{-1})_{m_n}\cdot \prod_{j=1}^{n-1} (\ell^{-1};\ell^{-2})_{\lceil m_i/2\rceil}.$$
\end{prop}

It would be interesting to generalize this to the case of arbitrary $t$ in a simple closed form.

\begin{proof}

We must count the number $A_{0,n}(G)$ of allowable $\psi_G$ which are invertible. Note that each such can be decomposed into $\psi_G=\psi^++\psi^-$  where
$\psi^+$ is symmetric, and $\psi^-$ is antisymmetric. The number of symmetric $\psi^+$ is $\sym^2G[\ell^n]$, and this cancels out in the compatibility condition for $\omega_G$,
so all such $\psi^+$ can occur.

To understand the allowability condition on $\psi^-$, we proceed as above writing $G=\oplus\Z/\ell^{e_i}\Z\cdot b_i$. Set $r_j=\max(0,e_j-n)$. Then a basis for 
$G[\ell^n]$ is $\ell^{r_i} b_i$ and a basis for $G^\vee[\ell^n]$ is $\ell^{r_i}b_i^\vee$. 
Let $c_{i,j}\in\Z/\ell^{\min(e_i,e_j)}\Z$ be defined such that $\psi^-(\ell^{r_i}b_i)^\vee=\sum_j c_{i,j}\ell^{e_j-\min(e_i,e_j)} b_j$ so that 
$c_{i,j}=-\min c_{j,i}$ since $\psi^-$ is anti-symmetric.

Assume wlog $e_j\geq e_i$. The allowability condition on $\psi^-$ gives the following restrictions. We a

\begin{itemize}

\item \emph{Case 1: $e_j\geq n$.} Then the allowability condition is empty, giving $\min(e_i,e_j,n)$ possible values for $c_{i,j}$.

\item \emph{Case 2: $n>e_i+e_j.$} Then again,  the allowability condition is empty, giving $\min(e_i,e_j)$ possible values for $c_{i,j}$.

\item \emph{Case 3: $e_i+e_j\geq n>e_j$}. Then the allowability condition gives that 
$$0=\langle \ell^{e_i+e_j-n}b_j^\vee,b_i^\vee\rangle = c_{i,j}\ell^{e_i+e_j-n+(n-e_i)}=c_{i,j}\ell^{e_j}=0\mod\ell^n.$$
This implies  $\ell^{n-e_j}\mid c_{i,j}$. Since $c_{i,j}$ is only defined modulo $\ell^{e_i}$,
this gives $\ell^{e_i+e_j-n}$ possibilities for $c_{i,j}$.
\end{itemize}

We thus see that 
\begin{align*}
&\#\{\textrm{allowable }\psi^-\}\ell^{R_n(G)} \\
&=\prod_{e_j\geq n}\ell^{\min(e_i,e_j,n)}\prod_{n>e_i+e_j}\ell^{\min(e_i,e_j,n)}\prod_{e_i+e_j>n>\max(e_i,e_j)}\ell^{(e_i+e_j-n)+(n-\max(e_i,e_j))}\\
&=\prod_{i<j}\ell^{\min(e_i,e_j,n)}\\
&=\mid\wedge^2G[\ell^n]\mid\\
\end{align*}

It remains to restrict to those $\psi$ which are invertible, which is equivalent to $\psi\mod\ell$ being invertible. Note that in the natural basis, $\psi\mod\ell$ is a block diagonal 
matrix whose blocks are of size $m_1,m_2,\dots,m_n$, so 
it is necessary and sufficient for the blocks to all be invertible. Moreover, by the computation above ,for all but the last 
$m_n\times m_n$ block the blocks are all symmetric (since we are working modulo $\ell$, whereas for the $m_n\times m_n$ block there is no restriction. Since the probability 
that a random $r\times r$ symmetric matrix over $\F_\ell$ is invertible is $ (\ell^{-1};\ell^{-2})_{\lceil r/2\rceil}$, the claim follows.
\end{proof}

\section{Proof of Theorem \ref{ffmain}} \label{functionfieldproof}

The goal of this section is to prove Theorem \ref{ffmain}. To do so, we express the moments for our intermediate measures $\mu^q_g$ in terms of point counts for certain
moduli spaces. These spaces are very similar to the ones studied in \cite{EVW}, and we use their results on cohomological stability to obtain our main theorem.

Fix an element $(G,\omega_*,\psi_* )\in\cC_{\ell,n}$. We will estimate, for a random hyperelliptic curve of genus $g$ over $\mathbb F_q$, where $\mathbb F_q$ is a finite field containing the $\ell^n$th roots of unity, the average number of maps from the class group to $G$ such that the $\omega$ and $\psi$-invariants induced by the map are $\omega_*$ and $\psi_*$ respectively. The error term will be independent of $g$ for $g$ sufficiently large, but depend on $G$, and go to zero with $q$.

To that end, let $N$ be a natural number with $\lfloor \frac{N-1}{2} \rfloor = g$. Let $Y_N$ be the Hurwitz space parameterizing degree $2$ covers of $\mathbb P^1$ over $\mathbb F_q$ ramified at $N$ points in $\mathbb A^1$ and also ramified at $\infty$ if $N$ is odd. Then $Y_N $ parameterizes a family of smooth, complete, hyperelliptic  curves of genus $g$. Thus we have a map $a: \pi_1(Y_N) \to \mathrm{GSp}_{2g}^{\ell^n} (\mathbb Z_\ell)$, where $ \mathrm{GSp}_{2g}^{\ell^n} (\mathbb Z_\ell)$ is the subgroup of $\mathrm{GSp}_{2g}( \mathbb Z_\ell)$ whose similitude character is $1$ mod $\ell^n$. (Note that, in defining this map, we have fixed a symplectic form on $\mathbb Z_\ell^{2g}$.) Furthermore, there is a structure map $d: \pi_1(Y_N) \to \widehat{\mathbb Z}$.

Let $\omega_*^o = \frac{ 2 \ell^n}{q-1} \omega_*$. 

Fix a surjection $\phi: \mathbb Z_\ell^{2g} \to G$ such that the pushforward of the fixed symplectic form on $\mathbb Z_\ell^{2g}$ is $\omega_*^o$. Let $H \subseteq \mathrm{GSp}_{2g}^{\ell^n} (\mathbb Z_\ell)$ be the subgroup of matrices which fix $\phi$.

For $F \in H$, let $\psi_\phi(F) \in \Hom( G^\vee[\ell^n], G[\ell^n]) $ be the map induced by the image of $F$ in $ \mathrm{GSp}_{2g}^{\ell^n} (\mathbb Z_\ell)$ defined in subsection \ref{abomegapsi}

%
%

\begin{lemma} The function $\psi_\phi: H \to  \operatorname{Hom}( G^\vee[\ell^n]\ra G[\ell^n])$ is a homomorphism. \end{lemma}

\begin{proof}Fix $F_1,F_2\in H$.

  Let $\alpha$ be an element in $H^\vee [\ell^n]$. Let $\alpha_0$ be an element of $V$ that maps to $\phi^\vee (\alpha) \in A[\ell^n] \subset V/T$.  By \eqref{unravelpsi}, $\psi_\phi (F_1) ( \alpha)$ is $\phi ( (1- F_1) \alpha_0) $, and similarly for $F_2$ and $F_1 F_2$. 
  
  \[ (1- F_1F_2) \alpha_0 = (1-F_1) \alpha_0 + (1-F_2) \alpha_0 + (1-F_1) (1-F_2) \alpha_0 \] and since $F_2$ fixes $\phi$, \[ F_2 \phi^\vee(\alpha) =\phi^\vee(\alpha) \] so \[ (1-F_2) \phi^\vee (\alpha) =0 \] so \[ (1-F_2) \alpha_0 \in T\] and thus \[ (1-F_1)(1-F_2) \alpha_0 \in (1-F_1) T \] so \[\phi ( (1-F_1) (1-F_2) \alpha_0 )=0 .\]

%
%
\end{proof}

Let $Y_N^{ G, \omega_*^o, \psi_*}$ be the finite \'{e}tale covering space of $Y_N$ corresponding to the subgroup $H^*$ of $\pi_1(Y_N)$  consisting of elements $\sigma$ with $a(\sigma) \in H$ and $\psi_\phi (a(\sigma) ) = d (\sigma)\cdot \psi_*$.


\begin{lemma} The image of $H \cap Sp_{2g} (\mathbb Z_\ell)$ under $\psi_\phi$ is the set of homomorphisms whose induced pairing on $G^{\vee}[\ell^n]$ is symmetric, which has cardinality $|\sym^2(G) [ \ell^n]| $. 
\end{lemma} 

\begin{proof}

This is exactly what is proven in the end of Lemma \ref{nonlinearmoments}, where this map is referred to as $R$.

\end{proof}

\begin{lemma}\label{ffieldcompare} Let $(G,\omega_*,\psi_*)\in\cC_{\ell,n}$.

Then $\left| Y_N^{ G, \omega_*^o, \psi_*}( \mathbb F_q) \right| $ is equal to $|\sym^2(G) [ \ell^n]| $ times the sum over points in $Y_N(\mathbb F_q)$ of the number of surjections from the class group of the corresponding function field to $G$  with $\omega_C^o= \omega_*^o$ and $\psi_C = \psi_*$. \end{lemma}

\begin{proof} For an element $y \in Y_N(\mathbb F_q)$, with Frobenius element $F \in \pi_1(Y_N)$, the number of points in $Y_N^{ G, \omega_*^o, \psi_*}[ \mathbb F_q]$ lying over $y$ is equal to  \[ | \{ g\in \pi_1(Y_N)/ H^* |  g^{-1} F g \in H^* \} | \] \[  = | \{ g\in \pi_1(Y_N)/ H^* |  a(g)^{-1} a(F) a(g) \in H,  \psi_\phi ( a(g)^{-1} a(F) a(g) ) = \psi_*  \} | \] because $d (g^{-1} F g) = d(g) =1 $. 

Now $a(g)^{-1} a(F) a(g) \in H$ if and only if $\phi \circ a(g)^{-1} a(F) a(g) = \phi$, which occurs if and only if $(\phi \circ a(g)^{-1}) \circ a(F) = (\phi \circ a(g){-1})$. Similarly, we have $ \psi_\phi ( a(g)^{-1} a(F) a(g) )  = \psi_{\phi \circ a(g)^{-1}} ( a(F))$.  So this count is equal to \[ \sum_{ \substack{ \phi' : \mathbb Z_\ell^{2g} \to G \\  \phi' \circ a(F) = \phi' \\ \psi_{\phi'}(a(F)) = \psi_*\\  \omega_{\phi'} = \omega_*  }}| \{ g \in \pi_1(Y_N)/ H^* |  \phi \circ a(g)^{-1} = \phi' \} |  \] with the condition on $\omega_{\phi'}$ following from the existence of a $g$ with $\phi \circ a(g)^{-1} = \phi' $

Because $Sp_{2g}$ acts transitively on the surjections to $G$ with a given value of $\omega^o$, and the image of $\pi_1$ under $a$ contains $Sp_{2g}$, the cardinality $| \{ g \in \pi_1(Y_N)/ H^* |  \phi \circ a(g)^{-1} = \phi' \} |$ is independent of the choice of $\phi'$. So we may assume $\phi' = \phi$, in which case the first condition is equivalent to $a(g) \in H$, and so the number of possibilities is equal to the cardinality of the image of the homomorphism \[ \pi:\sigma \mapsto  \psi( a(\sigma)) - d(\sigma)\cdot\psi_*. \]

By the previous lemma, $\im\pi$ contains all the elements whose pairing is symmetric. On the other hand, $a(\sigma)$ necessarily satisfies the compatibility condition \ref{psiomegacomp} with respect to $\frac{ q^{ d(\sigma)}-1 }{2 \ell^n } \omega_*^o $. 

Because  $\frac{ q^{ d(\sigma)}-1 }{2 \ell^n }$ is congruent modulo $\ell^n$ to $d(\sigma) \frac{q-1}{ \ell^n } $, and $ \ell^n \omega_*^o =0$, we have
\[ \frac{ q^{ d(\sigma)}-1 }{2 \ell^n } \omega_*^o = d(\sigma) \frac{q-1}{\ell^n} \omega_*^o.\] 

It follows that $a(\sigma)$ satisfies the compatibility condition \ref{psiomegacomp} with respect to $d(\sigma) \frac{q-1}{\ell^n} \omega_*^o= d(\sigma) \omega_* $. Because also satisfies the compatibility condition \ref{psiomegacomp} with respect to $d(\sigma)\cdot \psi_*$, so it follows that $\psi( a(\sigma)) - d(\sigma)\cdot\psi_*$ defines a symmetric pairing on $G^\vee [\ell^n]$ for any $\sigma$, and so the image of $\pi$ consists exactly of those elements that are symmetric.

\end{proof}

\begin{lemma}\label{ffieldcover} Let $G'$ be any group with a surjection $\pi: G' \to G$ such that $\pi(G'[\ell^n]) = 0$ inside $G[\ell^n]$. Then there is a finite etale covering from some component of $\mathbf{H}_{G', n, \overline{\mathbb F}_q}$ to $Y_N^{G , \omega_*^o, \psi_*, \overline{\mathbb F}_q}$, where $\mathbf{H}_{G',n}$ is the Hurwitz space defined in \cite[\S7.1]{EVW}. \end{lemma}

\begin{proof} Both spaces are finite \'{e}tale coverings of $Y_N$. It is sufficient to find a component of $\mathbf{H}_{G',n}$ such that its geometric fundamental group, viewed as a subgroup of $\pi_1(Y_N)$, is contained in the geometric fundamental group of $Y_N^{G, \omega_*,\psi_*}$. Because the $d$ homomorphism is trivial on the geometric fundamental group, the geometric fundamental group of $Y_N^{G, \omega_*,\psi_*}$ is simply the subgroup of $\pi_1^{\mathrm{geom}}(Y_N)$ consisting of $\sigma$ with $a(\sigma) \in H$, $\psi(a(\sigma))=0$.

 We can choose a surjection $\phi': \mathbb Z_\ell^{2g} \to G'$ such that $ \pi \circ \pi'= \phi$. Then some component of $\mathbf{H}_{G',n}$ has fundamental group consisting of those $\sigma$ in $\pi_1^{\mathrm{geom}}(Y_N)$ such that $a(\sigma)$ fixes $\phi'$.  Clearly this implies that $a(\sigma)$ fixes $\phi$ and thus lies in $H$, so it remains to check that $\psi_\phi(\sigma) $ vanishes for these $\sigma$. This is because $\psi_\phi$ is compatible with surjections of groups, so $\psi_\phi(\sigma)$ is given by
 \[ G^\vee [\ell^n] \to G^{' \vee}[\ell^n] \to G' [\ell^n] \to G[\ell^n]  \] with the middle arrow $\psi_{\phi'}(\sigma)$, but by our construction of $G$ the map $ G' [\ell^n] \to G[\ell^n] $ vanishes, so indeed $\psi_{\phi}(\sigma)$ vanishes for all such $\sigma$.\end{proof} 

\begin{lemma}\label{ffieldcount} The number of connected components of $Y_N^{ G, \omega_*^o, \psi_*}$ defined over $\mathbb F_q$ is one. \end{lemma}

\begin{proof} This follows from the fact that $\pi_1^{\mathrm{geom}}(Y_N)$ acts transitively on $\pi_1^{\mathrm{geom}}(Y_N)/ H^*$.  \end{proof}

\begin{thm} For $q$ sufficiently large with respect to $|G|$, the number of pairs of a degree $2$ cover of $\mathbb P^1_{\mathbb F_q}$, ramified at a divisor of degree $N$ in $\mathbb A^1_{\mathbb F_q}$, plus $\infty$ if $N$ is odd, and a quotient $G$ of the $\ell$-class group with $\omega_C^o = \omega_*^o$ and $\psi_C=\psi_* $ is $\frac{q^N}{ |\sym^2(G[\ell^n])| } +{ O (q^{N-1/2})}$. \end{thm}

\begin{proof} By Lemma \ref{ffieldcompare}, this is the same as $\frac{1}{ \sym^2(G[\ell^n])}$ times the number of $\mathbb F_q$-points of the space $Y_N^{ G, \omega_*^o, \psi_*}$.  Because $Y_N^{ G, \omega_*^o, \psi_*}$ is a finite \'{e}tale cover of $Y_N$, it has dimension $N$.  Let $\tilde{\ell}$ be a prime other than the characteristic of $\mathbb F_q$

Thus by the Lefschetz fixed point formula, for a sufficiently large prime 

We have \[ \left| Y_N^{ G, \omega_*^o, \psi_*}(\mathbb F_q)\right|  = \sum_{i}=0^{2N} (-1)^i  \operatorname{tr} \left( \Frob_q | \; H^i_c \left( Y_N^{ G, \omega_*^o, \psi_*, \overline{\mathbb F_q}} , \mathbb Q_{\tilde{\ell}} \right) \right) .\]

Because $Y_N^{ G, \omega_*^o, \psi_*}$ is geometrically irreducible, $H^{2n}( Y_N^{ G, \omega_*^o, \psi_*, \overline{\mathbb F_q}} , \mathbb Q_{\tilde{\ell}})$ is one-dimensional, with Frobenius action multiplication by $q^N$. Thus 

\begin{align*} 
\left|  \left| Y_N^{ G, \omega_*^o, \psi_*}(\mathbb F_q)\right| - q^N \right| &\leq \sum_{i=0}^{2N-1} \left|  \operatorname{tr} \left(\Frob_q | \; H^i_c \left( Y_N^{ G, \omega_*^o, \psi_*, \overline{\mathbb F_q}} , \mathbb Q_{\tilde{\ell}} \right) \right) \right| \\ 
&\leq \sum_{i=0}^{2N-1} q^{i/2} \dim H^i_c \left( Y_N^{ G, \omega_*^o, \psi_*, \overline{\mathbb F_q}} , \mathbb Q_{\tilde{\ell}} \right) \\
&\leq \sum_{i=0}^{2N-1} q^{i/2} \dim H^i_c \left( Y_N^{ G'} , \mathbb Q_{\tilde{\ell}} \right)
\end{align*}
where $Y_N^{G'}$ is a $\mathbf{H}_{G',n}$, by Deligne's Riemann hypothesis and because $Y_N^{G'}$ is a finite \'{e}tale cover of $Y_N^{ G, \omega_*, \psi_*, \overline{\mathbb F_q}}$. Now by Po\'{i}ncare duality,

\begin{align*} 
\sum_{i=0}^{2N-1} q^{i/2} \dim H^i_c \left( Y_N^{ G'} , \mathbb Q_{\tilde{\ell}} \right) &= \sum_{i=1}^{2N} q^{N- i/2} \dim H^i \left( Y_N^{ G'} , \mathbb Q_{\tilde{\ell}} \right) \\ 
&\leq \sum_{i=1}^{2N} q^{N- i/2} C(G' \rtimes \mathbb Z/2, (0,1))^{i+1} \\
&\leq  q^N \frac{  C(G' \rtimes \mathbb Z/2, (0,1))^{2}}{ \sqrt{q} }  \frac{1}{1- \frac{ C(G' \rtimes \mathbb Z/2, (0,1))}{\sqrt{q}}} \\  
&= O(q^{N-1/2})
\end{align*}

by \cite[Proposition 7.8]{EVW}, as long as $\tilde{\ell}$ is sufficiently large (which we may freely assume). \end{proof}

\begin{proof}[Proof of Theorem \ref{ffmain}]: Lemma \ref{ffieldcompare} says that $\frac{\#Y_N^{ G, \omega_*^o, \psi_*}}{q^N}$ is exactly equal to the moment
$\E_{\mu^q_g}\#\sur(*,(G,\omega_*,\psi_*))$. The first part of Theorem \ref{ffmain} is then exactly the statement of \ref{ffieldcount}. 

For the second part, note that if $g,q\ra\infty$ then the moments of $\mu^q_g$ converge to the moments of $\mu$. The statement then follows from Theorem \ref{CLStability}.

\end{proof}

\section{Data} \label{data}

Below, we present computational data for the case of $K = \Q(\mu_3), \ell = 3$ so that $t = 1.$  Let $\zeta$ be a fixed primitive third root of unity in $K.$  We tabulated fields of the form $K(\sqrt{z})$ for squarefree $z = m \cdot 1 + n \cdot \alpha$ for $0 \leq m,n \leq 2000$ for $\alpha = - \zeta^2$.  This amounted to 3105738 fields.
We remark that it is sufficient to look only in this ``first sextant" range of $z$ for the following reasons: 

First, Since $\zeta,\zeta^2$ are both squares it is sufficient to restrict to those $z$'s with argument between $-\pi/3$ and $\pi/3$. Second, the field $K(\sqrt{z})$ is abstractly isomorphic to $K(\sqrt{\bar{z}})$, and this isomorphism acts on $\mu_3$ by inversion.  So, if the 3-part of the class group of $K(\sqrt{z})$ and its $\psi$-invariant equal $G$ and $\psi_G$ respectively, then the 3-part of the class group of $K(\sqrt{\bar{z}})$ and its $\psi$-invariant equal $G$ and $-\psi_G$ respectively.  Thus, including also those $z$'s with argument between $-\pi/3$ and $0$ would have the effect of making the totals for $(G,\psi_G)$ and $(G,-\psi_G)$ identical for reasons of symmetry.  

Since $t = 1,\psi_G$ determines $\omega_G$; $\omega_G$ is the ``antisymmetric part" of $\psi_G.$  For this reason, we only record data pertaining to $\psi_G$ in the table below.  

In the below table, 
\begin{itemize}
\item
The column ``observed proportion" records the proportion of fields \emph{with fixed class group isomorphism type} observed to have the value given in the column ``$\psi_G$." 

\item
The column ``expected proportion" records $\frac{Q^t \mu(G,\psi)}{Q^t \mu(G)}.$

\item
The homomorphism $\psi:\Hom( G, \mu_3) \ra G[3]$ is recorded is as follows: Suppose 
\begin{align*}
G &=\oplus_{i=1}^n \frac{\Z}{3^{n_i}\Z}\cdot e_i, \\
\Hom(G,\mu_3) &=\oplus_{i=1}^n\frac{\Z}{3\Z}\cdot f_i
\end{align*} 
so that $f_i(e_j)=\zeta^{\delta_{i,j}}$. Then we write down
the matrix of $\psi$ with resect to the bases $f_i,3^{n_i-1}e_i$. 

\item
In the ``$\psi_G$" column, we list a complete set of representatives for the isomorphism classes of $\psi_G,$ per the convention outlined in the previous bullet point, for $G \cong \ZZ / 3, \ZZ/9, \ZZ/27, \ZZ/81, \ZZ/243,$ and $\ZZ / 3 \oplus \ZZ / 3.$ There were 1258 instances of other groups
that occured but due to the large number of isomorphism classes of $\psi$ that occur in these cases we chose not to include this in the table below.

\end{itemize}

\begin{tabu}{| c | c | c |[2pt] c | c |}
\hline
class group $G$  & $\psi_G$ & total tabulated & observed proportion & conjectured proportion\\
\hline
trivial & () & 2698000 & 1.0 & 1.0\\
\tabucline[2pt]{-}  
$\ZZ / 3$ & (0) & 89565 & 0.2516 & 0.25 \\
\hline
$\ZZ / 3$ & (1) & 132764 & 0.3730 & 0.375 \\
\hline
$\ZZ / 3$ & (2) & 133622 & 0.3754 & 0.375 \\
\tabucline[2pt]{-}
$\ZZ / 9$ & (0) & 9186 & 0.2468 & 0.25 \\
\hline
$\ZZ / 9$ & (1) & 13866 & 0.3726 & 0.375 \\
\hline
$\ZZ / 9$ & (2) & 14161 & 0.3805 & 0.375 \\
\tabucline[2pt]{-}
$\ZZ / 27$ & (0) & 819 & 0.2495 & 0.25 \\
\hline
$\ZZ / 27$ & (1) & 1240 & 0.3778 & 0.375\\
\hline
$\ZZ / 27$ & (2) & 1223 & 0.3726 & 0.375\\
\tabucline[2pt]{-}
$\ZZ / 81$ & (0) & 31 & 0.2095 & 0.25 \\
\hline
$\ZZ / 81$ & (1) & 58 & 0.3919 & 0.375 \\
\hline
$\ZZ / 81$ & (2) & 59 & 0.3986 & 0.375 \\
\tabucline[2pt]{-}
$\ZZ / 243$ & (0) & 2 & 0.6667 & 0.25 \\
\hline
$\ZZ / 243$ & (1) & 0 & 0.0 & 0.375 \\
\hline
$\ZZ / 243$ & (2) & 1 & 0.3333 & 0.375 \\
\tabucline[2pt]{-}
$\ZZ / 3 \oplus \ZZ / 3$ & $\left(  \begin{array}{cc}  0 & 0 \\ 0 & 0 \end{array} \right)$ & 0 & 0.0 & 0.0 \\
\hline
$\ZZ / 3 \oplus \ZZ / 3$ & $\left(  \begin{array}{cc}  1 & 0 \\ 0 & 0 \end{array} \right)$ & 317 & 0.0321 & 0.0385\\
\hline
$\ZZ / 3 \oplus \ZZ / 3$ & $\left(  \begin{array}{cc}  2 & 0 \\ 0 & 0 \end{array} \right)$ & 315 & 0.0319 & 0.0385\\
\hline
$\ZZ / 3 \oplus \ZZ / 3$ & $\left(  \begin{array}{cc}  0 & 1 \\ 0 & 0 \end{array} \right)$ & 2008 & 0.2538 & 0.2308 \\
\hline
$\ZZ / 3 \oplus \ZZ / 3$ & $\left(  \begin{array}{cc}  0 & 1 \\ 1 & 0 \end{array} \right)$ & 1602 & 0.1621 & 0.1731 \\
\hline
$\ZZ / 3 \oplus \ZZ / 3$ & $\left(  \begin{array}{cc}  0 & 2 \\ 1 & 0 \end{array} \right)$ & 317 & 0.0321 & 0.0289 \\
\hline
$\ZZ / 3 \oplus \ZZ / 3$ & $\left(  \begin{array}{cc}  1 & 2 \\ 1 & 0 \end{array} \right)$ & 1188 & 0.1202 & 0.1154 \\
\hline
$\ZZ / 3 \oplus \ZZ / 3$ & $\left(  \begin{array}{cc}  2 & 2 \\ 1 & 0 \end{array} \right)$ & 1171 & 0.1185 & 0.1154 \\
\hline
$\ZZ / 3 \oplus \ZZ / 3$ & $\left(  \begin{array}{cc}  1 & 0 \\ 0 & 1 \end{array} \right)$ & 727 & 0.0736 & 0.0865 \\
\hline
$\ZZ / 3 \oplus \ZZ / 3$ &$\left(  \begin{array}{cc}  2 & 1 \\ 0 & 1 \end{array} \right)$ & 1737 & 0.1758 & 0.1731\\
\hline
\end{tabu}

\subsection{How computations were done} \label{computationmethod}

The above data was all tabulated in SAGE.  We used the Hilbert symbol method, described in \S \ref{psihilbertsymbol}, for computing the homomorphism $\psi_{L/K}$ for quadratic extensions $L / K = \mathbb{Q}(\mu_3).$  By the results of that section, $\psi$ is the composition:
\begin{equation} \label{psiascomposition}
\mathrm{Hom}(\mathrm{Cl}(K),\mu_3) \xrightarrow{\iota^{-1}} L^\times \cap O / (L^\times)^3 \xrightarrow{\text{inclusion}} L^\times \cap V / (L^\times)^3 \xrightarrow{\pi} \mathrm{Cl}(L)[3],
\end{equation}  

where 
\begin{itemize}
\item
$L^\times \cap O$ denotes the subset of $L^\times$ pairing trivially with $O_{L_v}^\times$ under the 3-Hilbert symbol at $v$ for all finite places $v$
\item
$L^\times \cap V$ denotes those elements of $L^\times$ having valuation divisible by $3$ for all finite places $v.$  
\end{itemize}

The isomorphism $\iota^{-1}$ is the inverse of the Hilbert symbol map described in \S \ref{psihilbertsymbol}:
\begin{align} \label{sumoflocalhilbertsymbols}
\iota: L^\times \cap O / (L^\times)^3 &\rightarrow \mathrm{Hom}(\mathrm{Cl}(L),\mu_3) \nonumber \\
b &\mapsto \left( f_b: x \in L^\times \backslash \mathbb{A}_{L,\mathrm{fin}}^\times / \prod O_{L_v}^\times  \mapsto \sum_{v \text{ finite}} \langle x,b \rangle_{3,v} \right),
\end{align}
where $\langle \rangle_{3,v}$ denotes the 3-Hilbert symbol at the place $v.$ 

The surjection $\pi$ arises from the bottom row of the commutative diagram
$$\begin{CD}
0 @>>> O_L^\times / (O_L^\times)^3 @>>> H^1(O_L, \mu_3) @>>> H^1(O_L,\mathbb{G}_m)[3] @>>> 0   \\ 
@. @V{=}VV @V{\sim}VV @VVV @.\\
0 @>>> O_L^\times / (O_L^\times)^3 @>{j}>> L^\times \cap V / (L^\times)^3 @>{\pi}>> H^1(O_L, \mathbb{G}_m)[3] = \mathrm{Cl}(L)[3]  @>>> 0
\end{CD}$$

where the top row is the Kummer exact sequence, the middle isomorphism is explicated in \S \ref{psihilbertsymbol}, $\pi(b) := \prod_v v^{\mathrm{val}_v(b)/3},$ and $j$ is the obvious inclusion map.

To compute the composition from \eqref{psiascomposition}, we work backwards as follows:

\begin{itemize}
\item[(a)]
Represent every element of $\mathrm{Cl}(L)[3]$ using a fractional ideal.  

\item[(b)]
For every representative ideal $I$ from (a), find $\alpha_I$ for which $I^3 = (\alpha_I).$  Also, compute representatives for the global unit group $O_L^\times / (O_L^\times)^3.$  

By the bottom row of the above Kummer sequence commutative diagram, every element of $L^\times \cap V / (L^\times)^3$ may be represented uniquely as $u \cdot \alpha_I$ as $I$ ranges through representative ideals and $u$ ranges over representatives for $O_L^\times / (O_L^\times)^3.$

\item[(c)]
$L^\times \cap O / (L^\times)^3$ consists of those $b \in L^\times / (L^\times)^3$ for which $L(b^{1/3}) / L$ is everywhere unramified. \footnote{Recall that $O_{L_v}^\times \subset L_v^\times$ corresponds to the inertia subgroup of $G_{L_v}$ by local class field theory.  Pairing trivially with $O_{L_v}^\times$ under the local $3$-Hilbert symbol for all places $v$ is thus equivalent to $L_v( (u \cdot \alpha)^{1/3} ) / L_v$ being unramified for all $v,$ i.e. $L((u\cdot \alpha)^{1/3}) / L$ being everywhere unramified.}   Identify such $b$ among the representatives $u \cdot \alpha_I$ found in (b).

\item[(d)]
For elements $b := u \cdot \alpha_I$ determined in (c), compute the homomorphism $f_b$ from \eqref{sumoflocalhilbertsymbols}.

\item[(e)]
Do linear algebra to represent each  `standard basis element' $f$ (depending on a choice of 3rd root of unity in $K$) as $f = f_b$ for $b = \prod_{i = 1}^k (u_i \cdot \alpha_{I_i})^{n_i}$ and integers $n_i.$

The desired homomorphism $\psi$ is uniquely determined by
\begin{align*}
\psi: \mathrm{Hom}(\mathrm{Cl}(L),\mu_3) &\rightarrow \mathrm{Cl}(L)[3] \\
f &\mapsto \text{ ideal class of } \prod_{i = 1}^k I_i^{n_i}.
\end{align*}

\end{itemize}

\bigskip

All infrastructure for manipulations with linear algebra, ideals, class groups, and unit groups was readily available through SAGE.  However, we were unable to find pre-existing code in SAGE to compute local Hilbert symbols (for $\ell \neq 2$) needed in step (d), so we coded this ourselves.  For (d), we used global methods:  
\begin{itemize}
\item
For $b \in L^\times \cap O / (L^\times)^3$ and $x \in L_v,$ the Hilbert symbol $\langle x,u \rangle_{3,v}$ equals $\langle x',b \rangle_{3,v}$ for any $x' \in L$ with $\mathrm{val}_v(x') = \mathrm{val}_v(x).$  For appropriately chosen $x',$ we used Hilbert reciprocity to express $\langle x',b \rangle_{3,v}$ as a corresponding product of local symbols ``at favorable places" which we managed to calculate directly.
\end{itemize}  

\bigskip

We refer the interested reader to our annotated code, available at 

\texttt{https://sites.google.com/site/michaellipnowski/}

for further details.

\end{document}